\documentclass[a4paper, pdftex]{amsart}

\usepackage{etex}
\usepackage[utf8]{inputenc}
\usepackage[T1]{fontenc}
\usepackage{graphicx}
\usepackage{tikz}
\usepackage{pgfplots}
\usepackage{mathtools, eqparbox}
\usepackage{etoolbox}
\usepackage{adjustbox}
\usepackage{multirow}
\usepackage{import}
\usepackage{mathabx}
\usepackage{nextpage}

\makeatletter
\providerobustcmd*{\bigcupdot}{%
  \mathop{%
    \mathpalette\bigop@dot\bigcup
  }%
}
\newrobustcmd*{\bigop@dot}[2]{%
  \setbox0=\hbox{$\m@th#1#2$}%
  \vbox{%
    \lineskiplimit=\maxdimen
    \lineskip=-0.7\dimexpr\ht0+\dp0\relax
    \ialign{%
      \hfil##\hfil\cr
      $\m@th\cdot$\cr
      \box0\cr
    }%
  }%
}
\makeatother

\newcommand{\mystackrel}[3][T]{\stackrel{\eqmakebox[#1]{\scriptsize#2}}{#3}}
\usetikzlibrary{patterns}
\usetikzlibrary{arrows,decorations.markings}
\tikzset{
  schraffiert/.style={pattern=horizontal lines,pattern color=#1},
  schraffiert/.default=black
}
\usepackage[many]{tcolorbox}

\usepackage{float}
\usepackage{amsmath}
\makeatletter
\newcommand\numbereq{%
  \ifmeasuring@\else
    \refstepcounter{equation}%
  \fi
  \tag{\theequation}%
}
\makeatother
\allowdisplaybreaks
\usepackage{ltablex}
\usepackage{booktabs}
\usepackage{amsfonts}
\usepackage{amsthm}
\usepackage{amssymb}
\usepackage{setspace}	
\usepackage[format=hang]{caption}
\usepackage{thmtools}
\usepackage{thm-restate}
\usepackage{mathtools}
\usepackage{verbatim}
\usepackage[colorlinks,citecolor=blue]{hyperref}
\usepackage[capitalize]{cleveref}
\usepackage[all, cmtip]{xy}
\usepackage{chngcntr}
\usetikzlibrary{cd}
\usepackage{titlesec}
\titleformat{\section}
  {\normalfont\normalsize\bfseries}{\thesection.}{1em}{}
\titleformat{\subsection}
  {\normalfont\normalsize\bfseries}{\thesubsection.}{1em}{}
\titleformat{\subsubsection}
  {\normalfont\normalsize\bfseries}{\thesubsubsection.}{1em}{}
\titleformat{\listfigurename}
  {\normalfont\normalsize\bfseries}{Figures}{1em}{}
\counterwithin{figure}{section}

\usepackage[shortlabels]{enumitem}

\DeclareMathOperator{\Aut}{Aut}
\DeclareMathOperator{\Conf}{Conf}
\DeclareMathOperator{\PConf}{PConf}
\DeclareMathOperator{\MapOp}{Map}
\DeclareMathOperator{\PMapOp}{PMap}

\DeclareMathOperator{\HomeoOp}{Homeo}
\DeclareMathOperator{\PHomeoOp}{PHomeo}
\DeclareMathOperator{\Z}{Z}

\DeclareMathOperator{\B}{B}
\DeclareMathOperator{\Sym}{S}
\DeclareMathOperator{\PZ}{PZ}
\DeclareMathOperator{\PB}{PB}

\DeclareMathOperator{\ev}{ev}
\DeclareMathOperator{\id}{\mathrm{id}}

\DeclareMathOperator{\piOrb}{\pi_1^{orb}}

\title{Braid groups and mapping class groups \\
for 2-orbifolds} 
\author{Jonas Flechsig}
\date{May 7, 2023}
\makeatletter    
\def\l@section{\@tocline{1}{0,2pt}{2pc}{8mm}{\ \ }} 
\makeatletter    
\def\l@subsection{\@tocline{1}{0,2pt}{2pc}{8mm}{\ \ }} 

\renewcommand{\maketitle}{
    \begin{center}

        \phantom{.}  

        {\LARGE \bf \@title\par}
        \vspace{0.5cm}

    \end{center}
}\makeatother
\begin{document}
\newtheorem*{theorem*}{Theorem}
\newtheorem{theorem}{Theorem}[section]
\newtheorem{corollary}[theorem]{Corollary}
\newtheorem{lemma}[theorem]{Lemma}
\newtheorem{fact}[theorem]{Fact}
\newtheorem*{fact*}{Fact}
\newtheorem{proposition}[theorem]{Proposition}
\newtheorem{thmletter}{Theorem}
\newtheorem{observation}[theorem]{Observation}
\newtheorem{notation}[theorem]{Notation}
\renewcommand*{\thethmletter}{\Alph{thmletter}}
\theoremstyle{definition}
\newtheorem{example}[theorem]{Example}
\newtheorem{question}[theorem]{Question}
\newtheorem{definition}[theorem]{Definition}
\newtheorem{construction}[theorem]{Construction}
\theoremstyle{remark}
\newtheorem{remark}[theorem]{Remark}
\newtheorem{conjecture}[theorem]{Conjecture}
\newtheorem{case}{Case}
\newtheorem{claim}{Claim}
\newtheorem*{claim*}{Claim}
\newtheorem{step}{Step}
\counterwithin{case}{theorem}
\renewcommand{\thecase}{\arabic{case}}
\counterwithin{claim}{theorem}
\renewcommand{\theclaim}{\arabic{claim}}
\counterwithin{step}{theorem}
\renewcommand{\thestep}{\arabic{step}}

\newenvironment{intermediate}[1][\unskip]{%
\vspace*{5pt}
\par
\noindent
\textit{#1.}}
{}
\vspace*{5pt}

\newcommand{\doubletable}[1]{\begin{tabular}[l]{@{}l@{}}#1\end{tabular}}
\newcommand{\set}[1][ ]{\ensuremath{ \lbrace #1 \rbrace}}
\newcommand{\bsl}{\ensuremath{\setminus}}
\newcommand{\grep}[2]{\ensuremath{\left\langle #1 | #2\right\rangle}}
\renewcommand{\ll}{\left\langle}
\newcommand{\rr}{\right\rangle}
\newcommand{\cpxbrmn}{B(\TheOrder,\TheOrder,\TheStrand)}
\newcommand{\cpxbr}[2]{B(#1,#1,#2)}
\newcommand{\Map}[2]{\MapOp_{#1}\left({#2}\right)}
\newcommand{\PMap}[2]{\PMapOp_{#1}\left({#2}\right)}
\newcommand{\MapOrb}[2]{\MapOp_{#1}^{orb}\left({#2}\right)}
\newcommand{\PMapOrb}[2]{\PMapOp_{#1}^{orb}\left({#2}\right)}
\newcommand{\MapId}[2]{\MapOp_{#1}^{\id}\left({#2}\right)}
\newcommand{\PMapId}[2]{\PMapOp_{#1}^{\id}\left({#2}\right)}
\newcommand{\MapIdOrb}[2]{\MapOp_{#1}^{\id,orb}\left({#2}\right)}
\newcommand{\PMapIdOrb}[2]{\PMapOp_{#1}^{\id,orb}\left({#2}\right)}
\newcommand{\HomeoId}[2]{\HomeoOp_{#1}^{\id}({#2})}
\newcommand{\PHomeoId}[2]{\PHomeoOp_{#1}^{\id}({#2})}
\newcommand{\Homeo}[2]{\HomeoOp_{#1}({#2})}
\newcommand{\PHomeo}[2]{\PHomeoOp_{#1}({#2})}
\newcommand{\HomeoIdOrb}[2]{\HomeoOp_{#1}^{\id,orb}({#2})}
\newcommand{\PHomeoIdOrb}[2]{\PHomeoOp_{#1}^{\id,orb}({#2})}
\newcommand{\HomeoOrb}[2]{\HomeoOp_{#1}^{orb}({#2})}
\newcommand{\PHomeoOrb}[2]{\PHomeoOp_{#1}^{orb}({#2})}
\newcommand{\PHomeoOrbt}[3]{\PHomeoOp_{#1}^{orb}({#2})}

\newcommand{\omicron}{o}
\newcommand{\cp}{c}
\newcommand{\pct}{p}
\newcommand{\TheCone}{N}
\newcommand{\ThePct}{L}
\newcommand{\Pc}{\theta}
\newcommand{\Pct}{\iota}
\newcommand{\NPct}{\lambda}
\newcommand{\TheOrder}{m}
\newcommand{\Ord}{t}
\newcommand{\TheStrand}{n}
\newcommand{\Order}{l}
\newcommand{\Strr}{h}
\newcommand{\Str}{i}
\newcommand{\Strand}{j}
\newcommand{\NStrand}{k}
\newcommand{\NNStrand}{l}
\newcommand{\pStrand}{p}
\newcommand{\qStrand}{q}
\newcommand{\sStrand}{s}
\newcommand{\tStrand}{t}
\newcommand{\Subdivision}{i}
\newcommand{\TheSubdivision}{p}
\newcommand{\Dim}{i}
\newcommand{\NDim}{j}
\newcommand{\TheDim}{k}
\newcommand{\Subdiv}{i}
\newcommand{\TheSubdiv}{p}
\newcommand{\NSubdiv}{j}
\newcommand{\TheNSubdiv}{q}
\newcommand{\NNSubdiv}{k}
\newcommand{\TheNNSubdiv}{r}
\newcommand{\NNNSubdiv}{l}
\newcommand{\TheFrac}{\frac{2\pi}{\TheOrder}}
\newcommand{\HalfFrac}{\frac{\pi}{\TheOrder}}
\newcommand{\neigh}{\lambda}
\newcommand{\htwC}{h_1^\twsC}
\newcommand{\htw}{h_{\TheStrand-1}^{\twsC'}}
\newcommand{\col}{blue}
\newcommand{\colo}{olive}
\newcommand{\short}{red}
\newcommand{\mult}{orange}
\newcommand{\red}{red}
\newcommand{\green}{olive}
\newcommand{\gre}{green}
\newcommand{\blue}{blue}
\newcommand{\SGpct}{\Sigma_\freeprod(\ThePct,\TheCone)}
\newcommand{\SG}{\Sigma_\freeprod(\ThePct)}
\newcommand{\Spct}{\Sigma(\ThePct,\TheCone)}
\renewcommand{\S}{\Sigma(\ThePct)}
\newcommand{\Sk}[1]{\Sigma(#1)}
\newcommand{\orbtwo}{\Sigma_{\freeprodtwo}}
\newcommand{\G}{G}
\newcommand{\g}{g}
\newcommand{\freeprod}{\Gamma}
\newcommand{\freeprodtwo}{\cycm\ast\cyc{\TheOrder'}}
\newcommand{\freeprodex}{\Gamma_{\TheOrder_\nu}^\TheCone}
\newcommand{\freeprodexk}[1]{\Gamma_{\TheOrder_\nu}^{#1}}
\newcommand{\freegrp}[1]{F_{#1}}
\newcommand{\free}[1]{F^{(#1)}}
\newcommand{\cycm}{\ZZ_\TheOrder}
\newcommand{\cyc}[1]{\ZZ_{#1}}
\newcommand{\kernel}{K}
\newcommand{\inter}[1]{{#1}^{\circ}}
\newcommand{\interext}[1]{{#1}^{\circ,\text{ext}}}

\newcommand{\Twist}{A}
\newcommand{\TwistP}{B}
\newcommand{\TwistC}{C}
\newcommand{\TwP}[1]{T_{#1}}
\newcommand{\TwC}[1]{U_{#1}}
\newcommand{\twist}{a}
\newcommand{\twistP}{b}
\newcommand{\twistC}{c}
\newcommand{\twP}[1]{t_{#1}}
\newcommand{\twC}[1]{u_{#1}}
\newcommand{\twsP}{t}
\newcommand{\twsC}{u}
\newcommand{\Twistn}[1]{X_{#1}}
\newcommand{\TwistnP}[1]{Y_{#1}}
\newcommand{\TwistnC}[1]{Z_{#1}}
\newcommand{\twistn}[1]{x_{#1}}
\newcommand{\twistnP}[1]{y_{#1}}
\newcommand{\twistnC}[1]{z_{#1}}
\newcommand{\twistnsC}{z}
\newcommand{\TwsP}{T}
\newcommand{\TwsC}{U}
\newcommand{\FD}{F}

\newcommand{\MA}{\ensuremath{\mathcal{MA}_\TheStrand}}
\newcommand{\pMA}{\ensuremath{\mathcal{MA}_\TheStrand(F_\freeprod)}}

\newcommand{\HA}{\ensuremath{\mathcal{HA}_\TheStrand}}

\newcommand{\MAo}{\ensuremath{\mathcal{MA}_{\TheStrand}(\Sigma_\freeprod)}}
\newcommand{\pMAo}{\ensuremath{\mathcal{MA}_{\TheStrand}(\Sigma_\freeprod^\ast)}}

\newcommand{\MAoZ}{\ensuremath{\mathcal{MA}_{\TheStrand}(D_{\ZZ_\TheOrder})}}
\newcommand{\MAoD}{\ensuremath{\mathcal{MA}_{\TheStrand}(\CC_{D_\TheOrder})}}
\newcommand{\tMAo}{\ensuremath{\tilde{\mathcal{MA}}_{\TheStrand}(\Sigma_\freeprod)}}

\newcommand{\HAo}{\ensuremath{\mathcal{HA}_{\TheStrand}(\Sigma_\freeprod)}}

\newcommand{\bpHAk}[1]{\ensuremath{\mathcal{HA}_{\TheStrand,\TheStrand',#1}}}
\newcommand{\bpHAo}{\ensuremath{\mathcal{HA}_{\TheStrand,\TheStrand'}(\Sigma_\freeprod)}}
\newcommand{\bpHAok}[1]{\ensuremath{\mathcal{HA}_{\TheStrand,\TheStrand',#1}(\Sigma_\freeprod)}}
\newcommand{\bpMAo}{\ensuremath{\mathcal{MA}_{\TheStrand,\TheStrand'}(\Sigma_\freeprod)}}
\newcommand{\bpMAok}[1]{\ensuremath{\mathcal{MA}_{\TheStrand,\TheStrand',#1}(\Sigma_\freeprod)}}
\newcommand{\bpMAoF}{\ensuremath{\mathcal{MA}_{\TheStrand,\TheStrand',\ThePct}^{F_\freeprod}(\Sigma_\freeprod)}}

\newcommand{\pbpMAo}{\ensuremath{\mathcal{MA}_{\TheStrand,\TheStrand'}(\Sigma_\freeprod^\ast)}}
\newcommand{\tbpMAo}{\ensuremath{\mathcal{MA}_{\TheStrand,\TheStrand'}(\Sigma_\freeprod)}}
\newcommand{\bpMAos}{\ensuremath{\mathcal{MA}_{\TheStrand,\TheStrand'}^{sim}(\Sigma_\freeprod)}}
\newcommand{\bpMA}{\ensuremath{\mathcal{MA}_{\TheStrand,\TheStrand'}}}
\newcommand{\bpMAk}[2]{\ensuremath{\mathcal{MA}_{\TheStrand,\TheStrand',{#1}}}(#2)}
\newcommand{\pbpMA}{\ensuremath{\mathcal{MA}_{\TheStrand,\TheStrand'}(F_\freeprod^\ast)}}
\newcommand{\tbpMA}{\ensuremath{\mathcal{MA}_{\TheStrand,\TheStrand'}(\Sigma_\freeprod)}}

\newcommand{\bpM}{\ensuremath{\mathcal{M}_{\TheStrand,\TheStrand',k}}}
\newcommand{\seg}{s}
\newcommand{\st}{\mathrm{st}}
\newcommand{\map}{\rho}
\newcommand{\Rep}{T}

\newcommand{\GL}[2][\TheRank]{\ensuremath{\operatorname{GL_{#1}}(#2)}}
\newcommand{\hmu}[2]{h_{#2}^{\tau_{#1}}}
\newcommand{\Stab}{\operatorname{Stab}}
\newcommand{\CC}{\mathbb{C}}
\newcommand{\RR}{\mathbb{R}}
\newcommand{\ZZ}{\mathbb{Z}}
\newcommand{\NN}{\mathbb{N}}
\newcommand{\PP}{\mathbb{P}}
\newcommand{\HH}{\mathbb{H}}
\newcommand{\SSS}{\mathbb{S}}
\newcommand{\iotaPMap}{\iota_{\PMap_\TheStrand}}
\newcommand{\piPMap}{\pi_{\PMap_\TheStrand}}
\newcommand{\iotaMapast}{\iota_{\PMap_\TheStrand^\ast}}
\newcommand{\piMapast}{\pi_{\PMap_\TheStrand^\ast}}
\newcommand{\iotaPZ}{\iota_{\PZ_\TheStrand}}
\newcommand{\piPZ}{\pi_{\PZ_\TheStrand}}
\newcommand{\sPZ}{\mathrm{s}_{\PZ_\TheStrand}}
\newcommand{\iotaPZast}{\iota_{\PZ_\TheStrand^\ast}}
\newcommand{\piPZast}{\pi_{\PZ_\TheStrand^\ast}}
\newcommand{\iotaPZpct}{\iota_{\PZ_\TheStrand^\ast}}
\newcommand{\piPZpct}{\pi_{\PZ_\TheStrand^\ast}}
\newcommand{\Push}{\textup{Push}}
\newcommand{\Forget}{\textup{Forget}}
\newcommand{\PushPMap}{\textup{Push}_{\textup{PMap}_\TheStrand}}
\newcommand{\ForgetPMap}{\textup{Forget}_{\textup{PMap}_\TheStrand}}
\newcommand{\piMap}{\pi_{\textup{Map}}}
\newcommand{\new}[1]{\color{black} #1 \color{black}}
\newcommand{\enew}[1]{\color{black} #1 \color{black}}

\author{J. Flechsig}
\address{Jonas Flechsig: Fakult\"at f\"ur Mathematik, Universit\"at Bielefeld, D-33501 Bielefeld, Germany}

\maketitle
\begin{center}
Jonas Flechsig
\\[5pt]
May 7, 2023
\\[10pt]
\textbf{Abstract} 
\end{center}

\enew{The main result of this article is that pure orbifold braid groups fit into an exact sequence 
\[
1\rightarrow\kernel\rightarrow\pi_1^{orb}(\Sigma_\freeprod(\TheStrand-1+\ThePct))\xrightarrow{\iotaPZ}\PZ_\TheStrand(\Sigma_\freeprod(\ThePct))\xrightarrow{\piPZ}\PZ_{\TheStrand-1}(\Sigma_\freeprod(\ThePct))\rightarrow1. 
\]
In particular, we observe that the kernel $\kernel$ of $\iotaPZ$ is non-trivial. This corrects Theorem 2.14 in \cite{Roushon2021}. Moreover, we use the presentation of the pure orbifold mapping class group $\PMapIdOrb{\TheStrand}{\Sigma_\freeprod(\ThePct)}$ from \cite{Flechsig2023mcg} to determine $\kernel$. Comparing these orbifold mapping class groups with the orbifold braid groups, reveals a surprising behavior: in contrast to the classical case, the orbifold braid group is a proper quotient of the orbifold mapping class group. This yields a presentation of the pure orbifold braid group which allows us to read off the kernel $\kernel$.} 

\section{Introduction}

Orbifold braid groups are analogs of Artin braid groups or, more generally, surface braid groups. Instead of considering braids moving inside a disk or a surface, orbifold braids move inside a $2$-dimensional orbifold. Orbifold braid groups have attracted interest since some of them contain spherical and affine Artin groups of type $D_\TheStrand,\tilde{B}_\TheStrand$ and $\tilde{D}_\TheStrand$ as finite index subgroups by work of Allcock \cite{Allcock2002}. \new{For these Artin groups,} the orbifold braid groups provide us with braid pictures. Roushon published several articles on the structure of orbifold braid groups \cite{Roushon2021,Roushon2022b,Roushon2023,Roushon2022a} and the contained Artin groups \cite{Roushon2021a}. \new{Further,} Crisp--Paris~\cite{CrispParis2005} studied the outer automorphism group of the orbifold braid group. The present article also contributes to the study of the structure of orbifold braid groups. 

As in the work of Roushon \cite{Roushon2021}, we consider braid groups on orbifolds with finitely many punctures and cone points (of possibly different orders). The underlying orbifolds are defined using the following data: Let $\freeprod$ be a free product of finitely many finite cyclic groups. \new{As such,} $\freeprod$ acts on a planar, contractible surface $\Sigma$ (with boundary), obtained by thickening the Bass--Serre tree (see Example \ref{ex:good_orb_free_prod} for details). If we add $\ThePct$ punctures, we obtain a similar orbifold as studied by Roushon, which we denote by $\Sigma_\freeprod(\ThePct)$. \new{In contrast to his article,} we consider orbifolds with non-empty boundary but this does not affect the structure of the orbifold braid groups. The only singular points in the orbifold $\Sigma_\freeprod(\ThePct)$ are cone points that correspond to the finite cyclic factors of the free product $\freeprod$. 

The elements of orbifold braid groups $\Z_\TheStrand(\Sigma_\freeprod(\ThePct))$ are represented by braid diagrams (see Figure \ref{fig:braid_diagram_intro} for an example) with $\TheStrand$ strands (drawn in black), $\TheCone$ cone point bars (drawn in red with a cone at the top) and $\ThePct$ bars that represent the punctures (drawn in blue with a diamond at the top). The composition of these diagrams works as in Artin braid groups. 

\begin{figure}[H]
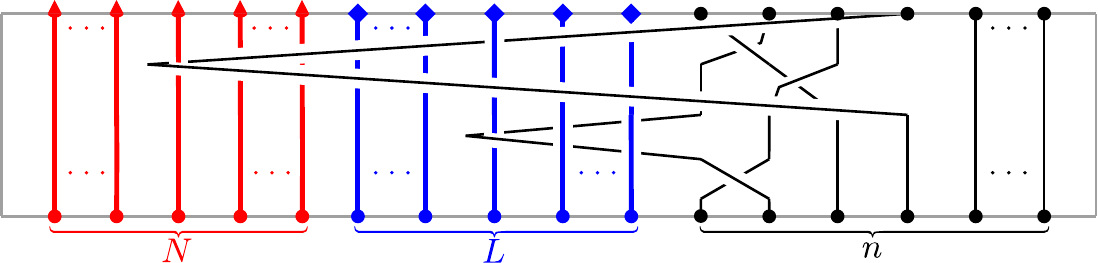
\caption{Example of an orbifold braid in $\Z_\TheStrand(\Sigma_\freeprod(\ThePct))$.}
\label{fig:braid_diagram_intro}
\end{figure}

\enew{Orbifold braid diagrams further allow us to observe that $\Z_\TheStrand(\Sigma_\freeprod(\ThePct))$ is generated by the braids 
\[
h_1,...,h_{\TheStrand-1},\twP{1},...,\twP{\ThePct},\twC{1},...,\twC{\TheCone}
\]
as pictured in Figure \ref{fig:generators_intro}. This recovers a result that is \new{implicitly contained} in \cite[Theorems 1.1 and 1.2]{Allcock2002} and~\cite[Section 4]{Roushon2021}. \new{Later,} we also deduce a finite presentation of $\Z_\TheStrand(\Sigma_\freeprod(\ThePct))$ in terms of these generators.}

\begin{figure}[H]
\centerline{
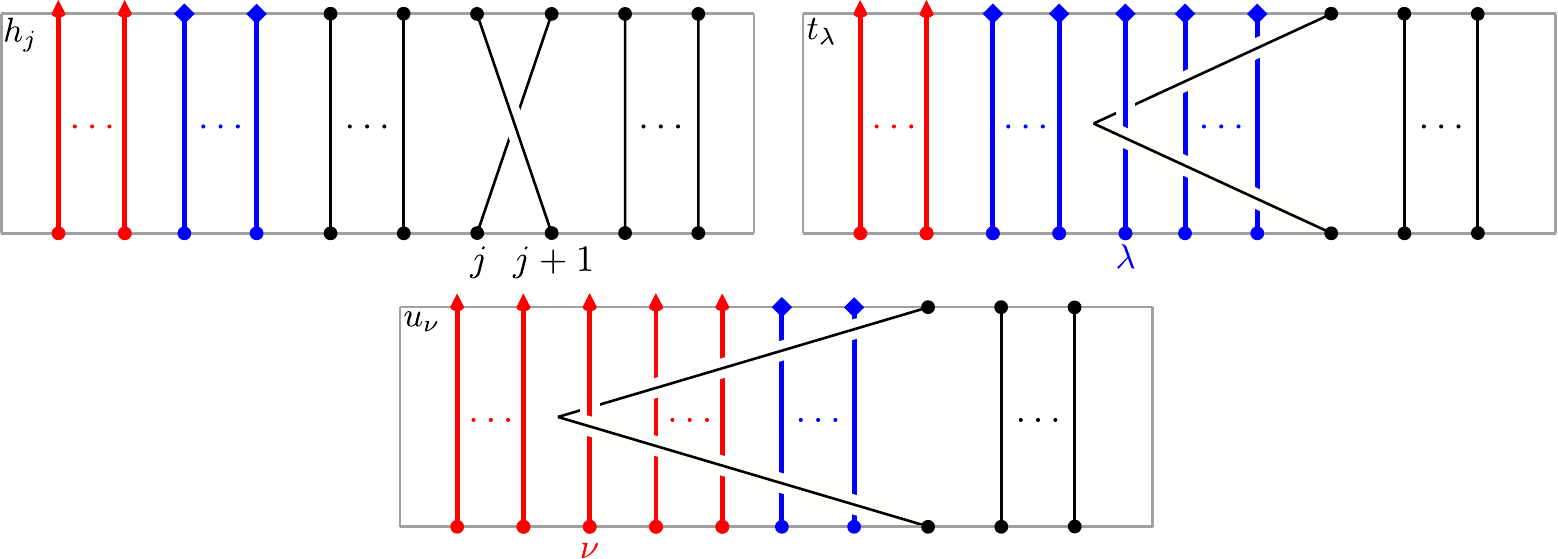
}
\caption{Generators of $\Z_\TheStrand(\Sigma_\freeprod(\ThePct))$.}
\label{fig:generators_intro}
\end{figure}

The key difference to the Artin braid groups is that a braid as in Figure \ref{fig:fin_order_el} that encircles a single cone point bar but no other strands or punctures represents an element of finite order (namely the order of the cyclic factor $\cyc{\TheOrder_\nu}$ associated to the encircled cone point), see Remark \ref{rem:braid_fin_order} for further details. 

\begin{figure}[H]
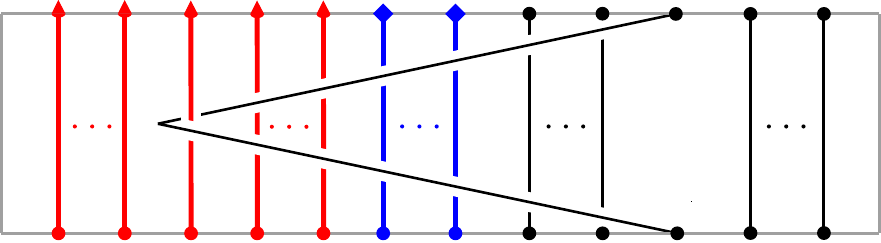
\caption{A finite order element in $\Z_\TheStrand(\Sigma_\freeprod(\ThePct))$.}
\label{fig:fin_order_el}
\end{figure}

\new{As in the case of the Artin braids,} each orbifold braid induces a permutation of the $\TheStrand$ strands. This yields an epimorphism $\Z_\TheStrand(\Sigma_\freeprod(\ThePct))\rightarrow\Sym_\TheStrand$ whose kernel is called the pure orbifold braid group, denoted by $\PZ_\TheStrand(\Sigma_\freeprod(\ThePct))$. As for the group $\Z_\TheStrand(\Sigma_\freeprod(\ThePct))$, the diagrams allow us to establish a finite generating set of the pure subgroup. This recovers a result of Roushon \cite[Lemma 4.1]{Roushon2021}. 

\enew{The orbifold braid group $\Z_\TheStrand(\Sigma_\freeprod(\ThePct))$ is deeply connected to the orbifold mapping class group of the punctured orbifold $\Sigma_\freeprod(\ThePct)$ with $\TheStrand$ marked points. This group, denoted by $\MapOrb{\TheStrand}{\Sigma_\freeprod(\ThePct)}$, is studied in \cite{Flechsig2023mcg}. A mapping class of $\Sigma_\freeprod(\ThePct)$ is represented by a $\freeprod$-equivariant homeomorphism of $\Sigma(\ThePct)$ that fixes cone points and the boundary~$\partial\Sigma(\ThePct)$ pointwise. Such a homeomorphism respects the $\TheStrand$ marked points if it preserves the $\freeprod$-orbit of the $\TheStrand$ marked points as a set. The equivalence relation is induced by $\freeprod$-equivariant ambient isotopies fixing cone points, marked points and the boundary. 

These orbifold mapping class groups admit a homomorphism \[
\Forget_\TheStrand^{orb}:\MapOrb{\TheStrand}{\Sigma_\freeprod(\ThePct)}\rightarrow\MapOrb{}{\Sigma_\freeprod(\ThePct)} 
\]
by forgetting the marked points. Let $\MapIdOrb{\TheStrand}{\Sigma_\freeprod(\ThePct)}$ be the kernel of $\Forget_\TheStrand^{orb}$.} 

\subsection*{Main results}

\enew{Concerning the relation between $\Z_\TheStrand(\Sigma_\freeprod(\ThePct))$ and $\MapIdOrb{\TheStrand}{\Sigma_\freeprod(\ThePct)}$, we observe the following: Evaluating a certain ambient isotopy at the marked points $p_1,...,p_\TheStrand$ yields a homomorphism 
\[
\ev:\MapIdOrb{\TheStrand}{\Sigma_\freeprod(\ThePct)}\rightarrow\Z_\TheStrand(\Sigma_\freeprod(\ThePct)),  
\]
see Section \ref{subsec:orb_braid_quot_orb_mcg} for details. In contrast to the classical situation, this evaluation map is not an isomorphism. However, the kernel of $\ev$ can be described in terms of $\frac{2\pi}{\TheOrder_\nu}$-twists $\TwC{\nu}$ and $\TwistC_{\NStrand\nu}$ of a marked point around the cone point $\cp_\nu$ for $1\leq\nu\leq\TheCone$. Locally, $\TwC{\nu}$ and $\TwistC_{\NStrand\nu}$ twist around the cone point as described in Figure \ref{fig:homeo_twist_marked_pts_cp_clockwise}. For further information about the embedding of the twisted disks, we refer to Section~\ref{sec:basics_mcg_orb}. 

\begin{figure}[H]
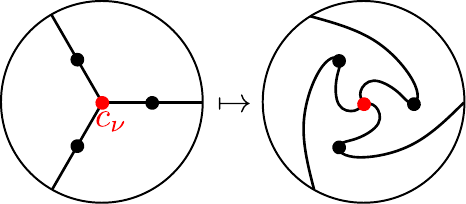
\caption{A $\frac{2\pi}{\TheOrder_\nu}$-twist for $\TheOrder_\nu=3$.}
\label{fig:homeo_twist_marked_pts_cp_clockwise}
\end{figure}

\begin{thmletter}
\label{thm-intro:kernel_ev}
The kernel of $\ev$ is the normal closure of $\{\TwC{\nu}^{\TheOrder_\nu}\mid1\leq\nu\leq\TheCone\}$ in $\MapIdOrb{\TheStrand}{\Sigma_\freeprod(\ThePct)}$. The kernel of the restricted map $\ev\vert_{\PMapIdOrb{\TheStrand}{\Sigma_\freeprod(\ThePct)}}$ is the normal closure of $\{\TwistC_{\NStrand\nu}^{\TheOrder_\nu}\mid1\leq\nu\leq \TheCone, 1\leq\NStrand\leq\TheStrand\}$ in $\PMapIdOrb{\TheStrand}{\Sigma_\freeprod(\ThePct)}$. 
\end{thmletter}

By \cite[Proposition 4.22]{Flechsig2023mcg}, we have a presentation of $\MapIdOrb{\TheStrand}{\Sigma_\freeprod(\ThePct)}$. \new{Together with Theorem \ref{thm-intro:kernel_ev},} this allows us to deduce the following presentation of $\Z_\TheStrand(\Sigma_\freeprod(\ThePct))$ in terms of the braids from Figure \ref{fig:generators_intro}: 

\begin{thmletter}
\label{thm-intro:pres_orb_braid_free_prod}
The orbifold braid group $\Z_\TheStrand(\Sigma_\freeprod(\ThePct))$ is presented by generators 
\[
h_1,...,h_{\TheStrand-1},\twP{1},...,\twP{\ThePct},\twC{1},...,\twC{\TheCone}
\]
and the following defining relations for $2\leq\Strand<\TheStrand$, $1\leq\Pc,\NPct\leq\ThePct$ with $\Pc<\NPct$ and $1\leq\mu,\nu\leq\TheCone$ with $\mu<\nu$: 
\begin{enumerate}
\item \label{thm-intro:pres_orb_braid_free_prod_rel1} 
$\twC{\nu}^{\TheOrder_\nu}=1$, 
\item \label{thm-intro:pres_orb_braid_free_prod_rel2} 
braid and commutator relations for the generators $h_1,...,h_{\TheStrand-1}$, 
\item \label{thm-intro:pres_orb_braid_free_prod_rel3} 
$[\twP{\NPct},h_\Strand]=1$ \; and \; 
$[\twC{\nu},h_\Strand]=1$, 
\item \label{thm-intro:pres_orb_braid_free_prod_rel4} 
$[h_1\twP{\NPct}h_1,\twP{\NPct}]=1$ \; and \; 
$[h_1\twC{\nu}h_1,\twC{\nu}]=1$, 
\item \label{thm-intro:pres_orb_braid_free_prod_rel5}
$[\twP{\Pc},\twistP_{2\NPct}]=1$, $[\twC{\mu},\twistC_{2\nu}]=1$ \; and \; $[\twP{\NPct},\twistC_{2\nu}]=1$
\\
with \; $\twistP_{2\NPct}=h_1^{-1}\twP{\NPct}h_1$ \; and \; $\twistC_{2\nu}=h_1^{-1}\twC{\nu} h_1$. 
\end{enumerate}
\end{thmletter}

All the relations from Theorem \ref{thm-intro:pres_orb_braid_free_prod} follow from geometric observations: The finite order relation \ref{thm-intro:pres_orb_braid_free_prod}\eqref{thm-intro:pres_orb_braid_free_prod_rel1} follows as described in Remark \ref{rem:braid_fin_order}, the relations \ref{thm-intro:pres_orb_braid_free_prod}\eqref{thm-intro:pres_orb_braid_free_prod_rel2} are well-known and the commutator relations \ref{thm-intro:pres_orb_braid_free_prod}\eqref{thm-intro:pres_orb_braid_free_prod_rel3}-\ref{thm-intro:pres_orb_braid_free_prod}\eqref{thm-intro:pres_orb_braid_free_prod_rel5} follow from the braid pictures in Figure \ref{fig:gens_satisfy_rels_free_prod_comm_rel}. Moreover, a similar presentation can also be deduced for the pure orbifold braid group $\PZ_\TheStrand(\Sigma_\freeprod(\ThePct))$ (see Corollary \ref{cor:pres_pure_free_prod}).}

The main contribution of this article is a result on the structure of pure orbifold braid groups: Let $\iotaPZ$ be the map that considers elements of the orbifold fundamental group of the punctured orbifold $\Sigma_\freeprod(\TheStrand-1+\ThePct)$ as braids in $\PZ_\TheStrand(\Sigma_\freeprod(\ThePct))$ that only move the $\TheStrand$-th strand. \new{Moreover,} let $\piPZ:\PZ_\TheStrand(\Sigma_\freeprod(\ThePct))\rightarrow\PZ_{\TheStrand-1}(\Sigma_\freeprod(\ThePct))$ be the map that forgets the $\TheStrand$-th strand. Using Theorems \ref{thm-intro:kernel_ev} and \ref{thm-intro:pres_orb_braid_free_prod}, we obtain: 

\begin{thmletter}
\label{thm-intro:kernel_ex_seq}
The pure orbifold braid group $\PZ_\TheStrand(\Sigma_\freeprod(\ThePct))$ fits into the exact sequence 
\[
1\rightarrow\kernel\rightarrow\pi_1^{orb}\left(\Sigma_\freeprod(\TheStrand-1+\ThePct)\right)\xrightarrow{\iotaPZ}\PZ_\TheStrand(\Sigma_\freeprod(\ThePct))\xrightarrow{\piPZ}\PZ_{\TheStrand-1}(\Sigma_\freeprod(\ThePct))\rightarrow1 
\]
where $\kernel$ is the normal closure of 
\[
\textup{PC}\left(\{(\twistn{\Strand}\twistnC{\nu})^{\TheOrder_\nu}(\twistn{\Strand}^{-1}\twistnC{\nu}^{-1})^{\TheOrder_\nu}\mid 1\leq\Strand<\TheStrand,1\leq\nu\leq\TheCone\}\right)
\]
in $\pi_1^{orb}\left(\Sigma_\freeprod(\TheStrand-1+\ThePct)\right)$ where $\textup{PC}(S)$ is the set of partial conjugates of elements in~$S$ (see Section \textup{\ref{subsec:ex_seq_pure_orb_braid}}, Steps \textup{\ref{it:det_ker_1}} and \textup{\ref{it:det_ker_2}} for further details). 
\end{thmletter}

\enew{This corrects Theorem 2.14 in \cite{Roushon2021} which claims that the kernel $\kernel$ is trivial. Moreover, Theorem \ref{thm-intro:kernel_ex_seq} implies that contrary to \cite[Proposition 4.1]{Roushon2021a} the natural homomorphism 
\[
\omega:\Z_\TheStrand(\Sigma_\freeprod(\ThePct))\rightarrow\Z_{\TheStrand+\ThePct}(\Sigma_\freeprod)
\]
which maps punctures to fixed strands is not injective for $\ThePct\geq1$ (see Proposition~\ref{prop:emb_orb_braid_grps_not_inj} for details).} 

\subsection*{Overview}

We introduce the group $\Z_\TheStrand(\Sigma_\freeprod(\ThePct))$ as the orbifold fundamental group of an orbifold configuration space. The relevant concepts are summarized in Section~\ref{sec:intro_orb}. Since the $\freeprod$-action on $\Sigma$ has a fundamental domain, we may reinterpret the elements in this group in terms of braid diagrams as in Figures \ref{fig:braid_diagram_intro}, \ref{fig:generators_intro} and \ref{fig:fin_order_el}. This description is the subject of Section \ref{sec:braid_grp_orb}. \new{In Section \ref{sec:basics_mcg_orb}, we give a brief overview about orbifold mapping class groups which summarizes the results from \cite{Flechsig2023mcg}. 
Section \ref{sec:braid_and_mcg} is the main achievement of this article: There we deduce Theorem \ref{thm-intro:kernel_ev} which yields a presentation for $\Z_\TheStrand(\Sigma_\freeprod(\ThePct))$ and $\PZ_\TheStrand(\Sigma_\freeprod(\ThePct))$.} The latter presentation allows us to compute the non-trivial kernel~$\kernel$ (Theorem \ref{thm-intro:kernel_ex_seq}). 

\section*{Acknowledgments} 

\new{I would like to thank my adviser Kai-Uwe Bux for his support and many helpful discussions. Many thanks are also due to José Pedro Quintanilha and Xiaolei Wu for their helpful advice at different points of this project. Moreover, I am grateful to José Pedro Quintanilha for his comments on a draft of this text.}

The author was funded by the German
Research Foundation (Deutsche Forschungsgemeinschaft, DFG) – 426561549. Further, the author was partially supported by Bielefelder Nachwuchsfonds. 

\newpage

\section{Orbifolds and their fundamental groups}
\label{sec:intro_orb}

\setlist[enumerate,1]{label=\textup{(\arabic*)}, ref=\thetheorem(\arabic*)}
\setlist[enumerate,2]{label=\textup{\alph*)}, ref=\alph*)}

In this article we only consider orbifolds that are given as the quotient of a manifold (typically a surface) by a proper group action. Recall that an action 
\[
\phi:\G\rightarrow\Homeo{}{M},\g\mapsto\phi_\g 
\]
on a manifold $M$ is \textit{proper} if for each compact set $K\subseteq M$ the set 
\[
\{\g\in\G\mid\phi_\g(K)\cap K\neq\emptyset\} 
\]
is compact. Since we endow $\G$ with the discrete topology, i.e.\ the above set is finite. Orbifolds that appear as proper quotients of manifolds are called \textit{developable} in the terminology of Bridson--Haefliger \cite{BridsonHaefliger2011} and \textit{good} in the terminology of Thurston \cite{Thurston1979}. 

\new{Above and in the following, all manifolds are orientable and all homeomorphisms are orientation preserving. }

\begin{definition}[{Orbifolds}, {\cite[Chapter III.G, 1.3]{BridsonHaefliger2011}}]
\label{def:good_orb}
Let $M$ be a 
manifold, possibly with boundary, 
and $\G$ a group with a monomorphism 
\[
\phi:\G\rightarrow\Homeo{}{M} 
\]
such that $\G$ acts properly on $M$. Under these conditions the $3$-tuple $(M,\G,\varphi)$ is called an \textit{orbifold}, which we denote by $M_\G$. If $\Stab_\G(\cp)\neq\{1\}$ for a point $\cp\in M$, the point $\cp$ is called a \textit{singular point} of $M_\G$. If $\Stab_\G(\cp)$ for a point $\cp\in M$ is cyclic of finite order $\TheOrder$, the point $\cp$ is called a \textit{cone point} of $M_\G$ of order~$\TheOrder$. 
\end{definition}

A first example of an orbifold is the following: 

\begin{example}
\label{ex:good_orb_D_cyc_m}
Let $\cycm$ be a cyclic group of order $\TheOrder$. The group $\cycm$ acts on a disk~$D$ by rotations around its center. The action is via isometries and the acting group is finite, i.e.\ the action is proper. Consequently, $D_{\cycm}$ is an orbifold with exactly one singular point in the center of $D$, which is a cone point. 
\end{example}

Example \ref{ex:good_orb_D_cyc_m} motivates a more general construction for a free product of finitely many finite cyclic groups which we describe briefly in the following. For further details, we refer to the author's PhD thesis \cite[Section 2.1]{Flechsig2023}. We will consider this generalization of the orbifold $D_{\cycm}$ throughout the article. 

\begin{example}
\label{ex:good_orb_free_prod}
Let $\freeprod$ be a free product of finite cyclic groups $\cyc{\TheOrder_1},...,\cyc{\TheOrder_\TheCone}$. The group $\freeprod$ is the fundamental group of the following graph of groups with trivial edge groups  
\vspace*{3mm}
\begin{center}
\begingroup%
  \makeatletter%
  \providecommand\color[2][]{%
    \errmessage{(Inkscape) Color is used for the text in Inkscape, but the package 'color.sty' is not loaded}%
    \renewcommand\color[2][]{}%
  }%
  \providecommand\transparent[1]{%
    \errmessage{(Inkscape) Transparency is used (non-zero) for the text in Inkscape, but the package 'transparent.sty' is not loaded}%
    \renewcommand\transparent[1]{}%
  }%
  \providecommand\rotatebox[2]{#2}%
  \newcommand*\fsize{\dimexpr\f@size pt\relax}%
  \newcommand*\lineheight[1]{\fontsize{\fsize}{#1\fsize}\selectfont}%
  \ifx\svgwidth\undefined%
    \setlength{\unitlength}{127.35202975bp}%
    \ifx\svgscale\undefined%
      \relax%
    \else%
      \setlength{\unitlength}{\unitlength * \real{\svgscale}}%
    \fi%
  \else%
    \setlength{\unitlength}{\svgwidth}%
  \fi%
  \global\let\svgwidth\undefined%
  \global\let\svgscale\undefined%
  \makeatother%
  \begin{picture}(1,0.10366749)%
    \lineheight{1}%
    \setlength\tabcolsep{0pt}%
    \put(0,0){\includegraphics[width=\unitlength,page=1]{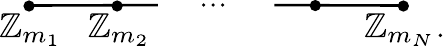}}%
  \end{picture}%
\endgroup%

\end{center}
\new{As such,} $\freeprod$ acts on its Bass--Serre tree $T$. The fundamental domain of this action is a path with $N-1$ edges. 
The action is free on edges and the vertex stabilizers are conjugates $\gamma\cyc{\TheOrder_\nu}\gamma^{-1}$ with $\gamma\in\freeprod$ and $1\leq\nu\leq\TheCone$. By the choice of a generator $\gamma_\nu$ for each $\cyc{\TheOrder_\nu}$ with $1\leq\nu\leq\TheCone$, the link of each vertex carries a cyclic ordering. 

Let us consider a proper embedding of the Bass--Serre tree~$T$ into $\CC$ that respects the local cyclic order on each link. If we choose a regular neighborhood of $T$ inside~$\CC$, we obtain a planar, contractible surface $\Sigma$ (with boundary), see Figure \textup{\ref{fig:constr_fund_domain}} for an example.

This surface $\Sigma$ inherits a proper $\freeprod$-action from the Bass--Serre tree such that vertex stabilizers act with respect to the cyclic order on the link of the stabilized vertex. Moreover, the action admits a fundamental domain corresponding to the fundamental domain in $T$. 
In particular, we obtain an orbifold structure $\Sigma_\freeprod$. 

\begin{figure}[H]
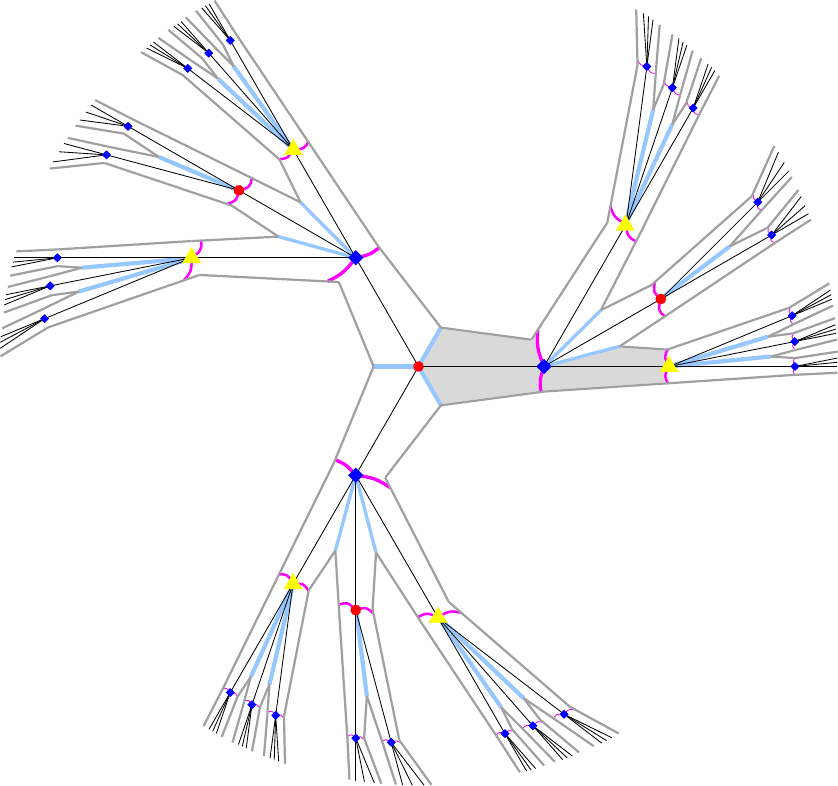
\caption{Thickened Bass--Serre tree for $\freeprod=\cyc{3}\ast\cyc{2}\ast\cyc{4}$ with fundamental domain shaded in gray. The red $\textcolor{red}{\bullet}$, blue~$\textcolor{blue}{\blacklozenge}$ and yellow $\textcolor{yellow}{\blacktriangle}$ are conjugates of the free factors $\cyc{3}$, $\cyc{2}$ and $\cyc{4}$, respectively.}
\label{fig:constr_fund_domain}
\end{figure}

A point in $\Sigma_\freeprod$ is a singular point if and only if it corresponds to a vertex of $T$. Hence, the singular points in $\Sigma_\freeprod$ are all cone points and decompose into $\TheCone$ orbits. The quotient $\Sigma/\freeprod$ is a disk with $\TheCone$ distinguished points that correspond to the orbits of the cone points. 

In general, we may choose a fundamental domain $\FD$ that is a disk as pictured in Figure \textup{\ref{fig:fund_domain}} and contains exactly $\TheCone$ cone points $\cp_1,...,\cp_\TheCone$ that lie on the boundary \new{such that each has exactly two adjacent boundary arcs that lie in the same $\freeprod$-orbit.} 
\begin{figure}[H]
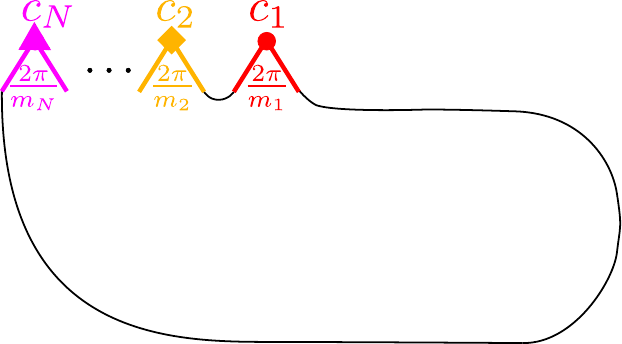
\caption{The fundamental domain $\FD$.}
\label{fig:fund_domain}
\end{figure}

If we remove the boundary of $\Sigma$, the quotient $\Sigma^\circ/\freeprod$ is homeomorphic to the complex plane with $\TheCone$ distinguished points and associated cyclic groups $\cyc{\TheOrder_\nu}$ for $1\leq\nu\leq\TheCone$. Adding $\freeprod$-orbits of punctures $\freeprod(r_\NPct)$ for $1\leq\NPct\leq\ThePct$ to $\Sigma$ such that $\freeprod(r_\Pc)\neq\freeprod(r_\NPct)$ for $1\leq\Pc,\NPct\leq\ThePct, \Pc\neq\NPct$, we obtain the orbifold called 
\[
\CC(\ThePct,\TheCone,\textbf{\TheOrder}) \text{ with } \textbf{\TheOrder}=(\TheOrder_1,...,\TheOrder_\TheCone)
\]
in \cite{Roushon2021}. In \cite{Allcock2002}, Allcock studied braids on these orbifolds for 
\[
(\ThePct,\TheCone,\textbf{\TheOrder})=(0,2,(2,2)), (0,1,(2)) \text{ and } (1,1,(2)). 
\]
\end{example}

Since we also want to study mapping class groups, which requires to fix the boundary, we will consider the orbifold $\Sigma_\freeprod(\ThePct)$ with boundary. \new{Moreover, we use the notation $\Sigma_\freeprod(\ThePct)$ for the orbifold with underlying surface (with boundary) 
\begin{equation}
\label{eq:Sigma_minus_pct}
\Sigma(\ThePct):=\Sigma\setminus\freeprod(\{r_1,...,r_\ThePct\}). 
\end{equation}}


We will consider orbifold fundamental groups using a concept of paths introduced in \cite[Chapter III.G, 3]{BridsonHaefliger2011}. There an orbifold is considered more generally as an \textit{\'etale groupoid} $(\mathcal{G},X)$, see \cite[Chapter III.G, 2]{BridsonHaefliger2011}. If $M_\G$ is an orbifold in the sense of Definition \ref{def:good_orb}, the associated \'etale groupoid is given by 
\[
(\mathcal{G},X)=(\G\times M,M). 
\]

\new{In the following,} we will simplify the notation using $\G$ instead of $\mathcal{G}=\G\times M$. In particular, we introduce $\G$-paths. These are the $\mathcal{G}$-paths in \cite{BridsonHaefliger2011}.  

\begin{definition}[{$\G$-path}, {\cite[Chapter III.G, 3.1]{BridsonHaefliger2011}}]
\label{def:G-path}
A \textit{$\G$-path} $\xi=(\g_0,c_1,\g_1,...,c_\TheSubdivision,\g_\TheSubdivision)$ in $M_\G$ with initial point $x\in M$ and terminal point $y\in M$ over a subdivision $a=t_0\leq...\leq t_\TheSubdivision=b$ of the interval $[a,b]$ consists of
\begin{enumerate}
\item continuous maps $c_\Subdivision:[t_{\Subdivision-1},t_\Subdivision]\rightarrow M$ for $1\leq\Subdivision\leq\TheSubdivision$ and
\item group elements $\g_\Subdivision\in\G$ such that $\g_0(c_1(t_0))=x$, $\g_\Subdivision(c_{\Subdivision+1}(t_\Subdivision))=c_\Subdivision(t_\Subdivision)$ for $1\leq\Subdivision<\TheSubdivision$ and $\g_\TheSubdivision(y)=c_\TheSubdivision(t_\TheSubdivision)$ (see Figure \ref{fig:G-path}). 
\end{enumerate}
We call $\xi$ a \textit{$\G$-loop based at $x$}, if the initial point $x\in M$ is also the terminal point. If an element $\g_\Subdiv$ is non-trivial, we say that $\xi$ contains a \textit{$\G$-leap} at time $t_\Subdiv$. 

For brevity, we write $(\g_0,c_1,\g_1,...,c_\TheSubdiv)$ for $(\g_0,c_1,\g_1,...,c_\TheSubdiv,\g_\TheSubdiv)$ if $\g_\TheSubdiv=1$. We say a $\G$-path is \textit{continuous} if it is of the form $(\g,c)$. 
\begin{figure}[H]
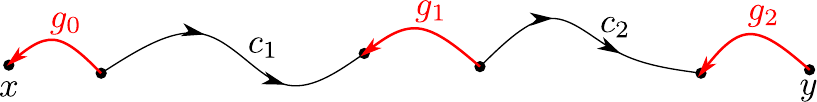
\caption{A $\G$-path.}
\label{fig:G-path}
\end{figure}
\end{definition} 

The following equivalence relation 
identifies certain $\G$-paths whose continuous pieces have the same $\G$-orbits. 

\begin{definition}[{Equivalence of $\G$-paths}, {\cite[Chapter III.G, 3.2]{BridsonHaefliger2011}}]
\label{def:equiv_paths}
Let 
\[
\xi=(\g_0,c_1,\g_1,...,c_\TheSubdivision,\g_\TheSubdivision) 
\]
be a $\G$-path over 
$a=t_0\leq...\leq t_\TheSubdivision=b$. 
\begin{enumerate}
\item
\label{def:equiv_paths_subdiv} 
A \textit{subdivision} of $\xi$ is a $\G$-path obtained from $\xi$ by choosing $t'\in[t_{\Subdivision-1},t_\Subdivision]$ for some $1\leq\Subdiv\leq\TheSubdiv$  and replacing the entry $c_\Subdiv$ with the sequence 
\[
(c_\Subdivision\vert_{[t_{\Subdivision-1},t']},1,c_\Subdivision\vert_{[t',t_\Subdivision]}). 
\] 
\item
\label{def:equiv_paths_shift} 
A \textit{shift} of $\xi$ is a $\G$-path obtained from $\xi$ by choosing $h\in\G$ and replacing a subsequence $(\g_{\Subdiv-1},c_\Subdiv,\g_\Subdiv)$ for some $1\leq\Subdiv\leq\TheSubdiv$ with 
\[
(\g_{\Subdiv-1}h^{-1},h\cdot c_\Subdiv,h\g_\Subdiv). 
\] 
\end{enumerate}
We say that two $\G$-paths are \textit{equivalent} if one can be obtained from the other by a sequence of subdivisions, inverses of subdivisions and shifts. 
\end{definition}
\begin{figure}[H]
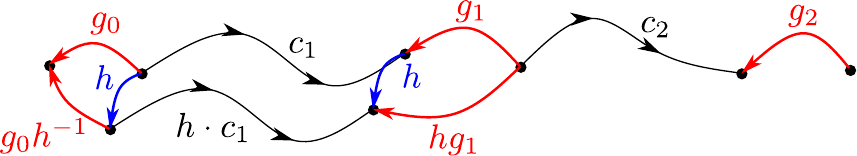
\caption{Two $\G$-paths equivalent by a shift.}
\label{fig:equivalent_G-paths}
\end{figure}

Using this equivalence relation, we mimic the homotopy relation for paths in topological spaces for $\G$-paths. 

\begin{definition}[{Homotopy of $\G$-paths}, {\cite[Chapter III.G, 3.5]{BridsonHaefliger2011}}]
\label{def:homo_G-paths}
An \textit{elementary homotopy} between two $\G$-paths $\xi$ and $\tilde{\xi}$ is a family of $\G$-paths $\xi_s=(\g_0,c_1^s,...,\g_\TheSubdivision)$ over the subdivisions $0=t_0\leq t_1\leq...\leq t_\TheSubdivision=1$. The family $\xi_s$ is parametrized by $s\in[s_0,s_1]$ such that $c_\Subdivision^s$ depends continuously on the parameter and $\xi^{s_0}=\xi$, $\xi^{s_1}=\tilde{\xi}$. 

\begin{figure}[H]
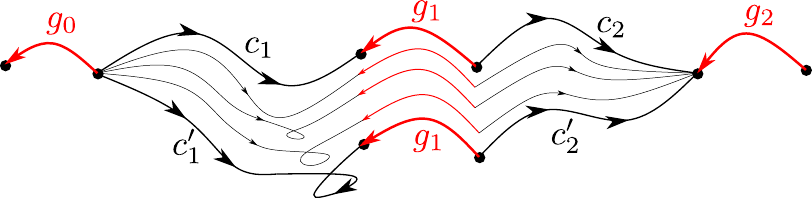
\caption{An elementary homotopy of $\G$-paths.}
\label{fig:el_homo_G-paths}
\end{figure}

Two $\G$-paths are \textit{homotopic (relative to their endpoints)} if one can pass from the first to the second by a sequence of the following operations: 
\begin{enumerate}
\item equivalence of $\G$-paths, 
\item elementary homotopies. 
\end{enumerate}
\end{definition}

\begin{definition}[{Orbifold fundamental group}, {\cite[Chapter III.G, 3.6]{BridsonHaefliger2011}}]
\label{def:fundam_grp_orb}
Let $x_0$ be a non-singular point in $M_\G$. On the set of homotopy classes of $\G$-loops based at $x_0$ one easily defines a composition, see \cite[Chapter III.G, 3.4]{BridsonHaefliger2011} for details. With this composition the set of homotopy classes of $\G$-loops has a group structure. This group is called the \textit{orbifold fundamental group} $\piOrb(M_\G,x_0)$ of $M_\G$. The neutral element is represented by the constant $\G$-loop based at $x_0$. 
\end{definition}

Throughout the article we restrict to orbifolds $M_\G$ such that each two points are connected by a $\G$-path, i.e.\ $M_\G$ is \textit{$\G$-path-connected}. As in the case of the fundamental group of a path-connected topological space, the choice of base point does not affect the fundamental group of a $\G$-path-connected orbifold (up to isomorphism), see \cite[Chapter III.G, Proposition 3.7]{BridsonHaefliger2011}. \new{Hence,} we shorten our notation $\piOrb(M_\G,x_0)$ to $\piOrb(M_\G)$ whenever the base point does not matter. 


We finish the section with two observations that relate orbifold fundamental groups to fundamental groups of 
topological spaces. In particular, we determine a presentation of $\piOrb(\Sigma_\freeprod)$. This sets the foundation to identify a semidirect product structure 
of pure orbifold braid groups on $\Sigma_\freeprod$. 

For a fixed index $1\leq\NSubdiv\leq\TheSubdiv$, the shift defined in Definition \ref{def:equiv_paths} allows for the following choice: \new{The element $h$ shifts $c_\NSubdiv$ to $c_\NSubdiv':=h\cdot c_\NSubdiv$, $\g_{\NSubdiv-1}'=\g_{\NSubdiv-1}h^{-1}$ and $\g_\NSubdiv'=h\g_\NSubdiv$. Thus, the path $\xi$ is equivalent to} 
\begin{equation}
\label{eq:restricted_shift}
\xi'=(\g_0,c_1,\g_1,...c_{\NSubdiv-1},\g_{\NSubdiv-1}',c_\NSubdiv',\g_\NSubdiv',c_{\NSubdiv+1},...,c_\TheSubdivision,\g_\TheSubdivision). 
\end{equation}
If we choose $h=\g_\NSubdiv^{-1}$, the element $\g_\NSubdiv'$ is trivial. \new{Replacing $c_\NSubdiv',1,c_{\NSubdiv+1}$ by $c_\NSubdiv'\cup c_{\NSubdiv+1}$,} we obtain the path 
\begin{equation}
\label{eq:inverted_subdiv}
\tilde{\xi}'=(\g_0,c_1,\g_1,...c_{\NSubdiv-1},\g_{\NSubdiv-1}',c_\NSubdiv'\cup c_{\NSubdiv+1},\g_{\NSubdiv+1},c_{\NSubdiv+1},\g_{\NSubdiv+2},...,c_\TheSubdivision,\g_\TheSubdivision)
\end{equation}
which is equivalent to $\xi$ and has shorter subdivision length. 

\new{For proofs of the following Lemma \ref{lem:cont_path_homo} and Corollary \ref{cor:fund_grp_orb_ses}, we refer to Lemma~2.12 and Corollary~2.13 in \cite{Flechsig2023}.} 

\begin{lemma}[{\cite[Chapter III.G, 3.9(1)]{BridsonHaefliger2011}}]
\label{lem:cont_path_homo}
\leavevmode
\begin{enumerate}
\item \label{lem:cont_path_homo_it1}
Every $\G$-path connecting $x$ to $y$ in $M_\G$ is equivalent to a unique continuous $\G$-path $(\g,c)$ with $c:I\rightarrow M$ connecting $\g^{-1}(x)$ to $y$. 
\item \label{lem:cont_path_homo_it2}
Let $(\g,c)$ and $(\g',c')$ be two $\G$-loops based at a non-singular point $x_0$ in $M_\G$. Then these $\G$-loops represent the same element of $\piOrb(M_\G,x_0)$ if and only if $\g=\g'$ and $c$ is homotopic to $c'$. 
\end{enumerate}
\end{lemma}

\begin{corollary}[{\cite[Chapter III.G, 3.9(1)]{BridsonHaefliger2011}}]
\label{cor:fund_grp_orb_ses}
Let $M_0$ be the path-component of a point $x_0\in M$. Then $\G_0=\{\g\in\G\mid \g^{-1}(x_0)\in M_0\}$ is a subgroup of $\G$ and every $\G$-loop at $x_0$ is equivalent to a unique $\G$-loop of the form $(\g,c)$ where $c$ is a path connecting $\g^{-1}(x_0)$ and $x_0$; 
therefore $\g\in\G_0$. Hence, we have a short exact sequence
\[
1\rightarrow\pi_1(M_0)\stackrel{i}\rightarrow\piOrb(M_\G)\stackrel{p}\rightarrow\G_0\rightarrow1. 
\]
In particular, the orbifold fundamental group of $\Sigma_\freeprod$ from Example \textup{\ref{ex:good_orb_free_prod}} is isomorphic to $\freeprod$. 
\end{corollary}

If the $\G$-action on $M$ is free, the space $M/\G$ also admits the structure of a manifold. The following is well known: 

\begin{lemma}
\label{lem:fund_grp_free_act}
Let $M$ be a manifold with a proper, free $\G$-action. 
If the quotient space $M/\G$ is path-connected, then $\piOrb(M_\G)\cong\pi_1(M/\G)$. 
\end{lemma}

\new{For instance, a proof is presented in \cite[Lemma 2.14]{Flechsig2023}.} 

\section{Orbifold braid groups}
\label{sec:braid_grp_orb}

\renewcommand{\Twist}[2]{A_{{#1}{#2}}}
\renewcommand{\TwistP}[2]{B_{{#1}{#2}}}
\renewcommand{\TwistC}[2]{C_{{#1}{#2}}}
\renewcommand{\twist}[2]{a_{{#1}{#2}}}
\renewcommand{\twistP}[2]{b_{{#1}{#2}}}
\renewcommand{\twistC}[2]{c_{{#1}{#2}}}

In this section we introduce orbifold braid groups. For the orbifolds $\SG$, we explain how elements in these groups are encoded as orbifold braid diagrams. Similar braid diagrams were considered by Allcock and Roushon for the orbifolds $\CC(\ThePct,\TheCone,\mathbf{\TheOrder})=\inter{\Sigma}_{\freeprod}(\ThePct)$ with $\mathbf{\TheOrder}=(\TheOrder_1,...,\TheOrder_\TheCone)\in\NN_\geq2^\TheCone$, see \cite{Allcock2002} and \cite{Roushon2021}.  

\subsection{Artin braid groups}
\label{subsec:Artin_braids}

Before we get into the details of the definition of orbifold braid groups, we recall the geometry 
of Artin braids. In particular, we discuss how these three-dimensional braids are encoded as two-dimensional Artin braid diagrams. 
\new{For additional information on Artin braid groups, we refer to \cite[Section 1]{KasselTuraev2008}}. 

\begin{definition}[{Geometric braids and Artin braid group}, {\cite[Section 1.2.1]{KasselTuraev2008}}]
\label{def:geom_braid}
Fix once and for all $\TheStrand$ distinct points $p_1,...,p_\TheStrand$ in the interior of a compact disks $D$. A \textit{geometric braid} is a set $b\subseteq\inter{D}\times I$ formed by $\TheStrand$ disjoint topological intervals $b_\Strand,1\leq\Strand\leq\TheStrand$, called the \textit{strands} of $b$, such that the natural projection $D\times I\rightarrow I$ maps each strand homeomorphically onto $I$ and 
\[
b\cap(D\times\{0\})=\{(p_1,0),...,(p_\TheStrand,0)\} \text{ and } b\cap(D\times\{1\})=\{(p_1,1),...,(p_\TheStrand,1)\}. 
\]
The above conditions imply that for each $\Strand$ the strand~$b_\Strand$ meets each disk $D\times\{t\}$ at exactly one point and connects $(p_\Strand,0)$ to $(p_{\sigma(\Strand)},1)$ for some $\sigma\in\Sym_\TheStrand$. Two geometric braids $b$ and $b'$ are \textit{isotopic} if $b$ can be continuously deformed into $b'$ inside the class of geometric braids.  
The operation of stacking braids along the $I$-factor of $D\times I$ descends to isotopy classes, giving a group structure on the set $\B_\TheStrand$ of isotopy classes of braids with $\TheStrand$ strands. $\B_\TheStrand$ is called the \textit{Artin braid group} on $\TheStrand$ strands. 
\end{definition} 

A geometric braid is pictured in a cylinder with the disk $D\times\{0\}$ at the top
and $D\times\{1\}$ at its bottom. By definition, the interval factor of $D\times I$ parametrizes each strand of the braid. So we typically think of the strands as oriented arcs $b_\Strand:I\rightarrow D\times I$ traversing the cylinder from top to bottom 
(see Figure \ref{fig:geom_braid_and_diagram}, left).

\begin{figure}[H]
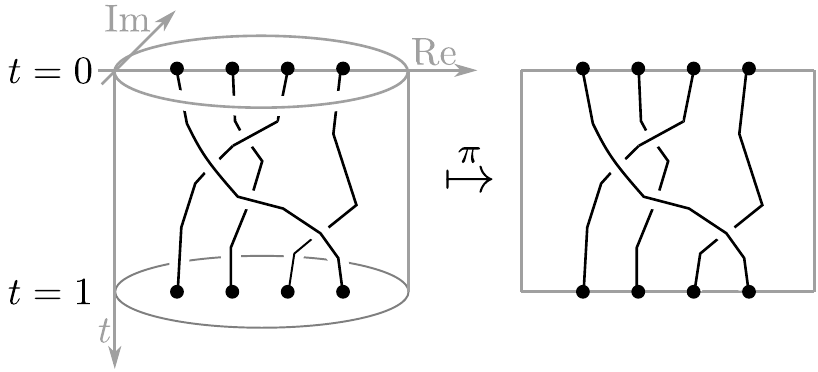
\caption{A geometric braid and its Artin braid diagram.}
\label{fig:geom_braid_and_diagram}
\end{figure}

A braid inside its ambient cylinder is a three-dimensional object. 
An Artin braid diagram 
is designed to capture the information of this 
object in a two dimensional picture. 

\begin{definition}[{Artin braid diagrams}, {\cite[Section 1.2.2]{KasselTuraev2008}}]
\label{def:braid_diagram}
Assume that $D\subseteq\CC$ is the compact disk centered at $\frac{\TheStrand+1}{2}$ with radius $\frac{\TheStrand+1}{2}$ and $p_\Strand:=\Strand$ for each $1\leq\Strand\leq\TheStrand$ (see Figure \ref{fig:embedding_disk}). 
\begin{figure}[H]
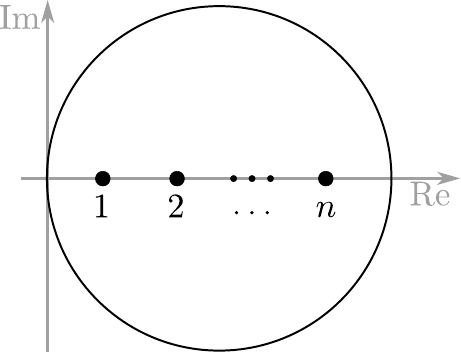
\caption{The embedding of the disk $D$.}
\label{fig:embedding_disk}
\end{figure}
Moreover, let $u_\Strand:I\rightarrow D$ be the continuous map such that the $\Strand$-th strand $b_\Strand$ meets $D\times\{t\}$ at the point $(u_\Strand(t),t)$ for each $t\in I$. 
If the map $u_\Strand$ is piecewise linear for each~$\Strand$, the corresponding geometric braid is called \textit{piecewise linear}. 
Define 
\[
\pi:D\times I\rightarrow[0,\TheStrand+1]\times I, (z,t)\mapsto(\text{Re}(z),t). 
\]
The image $\pi(b)$ of a geometric braid $b$ is called the \textit{projection} of $b$. 

Let $x,y\in b, x\neq y$ be such that $\pi(x)=\pi(y)=:(p,t)$. Then $x$ and $y$ are in distinct strands $b_\Str$ and $b_\Strand$, respectively. 

In this case, the point $(p,t)$ is called a \textit{crossing} at height $t$ in $\pi(b)$. 
The crossing is called \textit{transverse} if there is a neighborhood $U$ of $p$ in $[0,\TheStrand+1]\times I$ such that the pair $(U,\pi(b)\cap U)$ is locally homeomorphic to $(\RR^2,\RR\times\{0\}\cup\{0\}\times\RR)$ via a homeomorphism identifying $\pi(b_\Str)$ with $\RR\times\{0\}$ and $\pi(b_\Strand)$ with $\{0\}\times\RR$. 

\new{In the braid $b$,} the strand $b_\Str$ \textit{crosses over} $b_\Strand$ if $\textup{Im}(u_\Str(t))<\textup{Im}(u_\Strand(t))$. 
Otherwise $b_\Str$ \textit{crosses under} $b_\Strand$. 

We will consider the projection for those geometric braids $b$ which satisfy the following conditions: 
\begin{enumerate}
\item \label{def:braid_diagram_it1} 
$b$ is piecewise linear, 
\item \label{def:braid_diagram_it2}
at most one pair of strands crosses at a height and
\item \label{def:braid_diagram_it3} 
the strands cross transversely in each crossing. 
\end{enumerate}

In this case, the projection $\pi(b)$ together with the data of which strand crosses over (resp. under) is called an \textit{Artin braid diagram} for $b$. If we draw an Artin braid diagram, an under-crossing strand is indicated by a line that is broken near the crossing; an over-crossing strand is represented by a continued line (see Figure~\ref{fig:geom_braid_and_diagram}, right). 
\end{definition}

\begin{observation}[Generating the Artin braid group $\B_\TheStrand$]
\label{obs:gen_set_braid_grp_and_Conf}
Given an arbitrary geometric braid $b$, there exists an isotopic braid $\tilde{b}$ such that $\pi(\tilde{b})$ with the data of which strand crosses over (resp. under) at every crossing is an Artin braid diagram. 

Further, the conditions \textup{\ref{def:braid_diagram_it1}}-\textup{\ref{def:braid_diagram_it3}} allow us to decompose $b$ into pieces $b\cap(D\times[t_{\Subdiv-1},t_\Subdiv])$ such that each piece 
contains exactly one crossing. While the first piece starts at $p_1,...,p_\TheStrand$ and the last piece ends in these points, the other pieces a priori neither \new{start at or end in} the points $p_1,...,p_\TheStrand$. However, an isotopy that pulls back the endpoints of every piece (see Figure \textup{\ref{fig:geometric_braid_iso}}) allows us to assume that each piece connects $p_1,...,p_\TheStrand$ to $p_{\sigma(1)},...,p_{\sigma(\TheStrand)}$ for a permutation $\sigma\in\Sym_\TheStrand$ depending on the piece. Since each piece contains only one crossing, the crossing strands are adjacent. Consequently, each of these pieces is isotopic to a braid from Figure \textup{\ref{fig:gens_B_n}} or an inverse. 
Hence the braids~$h_\Strand$ for $1\leq\Strand<\TheStrand$ generate the Artin braid group~$\B_\TheStrand$. 
\end{observation}

\begin{figure}[H]
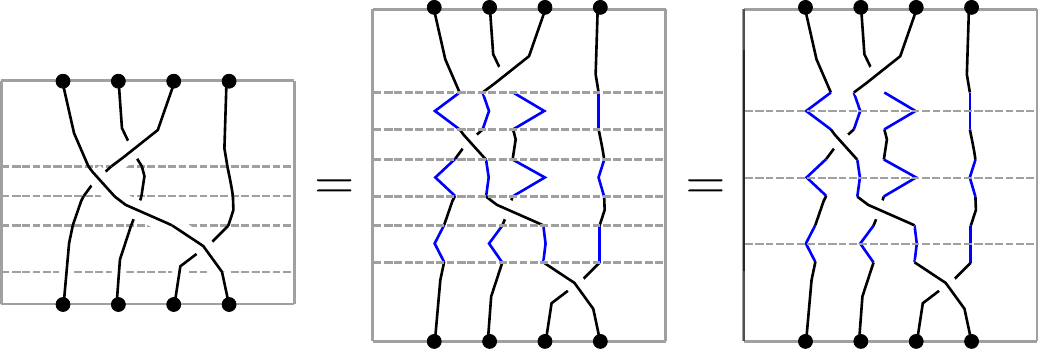
\caption{A decomposition of a braid into generators.}
\label{fig:geometric_braid_iso}
\end{figure}

\begin{figure}[H]
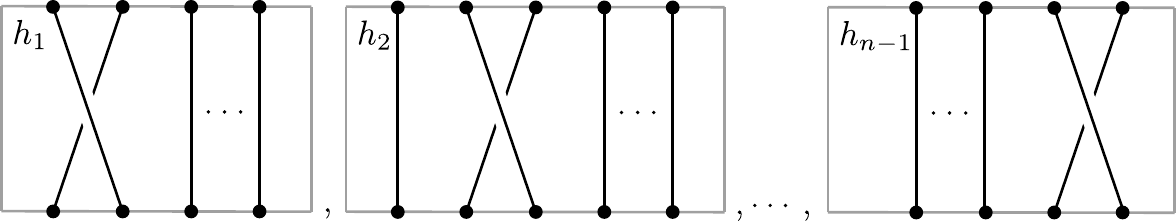
\caption{Generators of $\B_\TheStrand$.}
\label{fig:gens_B_n}
\end{figure}

\begin{observation}[Geometric braids and configuration spaces]
A geometric braid~$b$ corresponds to a closed path with base point $\{p_1,...,p_\TheStrand\}$ in the configuration space 
\[
\Conf_\TheStrand(\inter{D}):=\{(x_1,...,x_\TheStrand)\in (\inter{D})^\TheStrand\mid x_\Str\neq x_\Strand,1\leq\Str,\Strand\leq\TheStrand,\Str\neq\Strand\}/\Sym_\TheStrand 
\]
mapping $b$ to $\{u_1(t),...,u_\TheStrand(t)\}$ and vice versa. Moreover, two geometric braids are isotopic if and only if the corresponding paths in the configuration space are homotopic. Hence $\B_\TheStrand$ is isomorphic to $\pi_1(\Conf_\TheStrand(\inter{D}))$, see \cite[Section 1.4]{KasselTuraev2008} for further details. 
\end{observation}

\subsection{The definition of orbifold braid groups}
\label{subsec:def_orb_braid_grps}

The next goal is to establish a similar projection that induces braid diagrams for \textit{orbifold braids}. We begin with the definition of \textit{orbifold braid groups} as the orbifold fundamental group of an orbifold configuration space. In particular, the following definition is equivalent to the definition given in \cite{Allcock2002}. 

\newcommand{\ConfnOrb}{\Conf_\TheStrand^{\G}(M_\G)}
\newcommand{\PConfnOrb}{\PConf_\TheStrand^{\G}(M_\G)}

\begin{definition}[Orbifold braid group]
\label{def:orbi_braid_grp}
Let $M_\G$ be an orbifold. The orbifold 
\[
\PConfnOrb:=(M^\TheStrand\setminus\Delta_\TheStrand^\G(M))_{\G^\TheStrand} 
\]
with $\Delta_\TheStrand^\G(M)=\{(x_1,...,x_\TheStrand)\in M^\TheStrand\mid x_i=\g(x_j)$ for some $\g\in\G,i\neq j\}$ is called the \textit{$\TheStrand$-th pure configuration space} over $M_\G$. Since $\G$ acts properly on $M$, the coordinatewise action of $\G^\TheStrand$ on $(M^\TheStrand\setminus\Delta_\TheStrand^\G(M))$ is also proper. Hence $\PConfnOrb$ is an orbifold and its orbifold fundamental group $\piOrb(\PConfnOrb)$ is called the \textit{$\TheStrand$-th pure orbifold braid group}, denoted by $\PZ_\TheStrand(M_\G)$. The orbifold 
\[
\ConfnOrb:=(M^\TheStrand\setminus\Delta_\TheStrand^\G(M))_{\G^\TheStrand\rtimes\Sym_\TheStrand}
\]
is called the \textit{$\TheStrand$-th configuration space} over $M_\G$. As above, the normal subgroup~$\G^\TheStrand$ acts coordinatewise and $\Sym_\TheStrand$ acts (on $\G^\TheStrand$ and $M^\TheStrand\setminus\Delta_\TheStrand^\G(M)$) via permutation of coordinates. The $\G^\TheStrand\rtimes\Sym_\TheStrand$-action is also proper, i.e.\ $\ConfnOrb$ is an orbifold. Its orbifold fundamental group $\piOrb(\ConfnOrb)$ is called the \textit{$\TheStrand$-th orbifold braid group}, denoted by $\Z_\TheStrand(M_\G)$. 
\end{definition}

\subsection{A decomposition of orbifold braids into strands}
\label{subsec:decomp_orb_braids}

At first we observe that elements in orbifold braid groups decompose into \textit{strands}. 


\begin{observation}
\label{obs:closed_path_config_space}
A closed $\G^\TheStrand\rtimes\Sym_\TheStrand$-path $\xi$ in $\ConfnOrb$ is equivalent to a $\G^\TheStrand\rtimes~\Sym_\TheStrand$-path that corresponds to an $\TheStrand$-tuple 
$(\xi_1,...,\xi_\TheStrand)$ of $\G$-paths 
\[
\xi_\Strand=\left(\g_0^\Strand,c_1^\Strand,\g_1^\Strand,
...,c_\TheNSubdiv^\Strand,\g_\TheNSubdiv^\Strand\right) 
\]
in $M_\G$. The $\G$-paths $\xi_\Strand$ have the initial points $p_\Strand$ and terminal points $p_{\sigma(\Strand)}$ for $\sigma\in\Sym_\TheStrand$. Moreover, these $\G$-paths share a subdivision $0=t_0\leq...\leq~t_\TheNSubdiv=1$ and 
satisfy the condition 
\[
c_\Subdiv^\NStrand(t)\cap\G(c_{\Subdiv}^\NNStrand(t))=\emptyset 
\]
for all $t\in I$, $1\leq\NStrand,\NNStrand\leq\TheStrand,\NStrand\neq\NNStrand$ and a suitable $1\leq\Subdiv\leq\TheNSubdiv$ depending on $t$. 

For $\sigma=id_\TheStrand$, the $\G^\TheStrand\rtimes\Sym_\TheStrand$-path $\xi$ induces a $\G^\TheStrand$-path that represents an element in $\PZ_\TheStrand(M_\G)$. In particular, $\PZ_\TheStrand(M_\G)$ is a subgroup of $\Z_\TheStrand(M_\G)$. 
\end{observation}

For a closed $\G^\TheStrand\rtimes\Sym_\TheStrand$-path $\xi$, the $\G$-paths $\xi_\Strand$ as described above are called the \textit{strands}. 
We fix the notation $\xi$ for a $\G^\TheStrand\rtimes\Sym_\TheStrand$-path with strands $\xi_\Strand$ of the same form as in Observation \ref{obs:closed_path_config_space} for the rest of the section. 

While the Artin braid diagrams keep track of the crossings of strands, we want orbifold braid diagrams to keep track of crossings and the $\freeprod$-leaps 
in the strands of a $\G^\TheStrand\rtimes\Sym_\TheStrand$-path. 

If $\xi_\Strand$ does not contain any $\freeprod$-leaps in $[a,b]$, we may consider $\xi_\Strand\vert_{[a,b]}:[a,b]\rightarrow M$ as a continuous function. If $\xi_\Strand$ contains a $\freeprod$-leap at time~$t_\Subdivision$, 
the group element $\g_\Subdivision^\Strand$ translates $c_{\Subdivision+1}^\Strand(t_\Subdivision)$ to $c_\Subdivision^\Strand(t_\Subdivision)$. In this case, we consider $\xi_\Strand$ as the following continuous function defined on the disjoint union $\bigcupdot_{\Subdiv=1}^{\TheNSubdiv}[t_{\Subdiv-1},t_\Subdiv]$: 
\[
\bigcupdot_{\Subdiv=1}^\TheNSubdiv[t_{\Subdiv-1},t_\Subdiv]\rightarrow M, s\mapsto c_\Subdivision^\Strand(s) \text{ for } s\in[t_{\Subdiv-1},t_\Subdiv]. 
\]
Further, let $\xi$ denote the union $\xi_\Strand\left(\bigcupdot_{\Subdiv=1}^\TheNSubdiv[t_{\Subdiv-1},t_\Subdiv]\right)$ inside $\ConfnOrb$. 

\subsection{Orbifold braids in $\Z_\TheStrand(\Sigma_\freeprod(\ThePct))$ and their braid diagrams}
\label{subsec:Z_n_Sigma_freeprod_braid_diagrams}

To specify elements of orbifold braid groups through braid diagrams, we want to establish a similar projection as in Figure \ref{fig:geom_braid_and_diagram}. Therefore, we restrict to the orbifolds~$\Sigma_\freeprod$ from Example~\ref{ex:good_orb_free_prod}. 
Even though in this case the underlying surface~$\Sigma$ 
embeds into the complex plane, it is not suitable for our purpose to consider a projection from $\Sigma\subseteq\CC$ to $\RR$ directly. Since a direct projection would not allow us to recover the projected orbifold braid group element, we will instead restrict to the fundamental domain $\FD$ before projecting. 


\new{Before we explain the projection, we endow the fundamental domain $\FD$ with a set of marked points $r_1,...,r_\ThePct$ in the interior of $\FD$ such that $\freeprod(r_\Pc)\neq\freeprod(r_\NPct)$ for all $1\leq\Pc,\NPct\leq\ThePct$ with $\Pc\neq\NPct$. Removing all the $\freeprod$-translates of these points, we obtain the surface 
\begin{equation}
\label{eq:Sigma_minus_pct}
\S:=\Sigma\setminus\freeprod(\{r_1,...,r_\ThePct\}) 
\end{equation}
with a proper $\freeprod$-action. The set $\FD(\ThePct):=\FD\cap\S$ is a fundamental domain of the $\freeprod$-action on $\S$. Let $\SG$ denote the induced orbifold structure on $\S$. 

If we further remove the cone points $\freeprod(\{\cp_1,...,\cp_\TheCone\})$ from $\S$, we denote the resulting surface by $\Spct$. The $\freeprod$-action on $\Spct$ has a fundamental domain $\FD(\ThePct,\TheCone):=\FD\cap\Spct$. Let $\SGpct$ denote the induced orbifold structure. 

Recalling the shape of the fundamental domain $\FD$ from Figure \ref{fig:fund_domain}, we can embed $\FD(\ThePct)$ (likewise $\FD(\ThePct,\TheCone)$) in $\CC$ as the disk of radius $\frac{\TheStrand+\ThePct+\TheCone+2}{2}$ centered at $\frac{\TheStrand-\ThePct-\TheCone}{2}$. For each $1\leq\nu\leq\TheCone$, let $\cp_\nu$ be the upper boundary point of $\partial \FD(\ThePct)$ with $\text{Re}(\cp_\nu)=-\ThePct-\nu$, for each $1\leq\NPct\leq\ThePct$ let $r_\NPct$ be the point $-\NPct\in\RR$ and for each $1\leq\Strand\leq\TheStrand$, let $p_\Strand$ be the point $\Strand$ in $\RR$ (see Figure \ref{fig:embedding_fund_domain}). Moreover, recall that each cone point in $\partial\FD(\ThePct)$ has two adjacent arcs that lie in $\partial\FD(\ThePct)$. For technical reasons, let us assume that the arcs adjacent to $\cp_\nu$ embed into $\partial\FD(\ThePct)$ as the boundary arcs with positive imaginary part and real part between $-\ThePct-\nu-\frac{1}{2}$ and $-\ThePct-\nu$ or $-\ThePct-\nu$ and $-\ThePct-\nu+\frac{1}{2}$, respectively. This is not needed in this section but will be helpful in Section \ref{sec:braid_and_mcg}.} 

\begin{figure}[H]
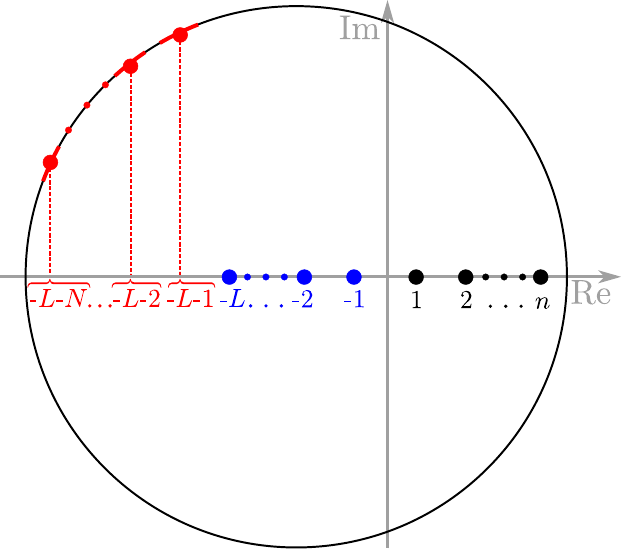
\caption{The embedding of the fundamental domain $\FD(\ThePct)$ into $\CC$.}
\label{fig:embedding_fund_domain}
\end{figure}

\begin{proposition}[Reduction of $\freeprod^\TheStrand\rtimes\Sym_\TheStrand$-paths]
\label{prop:paths_rep_by_braids}
Every element in $\Z_\TheStrand(\Sigma_\freeprod(\ThePct))$ can be represented by a $\freeprod^\TheStrand\rtimes\Sym_\TheStrand$-path $\xi$ whose strands $\xi_\Strand=(\gamma_0^\Strand,c_1^\Strand,\gamma_1^\Strand,...,c_\TheSubdiv^\Strand,\gamma_\TheSubdiv^\Strand)$ satisfy the following conditions. For each $1\leq\Subdiv\leq\TheSubdiv$ and $1\leq\Strand\leq\TheStrand$, 
\begin{enumerate}
\item \label{it_simpl_paths_pw_geod} $c_\Subdiv^\Strand$ is piecewise linear with image in the interior of $\Sigma(\ThePct)$. 
\item \label{it_simpl_paths_pw_no_cone_pt} 
$c_\Subdiv^\Strand$ does not intersect any cone points. 
\item \label{it_simpl_paths_pw_red_fd} $c_\Subdivision^\Strand([t_{\Subdivision-1},t_\Subdivision])\subseteq \FD(\ThePct)$. 
\end{enumerate}

The same holds for every element in $\Z_\TheStrand(\Sigma_\freeprod(\ThePct,\TheCone))$. 
\end{proposition}

To adjust $\xi$, such that it satisfies the above properties, we use so called \textit{$\Delta$-moves}. 

\begin{definition}[{$\Delta$-move}, {\cite[p.\ 11]{KasselTuraev2008}}]
Let $y_0:=(x_0,t_0),y_1:=(x_1,t_1)$ and $y_2:=(x_2,t_2)$ be three points in $\Sigma(\ThePct)\times I$ with $t_0<t_1<t_2$ such that the linear $2$-simplex~$\Delta$ spanned by the points is contained in $\Sigma(\ThePct)\times I$. In particular, $\Delta$ does not contain any punctures. Further, let $\xi$ be a $\freeprod^\TheStrand\rtimes\Sym_\TheStrand$-path that represents an element in $\Z_\TheStrand(\Sigma_\freeprod(\ThePct))$. If the $\freeprod$-orbits of the strands of $\xi$ intersect $\Delta$ precisely along the linear segment $\overline{y_0y_2}$, we may replace $\overline{y_0y_2}$ in the orbifold braid $\xi$ by the concatenation $\overline{y_0y_1}\cup\overline{y_1y_2}$. Since $\Delta$ does not intersect any punctures or $\freeprod$-translates of further strands, the resulting $\freeprod^\TheStrand\rtimes\Sym_\TheStrand$-path is homotopic to $\xi$. Due to the bounded $2$-simplex, we call the above operation and its inverse \textit{$\Delta$-moves}. 
%
%
\end{definition}
\begin{proof}[Proof of Proposition \textup{\ref{prop:paths_rep_by_braids}}]
By Lemma \ref{lem:cont_path_homo_it1}, each strand $\xi_\Strand$ is equivalent to a unique continuous $\freeprod$-arc $(\gamma,c)$. Since the endpoints of $c$ lie in the interior of $\S$, it can be homotoped (relative endpoints) such that $c$ lies entirely in the interior of $\S$. Using piecewise linear approximation in $\CC$, we can find an approximation of $c$ that lies in the open subspace $\inter{\Sigma}(\ThePct)$ of $\CC$, i.e.\ in every strand the continuous parts $c_\Subdivision^\Strand$ are paths inside $\S$ which embeds into $\CC$, whence \ref{it_simpl_paths_pw_geod}. 

If $\xi$ represents an element in $\Z_\TheStrand(\Sigma_\freeprod(\ThePct,\TheCone))$, property \ref{it_simpl_paths_pw_no_cone_pt} is automatically satisfied. If $\xi$ represents an element in $\Z_\TheStrand(\Sigma_\freeprod(\ThePct))$, we begin with reducing the number of situations where a strand stays \new{at} a cone point for a period of time. \new{By subdivision of these intervals,} we can assume that $\freeprod$-leaps do not occur in the interior of these intervals. Now performing a $\Delta$-move on the constant pieces 
(see Figure \ref{fig:red_cone_pts}, top) allows us to assume that paths may intersect with but do not stay in cone points. Another $\Delta$-move (as indicated in the bottom half of Figure~\ref{fig:red_cone_pts} possibly affected by a $\freeprod$-leap) allows us to remove the remaining cone point intersections. 

\begin{figure}[H]
\centerline{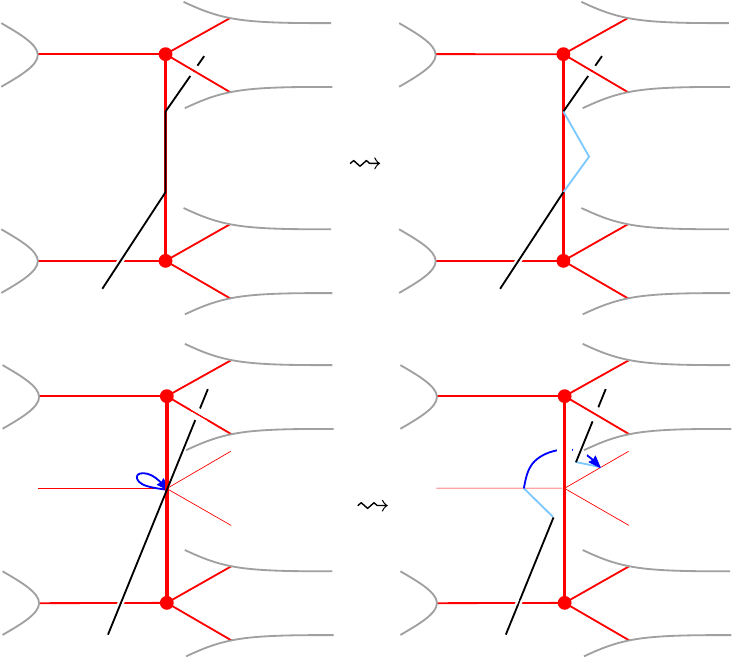}
\caption{Removing cone point intersections.}
\label{fig:red_cone_pts}
\end{figure}

Since each strand of $\xi$ contains only piecewise linear paths, we can adjust the contained paths such that 
there are only finitely many intersections with $\freeprod(\partial \FD(\ThePct))$. Subdividing $\xi$ at all times with a boundary intersection, the application of suitable shifts reduces $\xi$ to the fundamental domain $\FD(\ThePct)$ as claimed in \ref{it_simpl_paths_pw_red_fd}. 

\new{For \ref{it_simpl_paths_pw_geod} and \ref{it_simpl_paths_pw_red_fd}, the same arguments apply if $\xi$ is a $\freeprod^\TheStrand\rtimes\Sym_\TheStrand$-path that represents an element in $\Z_\TheStrand(\Sigma_\freeprod(\ThePct,\TheCone))$.} 
\end{proof} 

\begin{corollary}
\label{cor:surject_Z_n_ast_to_Z_n}
The homomorphisms 
\[
\Z_\TheStrand(\SGpct)\rightarrow\Z_\TheStrand(\SG) \; \text{ and } \; \PZ_\TheStrand(\SGpct)\rightarrow\PZ_\TheStrand(\SG) 
\]
induced by the inclusion $\SGpct\hookrightarrow\SG$ are surjective. 
\end{corollary}

Recall from Definition \ref{def:geom_braid} that geometric braids are strands inside a cylinder. \new{Due to the reduction from Proposition \ref{prop:paths_rep_by_braids},} we have a similar picture for orbifold braids. In this case, the strands are contained in a cylinder with base $\FD(\ThePct)$ (see Figure~\ref{fig:orb_braid_and_diagram}, left). In contrast to Artin braids, the strands of orbifold braids may have finitely many discontinuity points. \new{At these points,} a $\freeprod$-leap compensates the gap 
between the adjacent pieces of the strand. 

\begin{figure}[H]
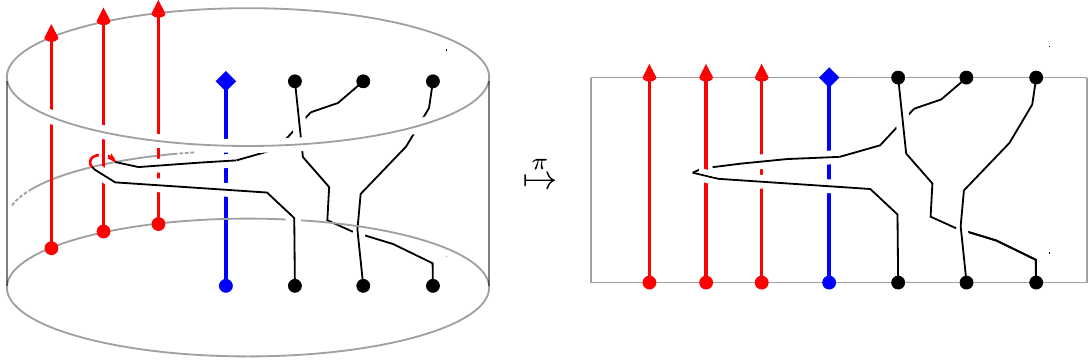
\caption{An orbifold braid and its braid diagram.}
\label{fig:orb_braid_and_diagram}
\end{figure}

\new{As indicated in Figure \ref{fig:orb_braid_and_diagram},} this picture allows us to describe a similar projection as in Figure \ref{fig:geom_braid_and_diagram}. 

\begin{definition}[Orbifold braids and orbifold braid diagrams]
\label{def:braid_proj_cross}
A $\freeprod^\TheStrand\rtimes\Sym_\TheStrand$-path~$\xi$ that satisfies the properties \ref{it_simpl_paths_pw_geod}-\ref{it_simpl_paths_pw_red_fd} 
is called an \textit{orbifold braid} with strands $\xi_\Strand$ for $1\leq\Strand\leq\TheStrand$. 
Orbifold braids will be specified by a projection 
\begin{align*}
\pi:\FD(\ThePct)\times\bigcupdot_{\Subdiv=1}^\TheNSubdiv[t_{\Subdiv-1},t_\Subdiv]\rightarrow\RR\times\bigcupdot_{\Subdiv=1}^\TheNSubdiv[t_{\Subdiv-1},t_\Subdiv]
\end{align*}
where the $\FD(\ThePct)$-coordinate projects to the real part with respect to the chosen embedding of $\FD(\ThePct)$ into $\CC$ and the interval coordinate $t$ maps identically to $\bigcupdot_{\Subdiv=1}^\TheNSubdiv[t_{\Subdiv-1},t_\Subdiv]$. 
The image $\pi(\xi)$ of an orbifold braid is called the \textit{projection} of $\xi. $

Analogously to Definition~\ref{def:braid_diagram}, we fix the following notations. 
Let $x,y\in\xi, x\neq y$ be such that $\pi(x)=\pi(y)=:(p,t)$, then $x$ and $y$ are in distinct strands $\xi_\Str$ and $\xi_\Strand$, respectively. 

In this case, the point $(p,t)$ is called a \textit{crossing} at height $t$ in $\pi(\xi)$. The crossing is called \textit{transverse} if there is a neighborhood $U$ of $(p,t)$ in $[0,\TheStrand+1]\times[t_{\Subdiv-1},t_\Subdiv]$ such that the pair $(U,\pi(\xi)\cap U)$ is locally homeomorphic to $(\RR^2,\RR\times\{0\}\cup\{0\}\times\RR)$ via a homeomorphism identifying $\pi(\xi_\Str)$ with $\RR\times\{0\}$ and $\pi(\xi_\Strand)$ with $\{0\}\times\RR$. 

In the orbifold braid $\xi$ the strand $\xi_\Str$ \textit{crosses over} $\xi_\Strand$ if $\textup{Im}(\xi_\Str(t))<\textup{Im}(\xi_\Strand(t))$. Otherwise $\xi_\Str$ \textit{crosses under} $\xi_\Strand$. 

Similarly, we consider crossings with the punctures $r_\NPct$. For this purpose, let $x\in\xi_\Str$ be such that $\pi(x)=(-\NPct,t)$ for some $1\leq\NPct\leq\ThePct$. 

In this case, the point $(-\NPct,t)$ is called a \textit{crossing} of $\pi(\xi)$. The crossing is called \textit{transverse} if there is a neighborhood $U$ of $p$ in $[0,\TheStrand+1]\times[t_{\Subdiv-1},t_\Subdiv]$ such that the pair 
\[
(U,(\pi(\xi)\cup\{-\NPct\}\times I)\cap U)
\]
is locally homeomorphic to $(\RR^2,\RR\times\{0\}\cup\{0\}\times\RR)$ via a homeomorphism identifying $\pi(\xi_\Str)$ with $\RR\times\{0\}$ and $\{-\NPct\}\times I$ with $\{0\}\times\RR$. 

In the orbifold braid $\xi$ the strand $\xi_\Str$ \textit{crosses over} $r_\NPct$ if $\textup{Im}(\xi_\Str(t))<0$. Otherwise $\xi_\Str$ \textit{crosses under} $r_\NPct$. 

We will consider the projection mainly for those orbifold braids $\xi$ which satisfy the following conditions: 
\begin{enumerate}
\item \label{def:orb_braid_diag_it1} 
at most one crossing, either of two strands or a strand and a puncture, appears at a height, 
\item \label{def:orb_braid_diag_it2} 
the strands and punctures cross transversely in each crossing, 
\item \label{def:orb_braid_diag_it4} 
no crossing occurs at the same time as a $\freeprod$-leap and 
\item \label{def:orb_braid_diag_it5} 
no two $\freeprod$-leaps occur at the same time. 
\end{enumerate}

In this case, the projection $\pi(\xi)$ of an orbifold braid $\xi$ together with the data of which strand crosses over (resp. under) and the data which strand at which time contains a $\freeprod$-leap is called an \textit{orbifold braid diagram} for $\xi$. 
\end{definition}

\new{As for the Artin braids,} we can encode the orbifold braid diagrams in pictures (see Figure \ref{fig:orb_braid_and_diagram}, right for an example).  The crossings of two strands and the crossings of a strand and a puncture are illustrated as for the geometric braids. Further, recall that we have chosen cyclic generators $\gamma_\nu$ of the cyclic factors $\cyc{\TheOrder_\nu}$ in $\freeprod$ in Example~\ref{ex:good_orb_free_prod}. If a strand $\xi_\Strand$ contains a $\freeprod$-leap $\gamma_\Subdiv^\Strand=\gamma_\nu^\varepsilon$ at time $t_\Subdiv$, we draw the $\Strand$-th strand encircling the bar that corresponds to the $\nu$-th cone point. For $\varepsilon=1$, we draw the $\Strand$-th strand encircling the $\nu$-th cone point bar counterclockwise. For $\varepsilon=-1$, we draw the $\Strand$-th strand encircling the $\nu$-th cone point bar clockwise. Due to condition \ref{it_simpl_paths_pw_red_fd}, a $\freeprod$-leap by $\gamma_\nu^l$ with $l\not\in\{\pm1\}$ cannot occur. 

Proposition \ref{prop:paths_rep_by_braids} together with the standard techniques from the classical case and the shifts introduced in Definition \ref{def:equiv_paths_shift} yields: 

\begin{lemma}
\label{lem:braids_crossing_pattern}
Every element in $\Z_\TheStrand(\Sigma_\freeprod(\ThePct))$ and $\Z_\TheStrand(\SGpct)$ is represented by an orbifold braid that projects to an orbifold braid diagram. 
\end{lemma}

\new{In the following,} we will use the orbifold braid diagrams introduced above to encode orbifold braids. We will no longer distinguish between an orbifold braid and its homotopy class in $\Z_\TheStrand(\SG)$ and $\Z_\TheStrand(\SGpct)$, respectively. \new{Motivated by Artin braid groups,} we will begin with orbifold braids that do not braid with cone points or punctures: We define the orbifold braid $h_\Strand$ for $1\leq\Strand<\TheStrand$ as the one represented by the following braid diagram: 

\begin{figure}[H]
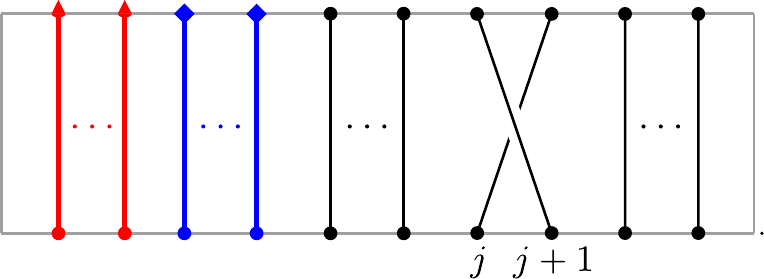
\caption{The generator $h_\Strand$.}
\label{fig:h_j}
\end{figure}

These orbifold braids 
generate a subgroup of $\Z_\TheStrand(\Sigma_\freeprod(\ThePct))$ and $\Z_\TheStrand(\Sigma_\freeprod(\ThePct,\TheCone))$, respectively, that is isomorphic to the Artin braid group $\B_\TheStrand$. We can further define 
\begin{equation}
\label{eq:def_a_ji}
\twist{\Strand}{\Str}:=h_{\Strand-1}^{-1}...h_{\Str+1}^{-1}h_\Str^2h_{\Str+1}...h_{\Strand-1} \text{ for } 1\leq\Str<\Strand\leq\TheStrand, 
\end{equation}
which is a braid in $\PZ_\TheStrand(\Sigma_\freeprod(\ThePct))$ with the following projection 
\begin{figure}[H]
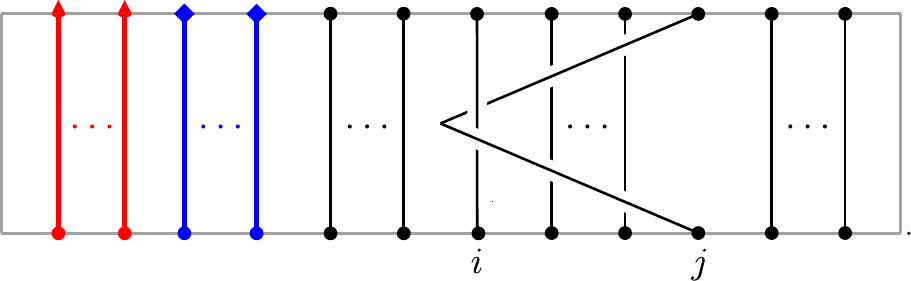
\caption{Braid diagram of $\twist{\Strand}{\Str}$.}
\label{fig:a_ji}
\end{figure}

Further, we introduce orbifold braids $\twP{\NPct}$ for $1\leq\NPct\leq\ThePct$ and $\twC{\nu}$ for $1\leq\nu\leq\TheCone$ that involve either cone points or punctures. \new{In both cases,} the last $\TheStrand-1$ strands are fixed. The first strands are pictured in Figure \ref{fig:tau_c_nu_tau_r_lambda}. 
\begin{figure}[h]
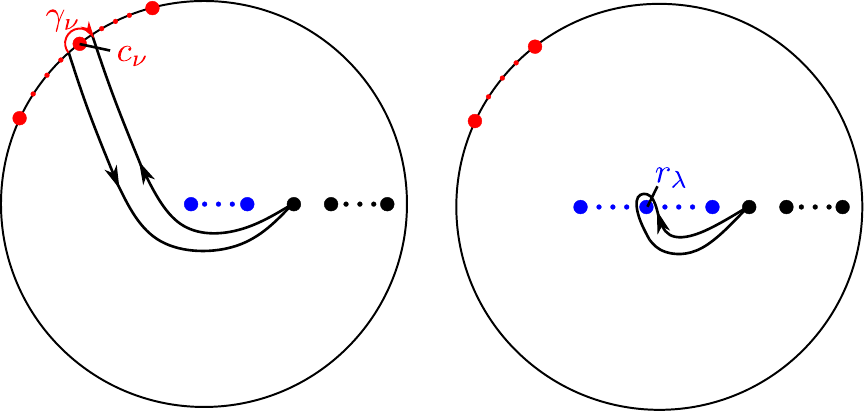
\caption{The first strand of $\twC{\nu}$ (left) and $\twP{\NPct}$ (right).}
\label{fig:tau_c_nu_tau_r_lambda}
\end{figure}

For $1\leq\NStrand\leq\TheStrand$, $1\leq\nu\leq\TheCone$ and $1\leq\NPct\leq\ThePct$, we further define 
\begin{equation}
\label{eq:def_a_kc_a_kr}
\twistC{\NStrand}{\nu}:=h_{\NStrand-1}^{-1}...h_1^{-1}\twC{\nu}h_1...h_{\NStrand-1} \; \text{ and } \; \twistP{\NStrand}{\NPct}:=h_{\NStrand-1}^{-1}...h_1^{-1}\twP{\NPct}h_1...h_{\NStrand-1}. 
\end{equation}

These elements project to the diagrams depicted in Figure \ref{fig:a_kc_nu_a_kr_lambda} below. 
\begin{figure}[H]
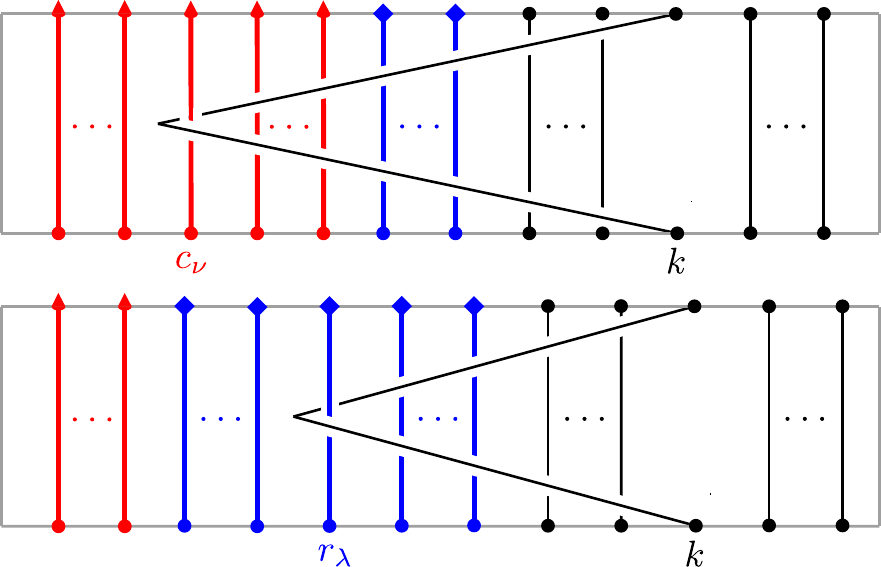
\caption{Braid diagrams of $\twistC{\NStrand}{\nu}$ (left) and $\twistP{\NStrand}{\NPct}$ (right).}
\label{fig:a_kc_nu_a_kr_lambda}
\end{figure}

\begin{remark}
\label{rem:braid_fin_order}
Even though orbifold braid diagrams look similar to Artin braid diagrams, there is an essential difference: For each $1\leq\NStrand\leq\TheStrand$ and $1\leq\nu\leq\TheCone$ the $\TheOrder_\nu$-th power of $\twistC{\NStrand}{\nu}$ is the trivial braid in $\Z_\TheStrand(\SG)$. That is because $\twistC{\NStrand}{\nu}^{\TheOrder_\nu}$ is homotopic to a braid with all strands except the $\NStrand$-th one fixed. Further, we may apply shifts to the $\NStrand$-th strand, so that this strand is a continuous $\freeprod$-path and encircles the $\nu$-th cone point (see Figure \ref{fig:gens_satisfy_rels_free_prod_t^m}). Since the loop is contractible in $\S$, this implies that $\twistC{\NStrand}{\nu}^{\TheOrder_\nu}$ is trivial. Hence, $\twistC{\NStrand}{\nu}$ is an element of finite order in $\Z_\TheStrand(\S)$. This behavior was already emphasized by Allcock \cite{Allcock2002}. 

\begin{figure}[H]
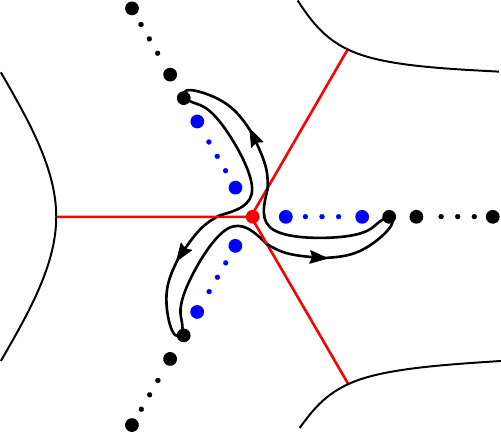
\caption{The relation $\twistC{\NStrand}{\nu}^{\TheOrder_\nu}=1$ for $\NStrand=1$.}
\label{fig:gens_satisfy_rels_free_prod_t^m}
\end{figure}

In contrast, the loop in Figure \ref{fig:gens_satisfy_rels_free_prod_t^m} is not contractible if we remove the cone point, i.e.\ $\twistC{\NStrand}{\nu}^{\TheOrder_\nu}$ is not trivial in $\Z_\TheStrand(\SGpct)$. This prevents the epimorphisms in Corollary \ref{cor:surject_Z_n_ast_to_Z_n} from being injective. 

Comparing orbifold braid diagrams for braids in $\Z_\TheStrand(\Sigma_\freeprod(\ThePct))$ to Artin braid diagrams, the relations $\twC{\nu}^{\TheOrder_\nu}\stackrel{\eqref{eq:def_a_kc_a_kr}}=\twistC{1}{\nu}^{\TheOrder_\nu}=1$ for $1\leq\nu\leq\TheCone$ reflect additional transformations allowed for the modification of orbifold braid diagrams (see Figure \ref{fig:orb-Reidemeister-move}). In analogy to the classical case, discussed in \cite[Section 1.2.3]{KasselTuraev2008}, this transformations can be seen as an additional Reidemeister moves for orbifold braid diagrams. 
\end{remark}

\begin{figure}[H]
\centerline{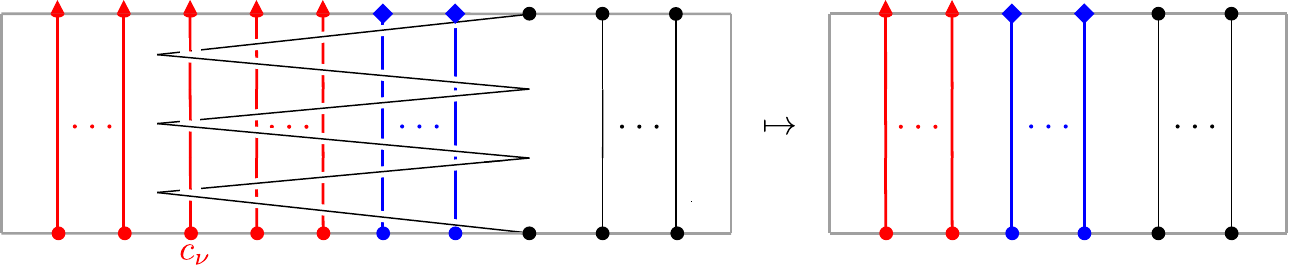}
\caption{Additional Reidemeister move for orbifold braid diagrams in $\Z_\TheStrand(\Sigma_\freeprod(\ThePct))$ for $\TheOrder_\nu=3$.}
\label{fig:orb-Reidemeister-move}
\end{figure}

\subsection{A finite generating set of $\Z_\TheStrand(\Sigma_\freeprod(\ThePct))$}
\label{subsec:fin_gen_set_Z_n_Sigma_freeprod}

\new{For the Artin braid group $\B_\TheStrand$ and its pure subgroup $\PB_\TheStrand$,} generating sets are well known, see for instance \cite[Theorem 1.12, Corollary 1.19]{KasselTuraev2008}. The main idea is described in Observation \ref{obs:gen_set_braid_grp_and_Conf}. Using orbifold braid diagrams, we similarly obtain generating sets for the orbifold braid group and its pure subgroup. For the pure orbifold braid group, this is analogous to \cite[Lemma 4.1]{Roushon2021}. Details can also be found in the author’s PhD thesis \cite[Section 3.5]{Flechsig2023}.

\begin{theorem}
\label{thm:gen_set_Z_n}
The orbifold braid group $\Z_\TheStrand(\SGpct)$ is generated by the elements 
\[
h_\Strand, \twP{\NPct} \; \text{ and } \; \twC{\nu} \; \text{ for } \; 1\leq\Strand<\TheStrand, 1\leq\NPct\leq\ThePct \; \text{ and } \; 1\leq\nu\leq\TheCone. 
\]
\end{theorem}

\new{By Corollary \ref{cor:surject_Z_n_ast_to_Z_n},} this in particular implies: 

\begin{corollary}
\label{cor:gen_set_Z_n}
The orbifold braid group $\Z_\TheStrand(\SG)$ is generated by the elements 
\[
h_\Strand, \twP{\NPct} \; \text{ and } \; \twC{\nu} \; \text{ for } \; 1\leq\Strand<\TheStrand, 1\leq\NPct\leq\ThePct \; \text{ and } \; 1\leq\nu\leq\TheCone. 
\]
\end{corollary}


\begin{theorem}
\label{thm:gen_set_PZ_n}
The pure orbifold braid group $\PZ_\TheStrand(\SGpct)$ is generated by 
\[
\twist{\Strand}{\Str}, \twistP{\NStrand}{\NPct} \; \text{ and } \; \twistC{\NStrand}{\nu} \; \text{ for } \; 1\leq\Str,\Strand,\NStrand\leq\TheStrand \; \text{ with } \; \Str<\Strand, \; 1\leq\NPct\leq\ThePct \; \text{ and } \; 1\leq\nu\leq\TheCone. 
\]
\end{theorem}

\new{By Corollary \ref{cor:surject_Z_n_ast_to_Z_n},} this in particular implies: 

\begin{corollary}
\label{cor:gen_set_PZ_n}
The pure orbifold braid group $\PZ_\TheStrand(\SG)$ is generated by 
\[
\twist{\Strand}{\Str}, \twistP{\NStrand}{\NPct} \; \text{ and } \; \twistC{\NStrand}{\nu} \; \text{ for } \; 1\leq\Str,\Strand,\NStrand\leq\TheStrand \; \text{ with } \;  \Str<\Strand, \; 1\leq\NPct\leq\ThePct \; \text{ and } \; 1\leq\nu\leq\TheCone. 
\]
\end{corollary}

\section{Orbifold mapping class groups}
\label{sec:basics_mcg_orb}

An important approach to the Artin braid groups is the identification with mapping class groups of punctured disks, see, for instance, \cite[Section 9.1.3]{FarbMargalit2011}. \new{In this section,} we recall the definition of orbifold mapping class groups (with marked points) and some results about them from \cite{Flechsig2023mcg}. This sets the basis to compare orbifold braid groups and orbifold mapping class groups. 

\new{Given an orbifold $M_\G$,} we want to define its mapping class group as a group of certain homeomorphisms of $M$ modulo an equivalence relation. This generalizes the concept of mapping class groups of manifolds. If the acting group $\G$ is trivial, then the orbifold mapping class group of $M_{\{1\}}$ defined below coincides with the mapping class group of $M$. 

\new{As it is usual in the case of manifolds, we consider homeomorphisms of $M$ 
that fix the boundary 
pointwise. Moreover, the orbifold mapping class group should reflect the structure of the $\G$-action on $M$. For this reason, we restrict to the subgroup $\HomeoOrb{}{M_\G,\partial M}\leq\Homeo{}{M,\partial M}$ of \textit{$\G$-equivariant} homeomorphisms, i.e.\ for each $H\in\HomeoOrb{}{M_\G,\partial M}$, we have $H(\g(x))=\g(H(x))$ for all $\g\in\G$ and $x\in M$. An \textit{ambient isotopy} is a continuous map 
\[
I\rightarrow\HomeoOrb{}{M_\G,\partial M}. 
\]
Two $\G$-equivariant homeomorphisms $H,H'$ are \textit{ambient isotopic}, denoted by $H\sim~H'$, if there exists an ambient isotopy $H_t$ with $H_0=H$ and $H_1=H'$. Subject to the equivalence relation induced by ambient isotopies, we define the orbifold mapping class group.} 

\begin{definition}[Orbifold mapping class group]
\label{def:mcg_orb}
The group of $\G$-equivariant homeomorphisms that fix the boundary pointwise modulo ambient isotopy
\[
\MapOrb{}{M_\G}:=\HomeoOrb{}{M_\G,\partial M}/\sim
\]
is called the \textit{mapping class group} of $M_\G$. 
\end{definition}

Based on the fact that $\HomeoOrb{}{M_\G,\partial M}$ is a topological group, the mapping class group also carries the structure of a topological group. \new{For an example of an orbifold mapping class group, we refer to \cite[Example 3.4]{Flechsig2023mcg}.}

\begin{definition}[Orbifold mapping class group with marked points]
\label{def:mcg_marked_pts}
Let $M_{\G}$ be an orbifold and let us fix a set of non-singular marked points $P=\{p_1,...,p_\TheStrand\}$ in~$M$ such that $\G(p_\Str)\neq\G(p_\Strand)$ for $1\leq\Str,\Strand\leq\TheStrand,\Str\neq\Strand$. 
By $\HomeoOrb{\TheStrand}{M_\G,\partial M}$ we denote the subgroup of homeomorphisms that preserve the orbit of the marked points $\G(P)$ as a set: 
\[
\{H\in\HomeoOrb{}{M_\G,\partial M}\mid H(\G(P))=\G(P)\}. 
\]
We consider these homeomorphisms up to ambient isotopies $I\rightarrow\HomeoOrb{\TheStrand}{M_\G,\partial M}$. The corresponding equivalence relation is denoted by $\sim_\TheStrand$. 
By 
\[
\MapOrb{\TheStrand}{M_\G}:=\HomeoOrb{\TheStrand}{M_\G,\partial M}/\sim_\TheStrand 
\]
we denote the \textit{orbifold mapping class group of $M_\G$ with respect to the $\TheStrand$ marked~points}. 
\end{definition}

We stress that the orbit of marked points $\G(P)$ is a discrete set. Hence, an ambient isotopy $H_t$ through $\HomeoOrb{\TheStrand}{M_\G,\partial M}$ is constant on marked points, i.e.\ $H_t(p_\Strand)=H_0(p_\Strand)=H_1(p_\Strand)$ for each $1\leq\Strand\leq\TheStrand$ and $t\in I$. \new{For an example of an orbifold mapping class group with marked points, we refer to \cite[Example 3.6]{Flechsig2023mcg}.} 

Homeomorphisms that map each marked point inside its $\G$-orbit yield the so called \textit{pure orbifold mapping class group}: 

\begin{definition}[Pure orbifold mapping class group]
\label{def:pure_mcg}
Let $\PHomeoOrb{\TheStrand}{M_\G,\partial M}$ be the group of \textit{pure homeomorphisms} 
\[
\{H\in\HomeoOrb{\TheStrand}{M_\G,\partial M}\mid H(p_\Strand)=\g_\Strand(p_{\Strand}) \text{ with } \g_\Strand\in\G \text{ for all } 1\leq\Strand\leq\TheStrand\}. 
\] 
The subgroup of $\MapOrb{\TheStrand}{M_\G}$ induced by pure homeomorphisms is called the \textit{pure orbifold mapping class group} 
\[
\PMapOrb{\TheStrand}{M_\G}:=\PHomeoOrb{\TheStrand}{M_\G,\partial M}/\sim_\TheStrand. 
\]
\end{definition}

\new{At this point,} we recall that a homeomorphism in the pure mapping class group of a manifold fixes each of the marked points. In contrast, we only require the homeomorphisms in $\PMapOrb{\TheStrand}{M_\G}$ to preserve the orbit of each marked point but not to fix the points themselves. Further, we emphasize that we allow different group actions on different orbits of marked points, i.e.\ $H(p_\Str)=\g_\Str(p_\Str)$ and $H(p_\Strand)=\g_\Strand(p_\Strand)$ with $\g_\Str\neq\g_\Strand$ for $\Str\neq\Strand$. 

\renewcommand{\Twist}{A}
\renewcommand{\TwistP}{B}
\renewcommand{\TwistC}{C}
\renewcommand{\twist}{a}
\renewcommand{\twistP}{b}
\renewcommand{\twistC}{c}

\new{In \cite{Flechsig2023mcg}, we studied the following subgroup of the orbifold mapping class group for the orbifold $\SG$ with underlying surface $\Sigma$ punctured in $\freeprod(\{r_1,...,r_\ThePct\})$: The kernel of the homomorphism  
\[
\Forget_\TheStrand^{orb}:\MapOrb{\TheStrand}{\Sigma_\freeprod(\ThePct)}\rightarrow\MapOrb{}{\Sigma_\freeprod(\ThePct)} 
\]
that forgets the marked points.} 
 
\begin{definition}
\label{def:Map_orb(ThePct)}
Let $\MapIdOrb{\TheStrand}{\Sigma_\freeprod(\ThePct)}$ denote the kernel of $\Forget_\TheStrand^{orb}$. 
This subgroup is induced by the subgroup
\[
\HomeoIdOrb{\TheStrand}{\Sigma_\freeprod(\ThePct),\partial\Sigma(\ThePct)}:=\{H\in\HomeoOrb{\TheStrand}{\Sigma_\freeprod(\ThePct),\partial\S}\mid H\sim\id_{\Sigma(\ThePct)}\} 
\]
of $\HomeoOrb{\TheStrand}{\Sigma_\freeprod(\ThePct)}$. Moreover, let $\PMapIdOrb{\TheStrand}{\Sigma_\freeprod(\ThePct)}:=\Forget_\TheStrand^{orb}\vert_{\PMapOrb{\TheStrand}{\Sigma_\freeprod(\ThePct)}}$. This subgroup is induced by the subgroup $\PHomeoIdOrb{\TheStrand}{\Sigma_\freeprod(\ThePct),\partial\Sigma(\ThePct)}$ that contains the pure homeomorphisms of $\HomeoIdOrb{\TheStrand}{\Sigma_\freeprod(\ThePct),\partial\Sigma(\ThePct)}$. 
\end{definition}

This subgroup can be identified with an analogous subgroup in $\Map{\TheStrand}{D(\ThePct,\TheCone)}$, see \cite[Proposition 4.3]{Flechsig2023mcg} for details. In particular, this yields: 

\begin{theorem}[{Birman exact sequence for orbifold mapping class groups}, {\cite[Theorem A]{Flechsig2023mcg}}]
\label{thm:Birman_es_orb}
The following diagram is a short exact sequence: 
\\
\begin{adjustbox}{center}
\begin{tikzcd}[column sep=40pt]
1\rightarrow\pi_1\left(\Conf_\TheStrand\left(D(\ThePct,\TheCone)\right)\right)\arrow[r,"\Push_\TheStrand^{orb}"] & \MapIdOrb{\TheStrand}{\Sigma_\freeprod(\ThePct)} \arrow[r,"\Forget_\TheStrand^{orb}"] & \underbrace{\MapIdOrb{}{\Sigma_\freeprod(\ThePct)}}_{=1}\rightarrow1.  
\end{tikzcd}
\end{adjustbox}
\end{theorem}

\begin{corollary}[{\cite[Theorem B]{Flechsig2023mcg}}]
\label{cor:pure_orb_mcg_ses}
The following diagram is a short exact sequence that splits:
\\
\begin{adjustbox}{center}
\begin{tikzcd}[column sep=50pt]
1\rightarrow\freegrp{\TheStrand-1+\ThePct+\TheCone} \arrow[r,"\PushPMap^{orb}"] & \PMapIdOrb{\TheStrand}{\Sigma_\freeprod(\ThePct)}\arrow[r,"\ForgetPMap^{orb}"] & \PMapIdOrb{\TheStrand-1}{\Sigma_\freeprod(\ThePct)}\rightarrow1. 
\end{tikzcd}
\end{adjustbox}
\end{corollary}

\begin{definition}
\label{def:semidir_prod}
A group $G$ is a \textit{semidirect product} with \textit{normal subgroup}~$N$ and \textit{quotient} $H$ if there exists a short exact sequence 
\[
1\rightarrow N\xrightarrow{\iota}G\xrightarrow{\pi}H\rightarrow1
\]
that has a 
section $s:H\rightarrow G$. In this case, we denote $G=N\rtimes H$. 
\end{definition}

In particular, Corollary \ref{cor:pure_orb_mcg_ses} shows $\PMapIdOrb{\TheStrand}{\Sigma_\freeprod(\ThePct)}$ has a semidirect product structure 
\[
\freegrp{\TheStrand-1+\ThePct+\TheCone}\rtimes\PMapIdOrb{\TheStrand-1}{\Sigma_\freeprod(\ThePct)}. 
\]

In the following, presentations of groups will be an important tool for us. In particular, presentations allow us to define group homomorphisms by assignments defined on generating sets. 

\begin{definition}
Let $G$ be a group with presentation 
\[
\langle X\mid R\rangle=\langle x_1,...,x_k\mid r_1=s_1,...,r_l=s_l\rangle 
\]
and $H$ a group generated by a set of elements $\{y_1,...,y_p\}$ with $p\geq k$. Moreover, let us assume that the words $r_j$ and $s_j$ are given by $x_{j_1}^{\varepsilon_1}...x_{j_q}^{\varepsilon_q}$ and $x_{\tilde{j}_1}^{\delta_1}...x_{\tilde{j}_r}^{\delta_r}$, respectively. Given \textit{assignments} $\phi:x_i\mapsto y_i$ for $1\leq i\leq k$, we apply them letterwise to words mapping $x_i^{-1}$ to $y_i^{-1}$. We say that the assignments $\phi$ \textit{preserve the relations in $R$} if the relation $y_{j_1}^{\varepsilon_1}...y_{j_q}^{\varepsilon_q}=y_{\tilde{j}_1}^{\delta_1}...y_{\tilde{j}_r}^{\delta_r}$ is valid in $H$ for each $1\leq j\leq l$. 
\end{definition}

\begin{theorem}[{von Dyck}, {\cite[p.\ 346]{Rotman2012}}]
\label{thm:von_Dyck}
Let $G$ be a group with presentation $\langle X\mid R\rangle$ as above and $H$ a group generated by a set $\{y_1,...,y_p\}$ with $p\geq k$. If the assignments 
\[
\phi:x_i\mapsto y_i 
\]
preserve the relations in $R$, then these assignments induce a homomorphism 
\[
\phi:G\rightarrow H. 
\]
\end{theorem}

\begin{lemma}[{\cite[Lemma 5.17]{Flechsig2023}}]
\label{lem:semidir_prod_pres}
Let $N$ and $H$ be groups given by presentations $N=\langle X\mid R\rangle$ and $H=\langle Y\mid S\rangle$. Then the following are equivalent: 
\begin{enumerate}
\item \label{lem:semidir_prod_pres_it1}
$G$ is a semidirect product with normal subgroup $N$ and quotient $H$. 
\item \label{lem:semidir_prod_pres_it2}
$G$ has a presentation 
\[
G=\langle X,Y\mid R,S,y^{\pm1}xy^{\mp1}=\phi_{y^{\pm1}}(x) \text{ for all } x\in X, y\in Y\rangle
\]
such that $\phi_{y^{\pm1}}(x)$ is a word in the alphabet~$X$ for all $x\in X$ and $y\in Y$. Moreover, for each $y\in Y$, the assignments 
\begin{equation}
\label{lem:semidir_prod_it2_cond_auto}
x\mapsto\phi_y(x)
\end{equation}
induce an automorphism $\phi_y\in\Aut(N)$ and the assignments 
\begin{equation}
\label{lem:semidir_prod_it2_cond_homo}
y\mapsto\phi_y 
\end{equation}
induce a homomorphism $H\rightarrow\Aut(N)$. 
\end{enumerate}
\end{lemma}

\begin{remark}
\label{rem:semidir_prod_pres}
Lemma \ref{lem:semidir_prod_pres} will be essential in the proof of Theorem \ref{thm-intro:kernel_ex_seq}. 
%
%
There we will apply it given a group $G$ with a presentation to deduce that $G$ is a semidirect product. In this case, we want to show that the given presentation satisfies the conditions from Lemma \ref{lem:semidir_prod_pres_it2}. 

The base to prove that is to divide the generating set into two disjoint subsets $X$ and $Y$ such that $X$ generates the 
normal subgroup and $Y$ generates the 
quotient. 

Further, we divide the relations into three disjoint subsets $R, S$ and $C$ such that~$R$ contains all relations in letters from $X$, $S$ contains all relations in letters from $Y$ and $C$ contains all the remaining relations. In particular, the relations in $C$ should be given in the form $y^{\pm1}xy^{\mp1}=\phi_{y^{\pm1}}(x)$ for all $x\in X$ and $y\in Y$. To deduce a semidirect product structure, it remains to check that the relations from $C$ satisfy the conditions on the assignments \eqref{lem:semidir_prod_it2_cond_auto} and \eqref{lem:semidir_prod_it2_cond_homo}. It is reasonable to check these conditions in the following order: 


\begin{step}
\label{rem:semidir_prod_pres_step1}
Using Theorem \ref{thm:von_Dyck}, the first step will be to check that the assignments $\phi:y\mapsto\phi_y$ from \eqref{lem:semidir_prod_it2_cond_homo} preserve the relations from $S$. If $S$ contains a relation $y_1^{\varepsilon_1}...y_q^{\varepsilon_q}=\tilde{y}_1^{\delta_1}...\tilde{y}_r^{\delta_r}$, this requires that the assignments
\begin{align*}
x & \mapsto\phi_{y_1^{\varepsilon_1}}\circ...\circ\phi_{y_q^{\varepsilon_q}}(x) \text{ and }
\\
x & \mapsto\phi_{\tilde{y}_1^{\delta_1}}\circ...\circ\phi_{\tilde{y}_r^{\delta_r}}(x)
\end{align*}
coincide on each letter $x\in X$ (up to relations in $R$). 

In particular, we may check if $\phi$ induces a homomorphism independently of the fact if $x\mapsto\phi_y(x)$ induces an automorphism of the group presented by $\langle X\mid R\rangle$. 
\end{step}

\begin{step}
\label{rem:semidir_prod_pres_step2}
In the second step, we check that the assignments $\phi_y:x\mapsto\phi_y(x)$ induce an automorphism of the group presented by $\langle X\mid R\rangle$ for all $y\in Y$. 

To apply Theorem \ref{thm:von_Dyck}, we check if for all $y\in Y$ the assignments $\phi_y$ preserve the relations from $R$. 
If this is the case, the assignments $\phi_y$ induce an endomorphism of the group presented by $\langle X\mid R\rangle$. By the first step, we further have 
\[
\phi_{y^{-1}}\circ\phi_y=\id_{\langle X\mid R\rangle}=\phi_y\circ\phi_{y^{-1}}, 
\]
i.e.\ the endomorphism induced by $x\mapsto\phi_y(x)$ is bijective and therefore an automorphism of the group presented by $\langle X\mid R\rangle$. 
\end{step}
\end{remark}

\new{Corollary \ref{cor:pure_orb_mcg_ses} and Lemma \ref{lem:semidir_prod_pres} together yield a presentation of $\PMapIdOrb{\TheStrand}{\Sigma_\freeprod(\ThePct)}$. For a description of this presentation, we need to establish a set of homeomorphisms that induces a generating set.} \enew{With respect to the embedding of $\FD(\ThePct)$ described in Figure \ref{fig:embedding_fund_domain}, let $D_{\Str,\Strand}\subseteq\FD(\ThePct)$ for every $1\leq\Str<\Strand\leq\TheStrand$ be the disk
\begin{align*}
& \left(B_{\frac{1}{4}}(p_\Str)\cup B_{\frac{1}{4}}(p_\Strand)\right)\cap\{x\in\CC\mid\textup{Im}(x)\geq0\}
\\
\cup & A_{\frac{\Strand-\Str}{2}-\frac{1}{4},\frac{\Strand-\Str}{2}+\frac{1}{4}}\left(\tfrac{p_\Str+p_\Strand}{2}\right)\cap\{x\in\CC\mid\textup{Im}(x)\leq0\}
\end{align*}
where $A_{r,R}(x)$ denotes the annulus with inner radius $r$ and outer radius $R$ centered around $x$. 
The disk $D_{\Str,\Strand}$ precisely contains the marked points $p_\Str$ and $p_\Strand$. See Figure~\ref{fig:disks_with_marked_pts} (left) for a picture of $D_{\Str,\Strand}$. 

Moreover, for every $1\leq\NStrand\leq\TheStrand$ and $1\leq\NPct\leq\ThePct$, let $D_{r_\NPct,\NStrand}\subseteq\FD(\ThePct)$ be the disk 
\begin{align*}
& \left(B_{\frac{1}{4}}(r_\NPct)\cup B_{\frac{1}{4}}(p_\NStrand)\right)\cap\{x\in\CC\mid\textup{Im}(x)\geq0\}
\\
\cup & A_{\frac{\NStrand+\NPct}{2}-\frac{1}{4},\frac{\NStrand+\NPct}{2}+\frac{1}{4}}\left(\tfrac{r_\NPct+p_\NStrand}{2}\right)\cap\{x\in\CC\mid\textup{Im}(x)\leq0\}. 
\end{align*}
The disk $D_{r_\NPct,\NStrand}$ precisely contains the marked point $r_\NPct$ and $p_\NStrand$. See Figure \ref{fig:disks_with_marked_pts} (right) for a picture of $D_{r_\NPct,\NStrand}$. 
\begin{figure}[H]
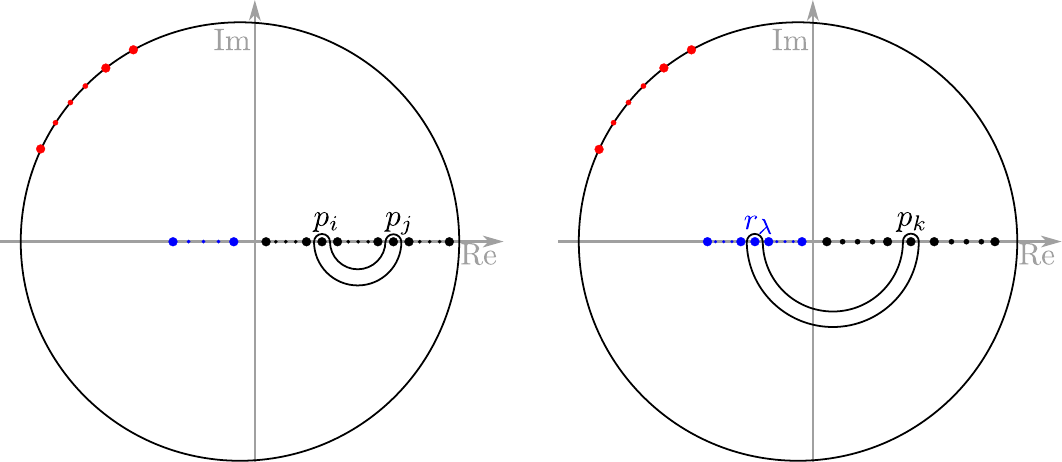
\caption{The disks $D_{\Str,\Strand}$ (left) and $D_{r_\NPct,\NStrand}$ (right).}
\label{fig:disks_with_marked_pts}
\end{figure}
}
The homeomorphisms $\Twist_{\Strand\Str}$ and $\TwistP_{\NStrand\NPct}$ perform the twists pictured in Figure \ref{fig:homeos_twist_marked_pts} on each $\freeprod$-translate of $D_{\Str,\Strand}$ and $D_{r_\NPct,\NStrand}$. 
\begin{figure}[H]
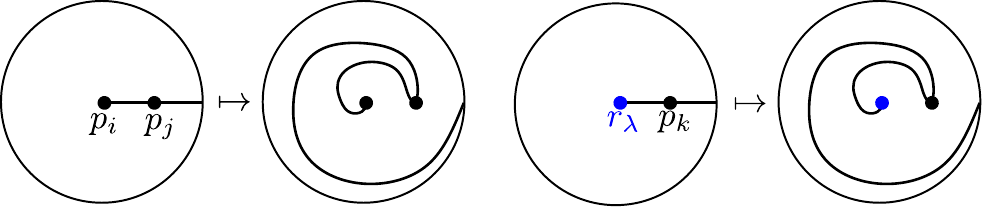
\caption{The twists induced by $\Twist_{\Strand\Str}$ (left) and $\TwistP_{\NStrand\NPct}$ (right).}
\label{fig:homeos_twist_marked_pts}
\end{figure}
\enew{Moreover, for every $1\leq\NStrand\leq\TheStrand$ and $1\leq\nu\leq\TheCone$, let $\tilde{D}_{\cp_\nu,\NStrand}$ be the disk 
\begin{align*}
& B_{\frac{1}{4}}(p_\NStrand)\cap\{x\in\CC\mid\textup{Im}(x)\geq0\}
\\
\cup & A_{\frac{\NStrand+\ThePct+\nu}{2}-\frac{1}{4},\frac{\NStrand+\ThePct+\nu}{2}+\frac{1}{4}}\left(\tfrac{-\ThePct-\nu+p_\NStrand}{2}\right)\cap\{x\in\CC\mid\textup{Im}(x)\leq0\}
\\
\cup & \left\lbrace x\in\FD\mid\textup{Im}(x)\geq0, \text{Re}(x)\in\left[-\ThePct-\nu-\tfrac{1}{4},-\ThePct-\nu+\tfrac{1}{4}\right]\right\rbrace. 
\end{align*}
Then $D_{\cp_\nu,\NStrand}:=\cyc{\TheOrder_\nu}\cdot\tilde{D}_{\cp_\nu,\NStrand}$ is a $\cyc{\TheOrder_\nu}$-invariant disk that contains the cone point $\cp_\nu$ and the adjacent marked points $\ZZ_{\TheOrder_\nu}(p_\NStrand)$. See Figure \ref{fig:disk_with_marked_pts_cp} for a picture of $\tilde{D}_{\cp_\nu,\NStrand}$ (left) and an example of the disk $D_{\cp_\nu,\NStrand}\subseteq\Sigma(\ThePct)$ (right). 
\begin{figure}[H]
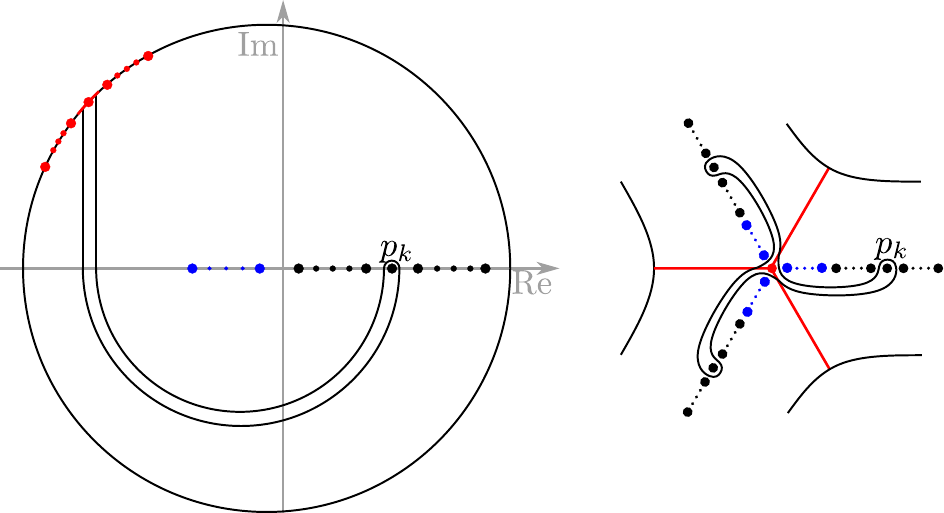
\caption{The disk $\tilde{D}_{\cp_\nu,\NStrand}\subseteq\FD$ (left) and $D_{\cp_\nu,\NStrand}\subseteq\Sigma(\ThePct)$ for a cone point of order three (right).}
\label{fig:disk_with_marked_pts_cp}
\end{figure}
}

\new{Let $\TwistC_{\NStrand\nu}$ be the homeomorphism that performs a $\frac{2\pi}{\TheOrder_\nu}$-twist as in Figure \ref{fig:homeo_twist_marked_pts_cp_clockwise} 
on each $\freeprod$-translate of $D_{\cp_\nu,\NStrand}$.}
\new{For the homeomorphisms $\Twist_{\Strand\Str}, \TwistP_{\NStrand\NPct}$ and $\TwistC_{\NStrand\nu}$, we will use their names as acronyms of the corresponding mapping classes. These elements satisfy the following relations: 

\begin{lemma}[{\cite[Lemma 4.18]{Flechsig2023mcg}}]
\label{lem:PMap_gens_sat_rels}
Let $1\leq\Strr,\Str,\Strand,\NStrand,\NNStrand<\TheStrand$ with $\Strr<\Str<\Strand<\NStrand<\NNStrand$, $1\leq\Pc,\NPct\leq\ThePct$ with $\Pc<\NPct$ and $1\leq\mu,\nu\leq\TheCone$ with $\mu<\nu$. Then the following relations hold: 
\begin{enumerate}
\item 
\label{lem:PMap_gens_sat_rels_it1}
\begin{enumerate}
\item \label{lem:PMap_gens_sat_rels_it1a}
$\Twist_{\NNStrand\Strand}\Twist_{\TheStrand\Strand}\Twist_{\NNStrand\Strand}^{-1}= \Twist_{\TheStrand\Strand}^{-1}\Twist_{\TheStrand\NNStrand}^{-1}\Twist_{\TheStrand\Strand}\Twist_{\TheStrand\NNStrand}\Twist_{\TheStrand\Strand}\; \text{ and } \; \hypertarget{lem:PMap_gens_sat_rels_it1a_2}{\textup{a')}}\;\Twist_{\NNStrand\Strand}^{-1}\Twist_{\TheStrand\Strand}\Twist_{\NNStrand\Strand}=\Twist_{\TheStrand\NNStrand}\Twist_{\TheStrand\Strand}\Twist_{\TheStrand\NNStrand}^{-1}$, 
\item \label{lem:PMap_gens_sat_rels_it1b}
$\Twist_{\Strand\Str}\Twist_{\TheStrand\Strand}\Twist_{\Strand\Str}^{-1}=\Twist_{\TheStrand\Str}^{-1}\Twist_{\TheStrand\Strand}\Twist_{\TheStrand\Str}\; \text{ and } \; \hypertarget{lem:PMap_gens_sat_rels_it1b_2}{\textup{b')}}\;\Twist_{\Strand\Str}^{-1}\Twist_{\TheStrand\Strand}\Twist_{\Strand\Str}=\Twist_{\TheStrand\Strand}\Twist_{\TheStrand\Str}\Twist_{\TheStrand\Strand}\Twist_{\TheStrand\Str}^{-1}\Twist_{\TheStrand\Strand}^{-1}$, 
\item \label{lem:PMap_gens_sat_rels_it1c}
$\TwistP_{\Strand\NPct}\Twist_{\TheStrand\Strand}\TwistP_{\Strand\NPct}^{-1}=\TwistP_{\TheStrand\NPct}^{-1}\Twist_{\TheStrand\Strand}\TwistP_{\TheStrand\NPct} \; \text{ and } \;\hypertarget{lem:PMap_gens_sat_rels_it1c_2}{\textup{c')}}\;\TwistP_{\Strand\NPct}^{-1}\Twist_{\TheStrand\Strand}\TwistP_{\Strand\NPct}=\Twist_{\TheStrand\Strand}\TwistP_{\TheStrand\NPct}\Twist_{\TheStrand\Strand}\TwistP_{\TheStrand\NPct}^{-1}\Twist_{\TheStrand\Strand}^{-1}$, 
\item \label{lem:PMap_gens_sat_rels_it1d}
$\TwistC_{\Strand\nu}\Twist_{\TheStrand\Strand}\TwistC_{\Strand\nu}^{-1}=\TwistC_{\TheStrand\nu}^{-1}\Twist_{\TheStrand\Strand}\TwistC_{\TheStrand\nu}\; \text{ and } \; \hypertarget{lem:PMap_gens_sat_rels_it1d_2}{\textup{d')}}\;\TwistC_{\Strand\nu}^{-1}\Twist_{\TheStrand\Strand}\TwistC_{\Strand\nu}=\Twist_{\TheStrand\Strand}\TwistC_{\TheStrand\nu}\Twist_{\TheStrand\Strand}\TwistC_{\TheStrand\nu}^{-1}\Twist_{\TheStrand\Strand}^{-1}$, 
\item \label{lem:PMap_gens_sat_rels_it1e}
$\TwistP_{\Strand\NPct}\TwistP_{\TheStrand\NPct}\TwistP_{\Strand\NPct}^{-1}=\TwistP_{\TheStrand\NPct}^{-1}\Twist_{\TheStrand\Strand}^{-1}\TwistP_{\TheStrand\NPct}\Twist_{\TheStrand\Strand}\TwistP_{\TheStrand\NPct}\; \text{ and } \; \hypertarget{lem:PMap_gens_sat_rels_it1e_2}{\textup{e')}}\;\TwistP_{\Strand\NPct}^{-1}\TwistP_{\TheStrand\NPct}\TwistP_{\Strand\NPct}=\Twist_{\TheStrand\Strand}\TwistP_{\TheStrand\NPct}\Twist_{\TheStrand\Strand}^{-1}$, 
\item \label{lem:PMap_gens_sat_rels_it1f}
$\TwistC_{\Strand\nu}\TwistC_{\TheStrand\nu}\TwistC_{\Strand\nu}^{-1}=\TwistC_{\TheStrand\nu}^{-1}\Twist_{\TheStrand\Strand}^{-1}\TwistC_{\TheStrand\nu}\Twist_{\TheStrand\Strand}\TwistC_{\TheStrand\nu}\; \text{ and } \; \hypertarget{lem:PMap_gens_sat_rels_it1f_2}{\textup{f')}}\;\TwistC_{\Strand\nu}^{-1}\TwistC_{\TheStrand\nu}\TwistC_{\Strand\nu}=\Twist_{\TheStrand\Strand}\TwistC_{\TheStrand\nu}\Twist_{\TheStrand\Strand}^{-1}$,
\end{enumerate} 
\item 
\label{lem:PMap_gens_sat_rels_it2}
\begin{enumerate}
\item \label{lem:PMap_gens_sat_rels_it2a}
$[\Twist_{\Str\Strr},\Twist_{\TheStrand\Strand}]=1$, \; $[\TwistP_{\Str\NPct},\Twist_{\TheStrand\Strand}]=1 \; \text{ and } \; [\TwistC_{\Str\nu},\Twist_{\TheStrand\Strand}]=1$, 
\item \label{lem:PMap_gens_sat_rels_it2b}
$[\Twist_{\NNStrand\NStrand},\Twist_{\TheStrand\Strand}]=1$, \; $[\Twist_{\Strand\Str},\TwistP_{\TheStrand\NPct}]=1$, \;
$[\TwistP_{\Strand\Pc},\TwistP_{\TheStrand\NPct}]=1$, 
\\
$[\Twist_{\Strand\Str},\TwistC_{\TheStrand\nu}]=1$, \; 
$[\TwistP_{\Strand\NPct},\TwistC_{\TheStrand\nu}]=1 \; \text{ and } \; [\TwistC_{\Strand\mu},\TwistC_{\TheStrand\nu}]=1$, 
\item \label{lem:PMap_gens_sat_rels_it2c}
$[\Twist_{\TheStrand\NNStrand}\Twist_{\TheStrand\Strand}\Twist_{\TheStrand\NNStrand}^{-1},\Twist_{\NNStrand\Str}]=1$, \; $[\Twist_{\TheStrand\NNStrand}\Twist_{\TheStrand\Strand}\Twist_{\TheStrand\NNStrand}^{-1},\TwistP_{\NNStrand\NPct}]=1 \;$ \text{ and } 
\\ 
$[\Twist_{\TheStrand\NNStrand}\Twist_{\TheStrand\Strand}\Twist_{\TheStrand\NNStrand}^{-1},\TwistC_{\NNStrand\nu}]=1$, 
\item \label{lem:PMap_gens_sat_rels_it2d} 
$[\Twist_{\TheStrand\Strand}\TwistP_{\TheStrand\Pc}\Twist_{\TheStrand\Strand}^{-1},\TwistP_{\Strand\NPct}]=1 \; \text{ and } \; [\Twist_{\TheStrand\Strand}\TwistP_{\TheStrand\Pc}\Twist_{\TheStrand\Strand}^{-1},\TwistC_{\Strand\nu}]=1$, 
\item \label{lem:PMap_gens_sat_rels_it2e}
$[\Twist_{\TheStrand\Strand}\TwistC_{\TheStrand\mu}\Twist_{\TheStrand\Strand}^{-1},\TwistC_{\Strand\nu}]=1$. 
\end{enumerate}
\end{enumerate}
\newpage
In particular, these relations imply: 
\begin{enumerate}[resume]
\item \label{lem:PMap_gens_sat_rels_it3}
\begin{enumerate}
\item \label{lem:PMap_gens_sat_rels_it3a}
$\Twist_{\NNStrand\Str}\Twist_{\TheStrand\Strand}\Twist_{\NNStrand\Str}^{-1} = \Twist_{\TheStrand\Str}^{-1}\Twist_{\TheStrand\NNStrand}^{-1}\Twist_{\TheStrand\Str}\Twist_{\TheStrand\NNStrand}\Twist_{\TheStrand\Strand}\Twist_{\TheStrand\NNStrand}^{-1}\Twist_{\TheStrand\Str}^{-1}\Twist_{\TheStrand\NNStrand}\Twist_{\TheStrand\Str}$, 
\item[\textup{a')}] \label{lem:PMap_gens_sat_rels_it3a_2}
$\Twist_{\NNStrand\Str}^{-1}\Twist_{\TheStrand\Strand}\Twist_{\NNStrand\Str} = \Twist_{\TheStrand\NNStrand}\Twist_{\TheStrand\Str}\Twist_{\TheStrand\NNStrand}^{-1}\Twist_{\TheStrand\Str}^{-1}\Twist_{\TheStrand\Strand}\Twist_{\TheStrand\Str}\Twist_{\TheStrand\NNStrand}\Twist_{\TheStrand\Str}^{-1}\Twist_{\TheStrand\NNStrand}^{-1}$, 
\item \label{lem:PMap_gens_sat_rels_it3b}
$\TwistP_{\NNStrand\NPct}\Twist_{\TheStrand\Strand}\TwistP_{\NNStrand\NPct}^{-1} = \TwistP_{\TheStrand\NPct}^{-1}\Twist_{\TheStrand\NNStrand}^{-1}\TwistP_{\TheStrand\NPct}\Twist_{\TheStrand\NNStrand}\Twist_{\TheStrand\Strand}\Twist_{\TheStrand\NNStrand}^{-1}\TwistP_{\TheStrand\NPct}^{-1}\Twist_{\TheStrand\NNStrand}\TwistP_{\TheStrand\NPct}$, 
\item[\textup{b')}] \label{lem:PMap_gens_sat_rels_it3b_2}
$\TwistP_{\NNStrand\NPct}^{-1}\Twist_{\TheStrand\Strand}\TwistP_{\NNStrand\NPct} = \Twist_{\TheStrand\NNStrand}\TwistP_{\TheStrand\NPct}\Twist_{\TheStrand\NNStrand}^{-1}\TwistP_{\TheStrand\NPct}^{-1}\Twist_{\TheStrand\Strand}\TwistP_{\TheStrand\NPct}\Twist_{\TheStrand\NNStrand}\TwistP_{\TheStrand\NPct}^{-1}\Twist_{\TheStrand\NNStrand}^{-1}$, 
\item \label{lem:PMap_gens_sat_rels_it3c}
$\TwistC_{\NNStrand\nu}\Twist_{\TheStrand\Strand}\TwistC_{\NNStrand\nu}^{-1} = \TwistC_{\TheStrand\nu}^{-1}\Twist_{\TheStrand\NNStrand}^{-1}\TwistC_{\TheStrand\nu}\Twist_{\TheStrand\NNStrand}\Twist_{\TheStrand\Strand}\Twist_{\TheStrand\NNStrand}^{-1}\TwistC_{\TheStrand\nu}^{-1}\Twist_{\TheStrand\NNStrand}\TwistC_{\TheStrand\nu}$, 
\item[\textup{c')}] \label{lem:PMap_gens_sat_rels_it3c_2}
$\TwistC_{\NNStrand\nu}^{-1}\Twist_{\TheStrand\Strand}\TwistC_{\NNStrand\nu} = \Twist_{\TheStrand\NNStrand}\TwistC_{\TheStrand\nu}\Twist_{\TheStrand\NNStrand}^{-1}\TwistC_{\TheStrand\nu}^{-1}\Twist_{\TheStrand\Strand}\TwistC_{\TheStrand\nu}\Twist_{\TheStrand\NNStrand}\TwistC_{\TheStrand\nu}^{-1}\Twist_{\TheStrand\NNStrand}^{-1}$, 
\item \label{lem:PMap_gens_sat_rels_it3d}
$\TwistP_{\Strand\NPct}\TwistP_{\TheStrand\Pc}\TwistP_{\Strand\NPct}^{-1} = \TwistP_{\TheStrand\NPct}^{-1}\Twist_{\TheStrand\Strand}^{-1}\TwistP_{\TheStrand\NPct}\Twist_{\TheStrand\Strand}\TwistP_{\TheStrand\Pc}\Twist_{\TheStrand\Strand}^{-1}\TwistP_{\TheStrand\NPct}^{-1}\Twist_{\TheStrand\Strand}\TwistP_{\TheStrand\NPct}$, 
\item[\textup{d')}] \label{lem:PMap_gens_sat_rels_it3d_2}
$\TwistP_{\Strand\NPct}^{-1}\TwistP_{\TheStrand\Pc}\TwistP_{\Strand\NPct} = \Twist_{\TheStrand\Strand}\TwistP_{\TheStrand\NPct}\Twist_{\TheStrand\Strand}^{-1}\TwistP_{\TheStrand\NPct}^{-1}\TwistP_{\TheStrand\Pc}\TwistP_{\TheStrand\NPct}\Twist_{\TheStrand\Strand}\TwistP_{\TheStrand\NPct}^{-1}\Twist_{\TheStrand\Strand}^{-1}$, 
\item \label{lem:PMap_gens_sat_rels_it3e}
$\TwistC_{\Strand\nu}\TwistP_{\TheStrand\NPct}\TwistC_{\Strand\nu}^{-1} = \TwistC_{\TheStrand\nu}^{-1}\Twist_{\TheStrand\Strand}^{-1}\TwistC_{\TheStrand\nu}\Twist_{\TheStrand\Strand}\TwistP_{\TheStrand\NPct}\Twist_{\TheStrand\Strand}^{-1}\TwistC_{\TheStrand\nu}^{-1}\Twist_{\TheStrand\Strand}\TwistC_{\TheStrand\nu}$, 
\item[\textup{e')}] \label{lem:PMap_gens_sat_rels_it3e_2}
$\TwistC_{\Strand\nu}^{-1}\TwistP_{\TheStrand\NPct}\TwistC_{\Strand\nu} = \Twist_{\TheStrand\Strand}\TwistC_{\TheStrand\nu}\Twist_{\TheStrand\Strand}^{-1}\TwistC_{\TheStrand\nu}^{-1}\TwistP_{\TheStrand\NPct}\TwistC_{\TheStrand\nu}\Twist_{\TheStrand\Strand}\TwistC_{\TheStrand\nu}^{-1}\Twist_{\TheStrand\Strand}^{-1}$,  
\item \label{lem:PMap_gens_sat_rels_it3f}
$\TwistC_{\Strand\nu}\TwistC_{\TheStrand\mu}\TwistC_{\Strand\nu}^{-1} = \TwistC_{\TheStrand\nu}^{-1}\Twist_{\TheStrand\Strand}^{-1}\TwistC_{\TheStrand\nu}\Twist_{\TheStrand\Strand}\TwistC_{\TheStrand\mu}\Twist_{\TheStrand\Strand}^{-1}\TwistC_{\TheStrand\nu}^{-1}\Twist_{\TheStrand\Strand}\TwistC_{\TheStrand\nu}$, 
\item[\textup{f')}] \label{lem:PMap_gens_sat_rels_it3f_2}
$\TwistC_{\Strand\nu}^{-1}\TwistC_{\TheStrand\mu}\TwistC_{\Strand\nu} = \Twist_{\TheStrand\Strand}\TwistC_{\TheStrand\nu}\Twist_{\TheStrand\Strand}^{-1}\TwistC_{\TheStrand\nu}^{-1}\TwistC_{\TheStrand\mu}\TwistC_{\TheStrand\nu}\Twist_{\TheStrand\Strand}\TwistC_{\TheStrand\nu}^{-1}\Twist_{\TheStrand\Strand}^{-1}$. 
\end{enumerate}
\end{enumerate}
\end{lemma}

Now Corollary \ref{cor:pure_orb_mcg_ses} together with  Lemmas \ref{lem:semidir_prod_pres} and \ref{lem:PMap_gens_sat_rels} induce the following presentation of $\PMapIdOrb{\TheStrand}{\Sigma_\freeprod(\ThePct)}$:} 

\begin{corollary}[{\cite[Corollary 4.19]{Flechsig2023mcg}}]
\label{cor:pres_PMap_free_prod}
The pure mapping class group $\PMapIdOrb{\TheStrand}{\Sigma_\freeprod(\ThePct)}$ has a presentation with generators 
\[
\Twist_{\Strand\Str}, \TwistP_{\NStrand\NPct} \; \text{ and } \; \TwistC_{\NStrand\nu}, 
\]
for $1\leq\Str,\Strand,\NStrand\leq\TheStrand$ with $\Str<\Strand$, $1\leq\NPct\leq\ThePct$ and $1\leq\nu\leq\TheCone$ and the following defining relations for $1\leq\Str,\Strand,\NStrand,\NNStrand\leq\TheStrand$ with $\Str<\Strand<\NStrand<\NNStrand$, $1\leq\Pc,\NPct\leq\ThePct$ with $\Pc<\NPct$ and $1\leq\mu,\nu\leq\TheCone$ with $\mu<\nu$: 
\begin{enumerate}
\item 
\label{cor:pres_PMap_free_prod_rel1}
$[\Twist_{\Strand\Str},\Twist_{\NNStrand\NStrand}]=1$, 
$[\TwistP_{\Strand\NPct},\Twist_{\NNStrand\NStrand}]=1 \; \text{ and } \; 
[\TwistC_{\Strand\nu},\Twist_{\NNStrand\NStrand}]=1$, 
\item 
\label{cor:pres_PMap_free_prod_rel2}
$[\Twist_{\NNStrand\Str},\Twist_{\NStrand\Strand}]=1$, 
$[\TwistP_{\NNStrand\NPct},\Twist_{\NStrand\Strand}]=1$, 
$[\TwistP_{\NNStrand\NPct},\TwistP_{\NStrand\Pc}]=1$, 
$[\TwistC_{\NNStrand\nu},\Twist_{\NStrand\Strand}]=1$, 
\\
$[\TwistC_{\NNStrand\nu},\TwistP_{\NStrand\NPct}]=1 \; \text{ and } \; 
[\TwistC_{\NNStrand\nu},\TwistC_{\NStrand\mu}]=1$, 
\item 
\label{cor:pres_PMap_free_prod_rel3}
$[\Twist_{\NNStrand\NStrand}\Twist_{\NNStrand\Strand}\Twist_{\NNStrand\NStrand}^{-1},\Twist_{\NStrand\Str}]=1$, 
$[\Twist_{\NStrand\Strand}\Twist_{\NStrand\Str}\Twist_{\NStrand\Strand}^{-1},\TwistP_{\Strand\NPct}]=1$, 
$[\Twist_{\NStrand\Strand}\TwistP_{\NStrand\Pc}\Twist_{\NStrand\Strand}^{-1},\TwistP_{\Strand\NPct}]=1$, 
$[\Twist_{\NStrand\Strand}\Twist_{\NStrand\Str}\Twist_{\NStrand\Strand}^{-1},\TwistC_{\Strand\nu}]=1$, 
$[\Twist_{\NStrand\Strand}\TwistC_{\NStrand\mu}\Twist_{\NStrand\Strand}^{-1},\TwistC_{\Strand\nu}]=1 \; \text{ and } \; 
[\Twist_{\NStrand\Strand}\TwistP_{\NStrand\NPct}\Twist_{\NStrand\Strand}^{-1},\TwistC_{\Strand\nu}]=~1$, 
\item 
\label{cor:pres_PMap_free_prod_rel4}
$\Twist_{\Strand\Str}\Twist_{\NStrand\Strand}\Twist_{\NStrand\Str}=\Twist_{\NStrand\Str}\Twist_{\Strand\Str}\Twist_{\NStrand\Strand}=\Twist_{\NStrand\Strand}\Twist_{\NStrand\Str}\Twist_{\Strand\Str}$, 
\\
$\Twist_{\Strand\Str}\TwistP_{\Strand\NPct}\TwistP_{\Str\NPct}=\TwistP_{\Str\NPct}\Twist_{\Strand\Str}\TwistP_{\Strand\NPct}=\TwistP_{\Strand\NPct}\TwistP_{\Str\NPct}\Twist_{\Strand\Str}$ and 
\\
$\Twist_{\Strand\Str}\TwistC_{\Strand\nu}\TwistC_{\Str\nu}=\TwistC_{\Str\nu}\Twist_{\Strand\Str}\TwistC_{\Strand\nu}=\TwistC_{\Strand\nu}\TwistC_{\Str\nu}\Twist_{\Strand\Str}$. 
\end{enumerate}
\end{corollary}

\new{Furthermore, we consider the group $\MapIdOrb{\TheStrand}{\Sigma_\freeprod(\ThePct)}$. Besides the elements mentioned above, this group contains elements represented by the following homeo-morphisms. Let $H_\Strand$ for $1\leq\Strand<\TheStrand$ be the homeomorphism that performs the following half-twist on each $\freeprod$-translate of the disk $D_{\Strand,\Strand+1}$, i.e.\ the disk $D_{\Str,\Strand+1}$ from Figure \ref{fig:disks_with_marked_pts} with $\Str=\Strand$:}  
\begin{figure}[H]
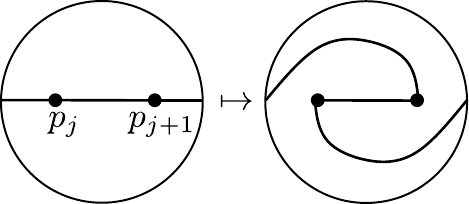
\caption{The half-twist $H_\Strand$.}
\label{fig:homeo_half-twist}
\end{figure}

\new{For every $1\leq\NPct\leq\ThePct$ and $1\leq\nu\leq\TheCone$, let $\TwP{\NPct}:=\TwistP_{1\NPct}$ and $\TwC{\nu}:=\TwistC_{1\nu}$. As for the pure generators, we will use the names $H_\Strand,\TwP{\NPct}$ and $\TwC{\nu}$ as acronyms for the represented mapping classes. 

The groups $\PMapIdOrb{\TheStrand}{\Sigma_\freeprod(\ThePct)}$ and $\MapIdOrb{\TheStrand}{\Sigma_\freeprod(\ThePct)}$ are related by the short exact sequence 
\[
1\rightarrow\PMapIdOrb{\TheStrand}{\Sigma_\freeprod(\ThePct)}\rightarrow\MapIdOrb{\TheStrand}{\Sigma_\freeprod(\ThePct)}\rightarrow\Sym_\TheStrand\rightarrow1.  
\]
This yields the following presentation of $\MapIdOrb{\TheStrand}{\Sigma_\freeprod(\ThePct)}$:} 

\begin{proposition}[{\cite[Proposition 4.22]{Flechsig2023mcg}}]
\label{prop:pres_map_kcp}
\new{For $\TheStrand\geq1$,} the group $\MapIdOrb{\TheStrand}{\Sigma_\freeprod(\ThePct)}$ is presented by generators
\[
H_1,...,H_{\TheStrand-1},\TwP{1},...,\TwP{\ThePct},\TwC{1},...,\TwC{\TheCone}
\]
and defining relations for $2\leq\Strand<\TheStrand$, $1\leq\Pc,\NPct\leq\ThePct$ with $\Pc<\NPct$ and $1\leq\mu,\nu\leq\TheCone$ with $\mu<\nu$: 
\begin{enumerate}
\item \label{prop:pres_map_kcp_rel1}
braid and commutator relations for the generators $H_1,...,H_{\TheStrand-1}$, 
\item \label{prop:pres_map_kcp_rel2}
\begin{enumerate}
\item \label{prop:pres_map_kcp_rel2a}
$[\TwP{\NPct},H_\Strand]=1$, 
\item \label{prop:pres_map_kcp_rel2b}$[\TwC{\nu},H_\Strand]=1$, 
\end{enumerate}
\item \label{prop:pres_map_kcp_rel3}
\begin{enumerate}
\item \label{prop:pres_map_kcp_rel3a}
$[H_1\TwP{\NPct}H_1,\TwP{\NPct}]=1$, 
\item \label{prop:pres_map_kcp_rel3b} 
$[H_1\TwC{\nu} H_1,\TwC{\nu}]=1$ and 
\end{enumerate}
\item \label{prop:pres_map_kcp_rel4}
\begin{enumerate}
\item \label{prop:pres_map_kcp_rel4a} 
$[\TwP{\Pc},\TwistP_{2\NPct}]=1$ for $\TwistP_{2\NPct}=H_1^{-1}\TwP{\NPct}H_1$, 
\item \label{prop:pres_map_kcp_rel4b} 
$[\TwC{\mu},\TwistC_{2\nu}]=1$ for $\TwistC_{2\nu}=H_1^{-1}\TwC{\nu}H_1$ and
\item \label{prop:pres_map_kcp_rel4c}
$[\TwP{\NPct},\TwistC_{2\nu}]=1$ for $\TwistC_{2\nu}=H_1^{-1}\TwC{\nu}H_1$. 
\end{enumerate}
\end{enumerate}
\end{proposition}

\new{Here above and in the following, we mean the relations $H_\Str H_{\Str+1}H_\Str=H_{\Str+1}H_\Str H_{\Str+1}$ for $1\leq\Str\leq\TheStrand-2$ and $[H_\Strand,H_\NStrand]=1$ for $1\leq\Strand,\NStrand<\TheStrand$ with $\vert\Strand-\NStrand\vert\geq2$ by \textit{braid and commutator relations} for $H_1,...,H_{\TheStrand-1}$.} 

\renewcommand{\Twist}{A}
\renewcommand{\TwistP}{B}
\renewcommand{\TwistC}{C}
\renewcommand{\twist}{a}
\renewcommand{\twistP}{b}
\renewcommand{\twistC}{c}

\section{Relating orbifold braid groups and orbifold mapping class groups}
\label{sec:braid_and_mcg}

This section highlights two fundamental differences between orbifold braid groups and Artin braid groups. \new{In the classical situation,} the generalized Birman exact sequence 
\[
1\rightarrow\underbrace{\pi_1(\Conf_\TheStrand(D))}_{=\B_\TheStrand}\xrightarrow{\Push_\TheStrand}\Map{\TheStrand}{D}\xrightarrow{\Forget_\TheStrand}\underbrace{\Map{}{D}}_{=1}\rightarrow1
\]
implies that the Artin braid group $\B_\TheStrand$ is isomorphic to $\Map{\TheStrand}{D}$, \new{see \cite[Theorem~9.1]{FarbMargalit2011} for details.} The inverse of the point-pushing map is the evaluation map $\ev:\Map{\TheStrand}{D}\rightarrow\B_\TheStrand$, which evaluates a certain ambient isotopy at the marked points. More precisely, given a self-homeomorphism $H$ of the disk $D$ that preserves the set of marked points $\{p_1,...,p_\TheStrand\}$ and fixes the boundary, the Alexander trick yields an ambient isotopy from $H$ to $\id_D$. Evaluated at the marked points, this ambient isotopy describes the strands of a braid $\ev([H])$. 

For orbifolds, we will establish an analogous map 
\[
\ev:\MapIdOrb{\TheStrand}{\Sigma_\freeprod(\ThePct)}\rightarrow\Z_\TheStrand(\Sigma_\freeprod(\ThePct)) 
\]
which, in contrast to the classical case, is not an isomorphism (see Theorem \ref{prop:orb_br_grp_quot}). 

This difference between the orbifold mapping class group and the orbifold braid group 
has fundamental consequences: 
recall that the pure subgroup $\PB_\TheStrand$ fits into a short exact sequence 
\[
1\rightarrow\underbrace{\pi_1(D(\TheStrand-1))}_{=\freegrp{\TheStrand-1}}\rightarrow\PB_\TheStrand\rightarrow\PB_{\TheStrand-1}\rightarrow1 
\]
that stems from the characterization of $\PB_\TheStrand$ as the pure mapping class group $\PMap{\TheStrand}{D}$ and a restriction of the generalized Birman exact sequence from \cite[Theorem~9.1]{FarbMargalit2011} to pure subgroups. For pure orbifold mapping class groups, we have a similar short exact sequence
\[
1\rightarrow\freegrp{\TheStrand-1+\ThePct}\rightarrow\PMapIdOrb{\TheStrand}{\Sigma_\freeprod(\ThePct)}\rightarrow\PMapIdOrb{\TheStrand-1}{\Sigma_\freeprod(\ThePct)}\rightarrow1
\] 
discussed in Corollary \ref{cor:pure_orb_mcg_ses}. In this section, we will consider similar maps for pure orbifold braid groups. We will show that we still have an exact sequence 
\[
\piOrb\left(\Sigma_\freeprod(\TheStrand-1+\ThePct)\right)\rightarrow\PZ_\TheStrand(\Sigma_\freeprod(\ThePct))\rightarrow\PZ_{\TheStrand-1}(\Sigma_\freeprod(\ThePct))\rightarrow1
\]
but the map $\piOrb(\Sigma_\freeprod(\TheStrand-1+\ThePct))\rightarrow\PZ_\TheStrand(\Sigma_\freeprod(\ThePct))$ surprisingly has a non-trivial kernel~$\kernel_\TheStrand$ (see Corollary \ref{cor:pure_orb_braid_semidir_prod_with_K}). This corrects Theorem 2.14 in \cite{Roushon2021}. Moreover, this implies that the canonical homomorphism $\Z_\TheStrand(\Sigma_\freeprod(\ThePct))\rightarrow\Z_{\TheStrand+\ThePct}(\Sigma_\freeprod)$ that sends punctures to fixed strands has non-trivial kernel (see Proposition \ref{prop:emb_orb_braid_grps_not_inj}). This corrects Proposition 4.1 in \cite{Roushon2021a}. 

\subsection{The orbifold braid group is a quotient of the orbifold mapping class group}
\label{subsec:orb_braid_quot_orb_mcg}

Our first goal is to define the evaluation map  
\[
\ev:\MapIdOrb{\TheStrand}{\Sigma_\freeprod(\ThePct)}\rightarrow\Z_\TheStrand(\Sigma_\freeprod(\ThePct)). 
\]
Moreover, we want to consider the orbifold braid group $\Z_\TheStrand(\Sigma_\freeprod(\ThePct,\TheCone))$ on the orbi-fold $\Sigma_\freeprod(\ThePct,\TheCone)$ that is also punctured at the cone points. We also want to establish a similar evaluation map
\[
\ev^\ast:\MapIdOrb{\TheStrand}{\Sigma_\freeprod(\ThePct)}\rightarrow\Z_\TheStrand(\Sigma_\freeprod(\ThePct,\TheCone)). 
\]

Together with the epimorphism $f:\Z_\TheStrand(\Sigma_\freeprod(\ThePct,\TheCone))\rightarrow\Z_\TheStrand(\Sigma_\freeprod(\ThePct))$ from Corollary~\ref{cor:surject_Z_n_ast_to_Z_n}, 
the evaluation maps will fit into a commutative diagram
\begin{equation}
\label{eq:comm_diag_ev-maps}
\begin{tikzcd}
\MapIdOrb{\TheStrand}{\Sigma_\freeprod(\ThePct)} \arrow[r,"\ev"] \arrow[d,"\ev^\ast"] & \Z_\TheStrand(\Sigma_\freeprod(\ThePct)). 
\\
\Z_\TheStrand(\Sigma_\freeprod(\ThePct,\TheCone))  \arrow[ru,"f"] & 
\end{tikzcd}
\end{equation}

The idea of the evaluation maps is the following: Recall that each homeomorphism $H$ that represents a mapping class in $\MapIdOrb{\TheStrand}{\Sigma_\freeprod(\ThePct)}$ by Definition \ref{def:Map_orb(ThePct)} is ambient isotopic to $\id_\Sigma$ if we forget the marked points $p_1,...,p_\TheStrand$. Let $H_t$ be such an ambient isotopy that fixes 
$r_1,...,r_\ThePct$ pointwise. Then evaluating $H_t$ at $p_1,...,p_\TheStrand$ represents 
a braid in $\Z_\TheStrand(\Sigma_\freeprod(\ThePct))$. However, it is not clear that the braid $[H_t(p_1,...,p_\TheStrand)]$ does not depend on the choice of the representative $H$ and the ambient isotopy. 

\new{To address this problem,} we recall that due to Proposition \ref{prop:pres_map_kcp} $\MapIdOrb{\TheStrand}{\Sigma_\freeprod(\ThePct)}$ has a finite presentation in terms of generators 
\[
H_\Strand, \TwP{\NPct} \; \text{ and } \; \TwC{\nu} \; \text{ for } \; 1\leq\Strand<\TheStrand, \; 1\leq\NPct\leq\ThePct \; \text{ and } \; 1\leq\nu\leq\TheCone. 
\]
For each generator $H_\Strand, \TwP{\NPct}$ and $\TwC{\nu}$, there is an 
ambient isotopy to $\id_\Sigma$ that forgets the marked points and performs the Alexander trick on each supporting disk, which is centered at a marked point or a cone point \new{for $\TwP{\NPct}$ and $\TwC{\nu}$, respectively.} Evaluating these ambient isotopies at the marked points, defines $\ev$ on the generators.

To check that the evaluation maps induce homomorphisms, 
we recall from Theorem \ref{thm:gen_set_Z_n} and Corollary \ref{cor:gen_set_Z_n} that the groups $\Z_\TheStrand(\Sigma_\freeprod(\ThePct,\TheCone))$ and $\Z_\TheStrand(\Sigma_\freeprod(\ThePct))$ are generated by braids 
\[
h_\Strand, \twP{\NPct} \; \text{ and } \; \twC{\nu} \; \text{ for } \; 1\leq\Strand<\TheStrand, \; 1\leq\NPct\leq\ThePct \; \text{ and } \; 1\leq\nu\leq\TheCone. 
\]
The induced map sends $H_\Strand$ to $h_\Strand$, $\TwP{\NPct}$ to $\twP{\NPct}$ and $\TwC{\nu}$ to~$\twC{\nu}$. In particular, this observation uses that all twists and half-twists pictured in Figures \ref{fig:homeo_twist_marked_pts_cp_clockwise}, \ref{fig:homeos_twist_marked_pts} and \ref{fig:homeo_half-twist} were defined moving clockwise. Consequently, the ambient isotopy to the identity moves the marked points counterclockwise. This matches the definition of $h_\Strand,\twP{\NPct}$ and $\twC{\nu}$ on page \pageref{fig:h_j}. 
In analogy to Proposition \ref{prop:pres_map_kcp}, we observe the following relations for the generators of the orbifold braid groups: 

\begin{lemma}
\label{lem:gens_sat_rels_free_prod}
The generators $h_1,...,h_{\TheStrand-1},\twP{1},...,\twP{\ThePct},\twC{1},...,\twC{\TheCone}$ of $\Z_\TheStrand(\Sigma_\freeprod(\ThePct))$ satisfy the following relations for $2\leq\Strand<\TheStrand$, $1\leq\Pc,\NPct\leq\ThePct,\Pc<\NPct$ and $1\leq\mu,\nu\leq\TheCone,\mu<\nu$: 
\begin{enumerate}
\item \label{lem:gens_sat_rels_free_prod_rel1} 
$\twC{\nu}^{\TheOrder_\nu}=1$, 
\item \label{lem:gens_sat_rels_free_prod_rel2} 
braid and commutator relations for the generators $h_1,...,h_{\TheStrand-1}$, 
\item \label{lem:gens_sat_rels_free_prod_rel3} 
\begin{enumerate}
\item \label{lem:gens_sat_rels_free_prod_rel3a} 
$[\twP{\NPct},h_\Strand]=1$ and 
\;\hypertarget{lem:gens_sat_rels_free_prod_rel3b}{\textup{b)}} $[\twC{\nu},h_\Strand]=1$, 
\end{enumerate}
\item \label{lem:gens_sat_rels_free_prod_rel4} 
\begin{enumerate}
\item \label{lem:gens_sat_rels_free_prod_rel4a} 
$[h_1\twP{\NPct}h_1,\twP{\NPct}]=1\;$ and \; 
\hypertarget{lem:gens_sat_rels_free_prod_rel4b}{\textup{b)}} $[h_1\twC{\nu}h_1,\twC{\nu}]$, 
\end{enumerate}
\item \label{lem:gens_sat_rels_free_prod_rel5}
\begin{enumerate}
\item \label{lem:gens_sat_rels_free_prod_rel5a} 
$[\twP{\Pc},\twistP_{2\NPct}]=1$, \hypertarget{lem:gens_sat_rels_free_prod_rel5b}{\textup{b)}} $[\twC{\mu},\twistC_{2\nu}]=1\;$ and \;\hypertarget{lem:gens_sat_rels_free_prod_rel5c}{\textup{c)}} $[\twP{\NPct},\twistC_{2\nu}]=1$ 
\\
with $\twistP_{2\NPct}=h_1^{-1}\twP{\NPct}h_1$ and $\twistC_{2\nu}=h_1^{-1}\twC{\nu} h_1$. 
\end{enumerate}
\end{enumerate}
\new{With the exception of \textup{\ref{lem:gens_sat_rels_free_prod_rel1}},} the same relations hold in $\Z_\TheStrand(\Sigma_\freeprod(\ThePct,\TheCone))$. 
\end{lemma}
\begin{proof}
The braid and commutator relations for $h_1,...,h_{\TheStrand-1}$ follow as in the surface case directly from the braid diagrams. The remaining commutator relations also follow from the braid diagrams (see Figure \ref{fig:gens_satisfy_rels_free_prod_comm_rel} for the relations that involve twists around cone points). 

The relation \ref{lem:gens_sat_rels_free_prod_rel1} was observed in Remark \ref{rem:braid_fin_order}. This observation was based on the contractibility of a loop that contains a cone point (see Figure \ref{fig:gens_satisfy_rels_free_prod_t^m}). This argument does not hold once cone points are removed. In fact, we will later see that the relation $\twC{\nu}^{\TheOrder_\nu}=1$ does not hold in $\Z_\TheStrand(\Sigma_\freeprod(\ThePct,\TheCone))$. 
\end{proof}

\begin{figure}[H]
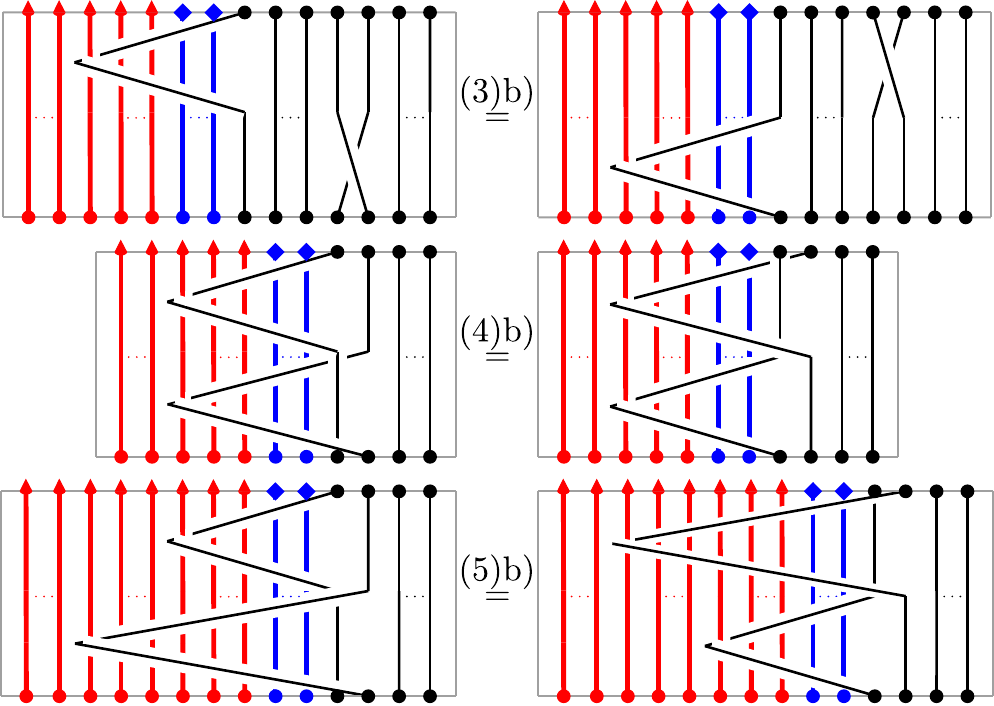
\caption{Observation of the relations \ref{lem:gens_sat_rels_free_prod_rel3}-\ref{lem:gens_sat_rels_free_prod_rel5} by consideration of orbifold braid diagrams.}
\label{fig:gens_satisfy_rels_free_prod_comm_rel}
\end{figure}

Lemma \ref{lem:gens_sat_rels_free_prod} in particular implies that the assignments $\ev$ and $\ev^\ast$ preserve the relations from Proposition \ref{prop:pres_map_kcp}. Hence, Theorem \ref{thm:von_Dyck} yields:  

\begin{corollary}
\label{cor:ev_maps_homo}
The maps $\ev^\ast$ and $\ev$ are homomorphisms that satisfy the condition from \eqref{eq:comm_diag_ev-maps}. 
\end{corollary}

Furthermore, using Lemma \ref{lem:fund_grp_free_act}, we may deduce: 

\begin{corollary}
\label{cor:ev-map_iso}
The map $\ev^\ast:\MapIdOrb{\TheStrand}{\Sigma_\freeprod(\ThePct)}\rightarrow\Z_\TheStrand(\Sigma_\freeprod(\ThePct,\TheCone))$ is an isomorphism. 
\end{corollary}
\newpage
\begin{proof}
The group $\Z_\TheStrand(\Sigma_\freeprod(\ThePct,\TheCone))$ is defined as the orbifold fundamental group of the orbifold configuration space 
\[
\Conf_\TheStrand^\freeprod(\Sigma_\freeprod(\ThePct,\TheCone))=\widetilde{\PConf}_\TheStrand^\freeprod(\Sigma(\ThePct,\TheCone))_{\freeprod^\TheStrand\rtimes\Sym_\TheStrand},  
\]
where $\widetilde{\PConf}_\TheStrand^\freeprod(\Sigma(\ThePct,\TheCone))$ denotes the underlying space 
\[
\left\lbrace(x_1,...,x_\TheStrand)\in\left(\Spct
\right)^\TheStrand\mid x_\Str\neq\gamma(\twistn{\Strand}) \text{ for all } 1\leq\Str,\Strand\leq\TheStrand,\Str\neq\Strand, \gamma\in\freeprod\right\rbrace. 
\]

To apply Lemma \ref{lem:fund_grp_free_act}, we want to show that the $\freeprod^\TheStrand\rtimes\Sym_\TheStrand$-action on $\widetilde{\PConf}_\TheStrand^\freeprod(\Sigma(\ThePct,\TheCone))$ is free. Therefore, we recall that the action of $((\gamma_1,...,\gamma_\TheStrand),\sigma)\in\freeprod^\TheStrand\rtimes\Sym_\TheStrand$ on $(x_1,...,x_\TheStrand)\in\widetilde{\PConf}_\TheStrand^\freeprod(\Sigma(\ThePct,\TheCone))$ is given by 
\[
((\gamma_1,...,\gamma_\TheStrand),\sigma)(x_1,...,x_\TheStrand)=(\gamma_{\sigma(1)}(x_{\sigma(1)}),...,\gamma_{\sigma(\TheStrand)}(x_{\sigma(\TheStrand)})). 
\]
Let us assume that $(x_1,...,x_\TheStrand)\in\widetilde{\PConf}_\TheStrand^\freeprod(\Sigma(\ThePct,\TheCone))$ is fixed by $((\gamma_1,...,\gamma_\TheStrand),\sigma)$ in $\freeprod^\TheStrand\rtimes\Sym_\TheStrand$. Since $x_\Str\neq\gamma(\twistn{\Strand})$ for all $\Str\neq\Strand$ and $\gamma\in\freeprod$, we obtain $\sigma=\id_\TheStrand$. Moreover, $\freeprod$ acts freely on $\Spct$ such that $\gamma_\Str=1$ for all $1\leq\Str\leq\TheStrand$, i.e.\ the action of $\freeprod^\TheStrand\rtimes S_\TheStrand$ is free. 

By Lemma \ref{lem:fund_grp_free_act}, this implies 
\[
\piOrb(\Conf_\TheStrand^\freeprod(\Sigma_\freeprod(\ThePct,\TheCone)))\cong\pi_1(\widetilde{\PConf}_\TheStrand^\freeprod(\Sigma(\ThePct,\TheCone))/\freeprod^\TheStrand\rtimes\Sym_\TheStrand). 
\]
Mapping $\left(\freeprod^\TheStrand\rtimes\Sym_\TheStrand\right)(x_1,...,x_\TheStrand)$ to $\Sym_\TheStrand(\freeprod(x_1),...,\freeprod(x_\TheStrand))$ defines a homeomorphism
\[
\widetilde{\PConf}_\TheStrand^\freeprod(\Sigma(\ThePct,\TheCone))/\freeprod^\TheStrand\rtimes\Sym_\TheStrand\rightarrow\PConf_\TheStrand(\Sigma(\ThePct,\TheCone)/\freeprod)/\Sym_\TheStrand. 
\]
Since $\Spct/\freeprod$ is homeomorphic to $D(\ThePct,\TheCone)$ and 
\[
\Conf_\TheStrand(D(\ThePct,\TheCone))=\PConf_\TheStrand(D(\ThePct,\TheCone))/\Sym_\TheStrand, 
\]
this implies that 
\[
\Z_\TheStrand(\Sigma_\freeprod(\ThePct,\TheCone))\cong\pi_1(\Conf_\TheStrand(D(\ThePct,\TheCone))). 
\]
From Theorem \ref{thm:Birman_es_orb} we further obtain an isomorphism between $\MapIdOrb{\TheStrand}{\Sigma_\freeprod(\ThePct)}$ and $\pi_1(\Conf_\TheStrand(D(\ThePct,\TheCone)))$ such that the following diagram commutes: 
\begin{center}
\begin{tikzcd}
\MapIdOrb{\TheStrand}{\Sigma_\freeprod(\ThePct)} \arrow[r,"\cong"] \arrow[d,"\ev^\ast"] & \pi_1(\Conf_\TheStrand(D(\ThePct,\TheCone))) \arrow[d,"\cong"]
\\
\Z_\TheStrand(\Sigma_\freeprod(\ThePct,\TheCone)) \arrow[r,equal] & \piOrb(\Conf_\TheStrand^\freeprod(\Sigma_\freeprod(\ThePct,\TheCone))). 
\end{tikzcd}
\end{center}
Thus, $\ev^\ast$ is an isomorphism. 
\end{proof}


The other homomorphism $\ev$ is no longer an isomorphism but we may determine its kernel. We will prove: 

\begin{proposition}
\label{prop:orb_br_grp_quot}
The kernel of $\ev$ is the normal closure of $\{\TwC{\nu}^{\TheOrder_\nu}\mid1\leq\nu\leq\TheCone\}$ in $\MapIdOrb{\TheStrand}{\Sigma_\freeprod(\ThePct)}$. The kernel of the restricted map $\ev\vert_{\PMapIdOrb{\TheStrand}{\Sigma_\freeprod(\ThePct)}}$ is the normal closure of $\{\TwistC_{\NStrand\nu}^{\TheOrder_\nu}\mid1\leq\nu\leq \TheCone, 1\leq\NStrand\leq\TheStrand\}$ in $\PMapIdOrb{\TheStrand}{\Sigma_\freeprod(\ThePct)}$. 
\end{proposition}

For the proof, we want to keep track of cone point intersections of homotopies. 
Similar to Definition \ref{def:braid_proj_cross} it is helpful to make assumptions on crossings; in this case crossings with cone points~$\cp_\nu$. Therefore, we recall the embedding of $\FD(\ThePct)$ as described in Figure~\ref{fig:embedding_fund_domain}. With respect to the projection 
\[
\pi:\FD(\ThePct)\times\bigcupdot_{\Subdiv=1}^\TheNSubdiv[t_{\Subdiv-1},t_\Subdiv]\rightarrow\RR\times\bigcupdot_{\Subdiv=1}^\TheNSubdiv[t_{\Subdiv-1},t_\Subdiv], 
\]
this embedding in particular satisfies $\pi((\cp_\nu,t))=(-\ThePct-\nu,t)$ for each $1\leq\nu\leq\TheCone$. 

\begin{definition}[Cone point crossings and generic representatives in $\Z_\TheStrand(\Sigma_\freeprod(\ThePct))$]
Let $\xi$ be a $\freeprod^\TheStrand\rtimes\Sym_\TheStrand$-path that represents an element in $\Z_\TheStrand(\Sigma_\freeprod(\ThePct))$. Assume that $\xi$ satisfies all the condition from the definition of a braid (see Definition \ref{def:braid_proj_cross}) except that it eventually intersects a cone point. 

Let $x\in\xi_\Strand$ such that $\pi(x)=(-\ThePct-\nu,t)$ for some $1\leq\nu\leq\TheCone$ and $t\in[t_{\Subdiv-1},t_\Subdiv]$. 

In this case, we say $\pi(\xi)$ \textit{crosses the cone point $\cp_\nu$} at height $t$. 
The cone point crossing is called \textit{transverse} if there is a neighborhood $U$ of $\pi(x)$ in $[0,\TheStrand+1]\times[t_{\Subdiv-1},t_\Subdiv]$ such that the pair 
\[
(U,(\pi(\xi)\cup\{-\ThePct-\nu\}\times I)\cap U)
\]
is locally homeomorphic to $(\RR^2,\RR\times\{0\}\cup\{0\}\times\RR)$ via a homeomorphism identifying $\pi(\xi_\Strand)$ with $\RR\times\{0\}$ and $\{-\ThePct-\nu\}\times I$ with $\{0\}\times\RR$. 

If all cone point crossings in $\xi$ are transverse, the representative $\xi$ is called \textit{generic}. 
\end{definition}

In particular, a generic $\freeprod^\TheStrand\rtimes\Sym_\TheStrand$-path does not stay in a cone point for a period of time. 

\begin{proof}[Proof of Proposition \textup{\ref{prop:orb_br_grp_quot}}]
If we recall that $\ev(\TwC{\nu})=\twC{\nu}$ for each $1\leq\nu\leq\TheCone$, the relation \ref{lem:gens_sat_rels_free_prod_rel1} implies that $\langle\langle\TwC{\nu}^{\TheOrder_\nu}\rangle\rangle_{\MapIdOrb{\TheStrand}{\Sigma_\freeprod(\ThePct)}}\subseteq\ker(\ev)$. Moreover, $\ev(\TwistC_{\NStrand\nu})=\twistC_{\NStrand\nu}$ and by \eqref{eq:def_a_kc_a_kr} $\twistC_{\NStrand\nu}$ is a $\Z_\TheStrand(\Sigma_\freeprod(\ThePct))$-conjugate of $\twC{\nu}$. Hence, $\twistC_{\NStrand\nu}^{\TheOrder_\nu}=1$ and 
\[
\langle\langle\TwistC_{\NStrand\nu}^{\TheOrder_\nu}\rangle\rangle_{\PMapIdOrb{\TheStrand}{\Sigma_\freeprod(\ThePct)}}\subseteq\ker(\ev\vert_{\PMapIdOrb{\TheStrand}{\Sigma_\freeprod(\ThePct)}}). 
\]

The opposite inclusions require more work. The idea is the following: Given an element $H\in\ker(\ev)$ in terms of the generators of the mapping class group $\MapIdOrb{\TheStrand}{\Sigma_\freeprod(\ThePct)}$, it projects to a braid that is homotopic to the trivial braid. If this homotopy intersects a cone point, a suitable adjustment of the homotopy allows us to read off the used orbifold Reidemeister moves discussed in Remark \ref{rem:braid_fin_order} and Figure \ref{fig:orb-Reidemeister-move}. In terms of the word that represents $H$, this transformation induces a non-trivial insertion or deletion of subwords conjugate to $\TwistC_{\NStrand\nu}^{\TheOrder_\nu}$. This allows us to deduce the claim. 

For example, the homotopy pictured in Figure \ref{fig:id_orb-Reidemeister-move} (left) reflects the orbifold Reidemeister move depicted on the right of the same figure. This Reidemeister move induces a deletion of $\TwC{\nu}^3$ in the word that represents the element $H\in\ker(\ev)$. 

\begin{figure}[H]
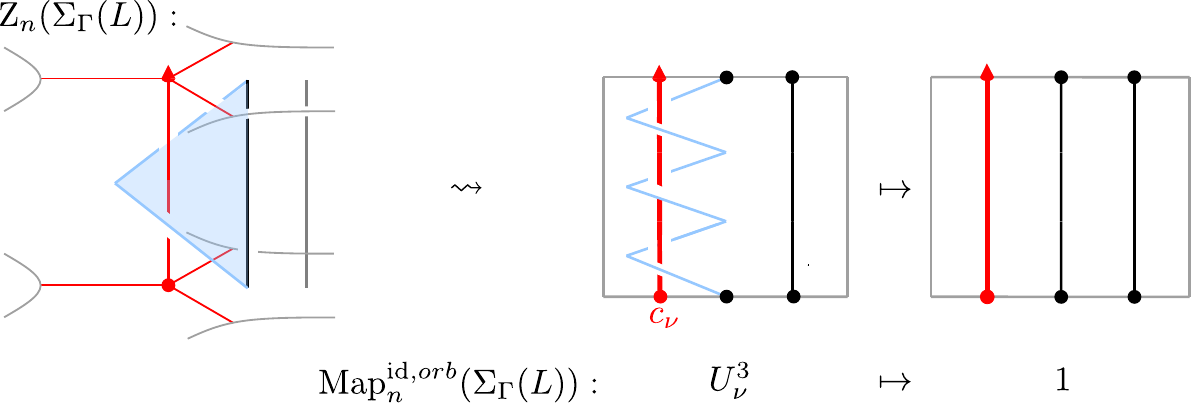
\caption{Identification of an orbifold Reidemeister move and induced deletion in $\MapIdOrb{\TheStrand}{\Sigma_\freeprod(\ThePct)}$.}
\label{fig:id_orb-Reidemeister-move}
\end{figure}

Let $H=\rho_1^{\varepsilon_1}...\rho_\TheNSubdiv^{\varepsilon_\TheNSubdiv}$ with 
\[
\rho_\Subdiv\in\{H_1,...,H_{\TheStrand-1},\TwP{1},...,\TwP{\ThePct},\TwC{1},...,\TwC{\TheCone}\} \quad \text{ and } \quad \varepsilon_\Subdiv\in\{\pm1\}
\]
be an element in $\ker(\ev)$. Then 
the braid $\ev(H)=b=\sigma_1^{\varepsilon_1}...\sigma_\TheNSubdiv^{\varepsilon_\TheNSubdiv}$ with $\sigma_\Subdiv=h_\Strand$ if $\rho_\Subdiv=H_\Strand$, $\sigma_\Subdiv=\twP{\NPct}$ if $\rho_\Subdiv=\TwP{\NPct}$ and $\sigma_\Subdiv=\twC{\nu}$ if $\rho_\Subdiv=\TwC{\nu}$ is trivial. 
Applying suitable shifts, we may assume that the strands $b_1,...,b_\TheStrand$ of $b$ 
are 
continuous. 
Since $b$ is trivial, 
we have a system $h_s$ of homotopies $h_s^{(\Strand)}$ connecting the strands $b_\Strand$ to the constant maps $I\rightarrow\Sigma,t\mapsto p_\Strand$ for $1\leq\Strand\leq\TheStrand$ such that the map $h_s:t\mapsto(h_s^{(1)}(t),...,h_s^{(\TheStrand)}(t))$ represents an element in $\Z_\TheStrand(\Sigma_\freeprod(\ThePct))$ for every $s\in I$. 

Since $b$ is piecewise linear, analogous to the classical case described in \cite[Claim~1.7]{KasselTuraev2008}, the homotopy $h_s$ can be realized as a sequence of $\Delta$-moves. For each $s\in I$, we may assume that each strand of $h_s$ intersects the set $\freeprod(\partial \FD(\ThePct))$ in finitely many points. Applying the same ideas as in the proof of Lemma \ref{it_simpl_paths_pw_red_fd}, this yields an equivalent representative $\bar{h}_s$ such that the image of each strand $\bar{h}_s^{(\Strand)}$ is contained in $\FD(\ThePct)$. Additionally, as for the classical case \cite[Claim 1.8]{KasselTuraev2008}, we may also adjust the $\Delta$-move such that the representatives $\bar{h}_s$ are generic for each $s\in I$. In particular, in such a sequence of $\Delta$-moves each strand $h_s^{(\Strand)}$ intersects only finitely many times with cone points. Hence, we may induct on the number of intersections. 
\begin{intermediate}[The base case] 
If the homotopy $h_s$ does not intersect any cone points, the diagram~\eqref{eq:comm_diag_ev-maps} implies that $H$ is also contained in $\ker(\ev^\ast)$. By Lemma \ref{cor:ev-map_iso}, we obtain $H=\id_{\Sigma(\ThePct)}$. 
\end{intermediate}
\begin{intermediate}[The case where the homotopy intersects cone points]
If the homotopy $h_s$ intersects the cone points at least once, 
there exists 
a triple $(s,t,\Strand)$ with $s,t\in I$ and $1\leq\Strand\leq\TheStrand$ for each cone point intersection. For such a triple, we have $h_s^{(\Strand)}(t)=\gamma(\cp_\nu)$ for some~$1\leq\nu\leq\TheCone$ and $\gamma\in\freeprod$, where $h_s^{(\Strand)}$ denotes the strand of $h_s$ that begins in $p_\Strand$. Let us fix such a triple $(s_0,t_0,\Strand_0)$. 
 
During the $\Delta$-move that contains the intersection $h_{s_0}^{(\Strand_0)}(t_0)=\gamma(\cp_\nu)$ the only moving strand is the $\Strand_0$-th one. Since $\bar{h}_s$ is a generic braid for each $s\in I$, we may assume that $\pi(\bar{h}_s)$ does not contain \new{an additional} cone point crossing at height $t_0$ by shifting any such cone point crossing to a different height (see Figure \ref{fig:kernel_Map_orb_shift_cp_crossing} for an example). 

At this point, some explaining words about the following Figures \ref{fig:kernel_Map_orb_shift_cp_crossing}, \ref{fig:kernel_Map_orb_braid_adjust_Delta} and~\ref{fig:kernel_Map_orb_braid_stretch} are in order: Even if these pictures are drawn diagrammatically, they should not be misinterpreted as orbifold braid diagrams in the sense of Definition~\ref{def:braid_proj_cross}, where strands move inside the fundamental domain $\FD(\ThePct)$. Instead, 
as in the more elaborated picture in Figure \ref{fig:id_orb-Reidemeister-move} (left), the arcs in these figures should be considered as strands of a $\freeprod^\TheStrand\rtimes\Sym_\TheStrand$-path in $\Z_\TheStrand(\Sigma_\freeprod(\ThePct))$ moving inside the neighborhood of a cone point. In these figures, the strands drawn in black indicate the braid $\bar{h}_s$ for some $s<s_0$ (resp. $s>s_0$). These braids are pictured as the strands of a braid diagram. Furthermore, the figures depict $\Delta$-moves that intersect a cone point. In contrast, to orbifold braid diagrams, after (resp. before) the application of such a $\Delta$-move, the moved strand (depicted in light blue) does not lie entirely in the fundamental domain. 

\begin{figure}[H]
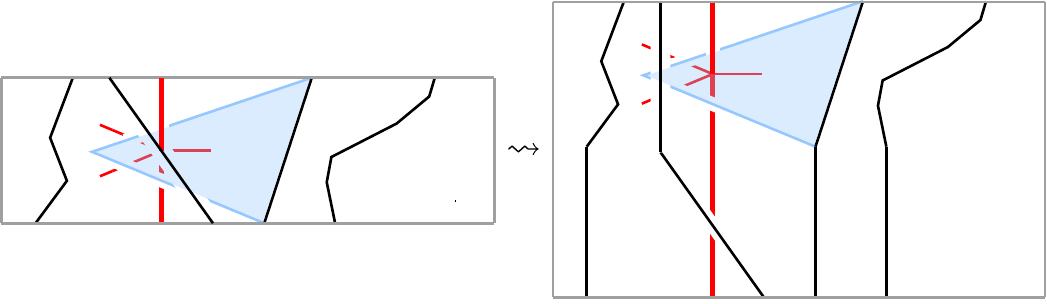
\caption{Shift of a cone point crossing.}
\label{fig:kernel_Map_orb_shift_cp_crossing}
\end{figure}

Let us say that the relevant $\Delta$-move has underlying triangle $T$. Since $\pi(\bar{h}_{s_0})$ crosses with $\cp_\nu$ at time $t_0$, a subdivision of the underlying triangle $T$ allows us to find a triangle $T'\subseteq T$ and $\delta_1,\delta_2,\delta_1',\delta_2'>0$ such that 
\begin{itemize}
\item $T'$ contains the cone point $\cp_\nu$ in its interior. 
\item $T'$ contains no point $p_\Strand$ for $1\leq\Strand\leq\TheStrand$ that is an end point of a strand. 
\item $h_s
:[s_0-\delta_1,s_0+\delta_2]\rightarrow\Sigma(\ThePct)$ describes the $\Delta$-move supported on $T'$. 
\item for every $s\in[s_0-\delta_1,s_0+\delta_2]$, 
no crossing in the sense of Definition \ref{def:braid_proj_cross} occurs in $\bar{h}_s\vert_{[t_0-\delta_1',t_0+\delta_2']}$. 
\end{itemize}
See Figure \ref{fig:kernel_Map_orb_braid_adjust_Delta} for an illustration. 
\begin{figure}[H]
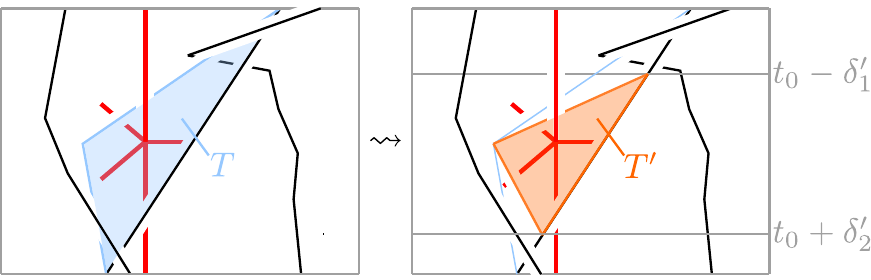
\caption{Adjustment of the $\Delta$-move to a smaller triangle.}
\label{fig:kernel_Map_orb_braid_adjust_Delta}
\end{figure}
Next, the idea is to adjust the $\Delta$-move supported on $T'$ such that it keeps all strands except the $\Strand_0$-th one fixed at positions $p_1,...,p_\TheStrand$. We divide this adjustment into three steps (see Figure \ref{fig:kernel_Map_orb_braid_stretch} for an illustrating example): 
\begin{enumerate}
\item First of all, for every $s\in[s_0-\delta_1,s_0+\delta_2]$, we reparametrize the $t$-component of the $\Delta$-move $h_s\vert_{[t_0-\delta_1',t_0+\delta_2']}$ shrinking the interval $[t_0-\delta_1',t_0+\delta_2']$ to an interval $[t_0-\varepsilon_1,t_0+\varepsilon_2]$ (see Figure \ref{fig:kernel_Map_orb_braid_stretch}, step (1)). 
\item Moreover, let $\delta_1'',\delta_2''>0$ such that 
\[
t_0-\delta_1''\in(t_0-\new{\delta_1'},t_0-\varepsilon_1) \quad \text{ and } \quad t_0+\delta_2''\in(t_0+\varepsilon_2,t_0+\new{\delta_2'}). 
\]
We want to adjust the endpoints of the homotopy $h_s\vert_{[t_0-\delta_1'',t_0+\delta_2'']}$ such that they are kept fixed. Therefore, we endow the strands with an order according to their real part in $\bar{h}_s(t_0-\delta_1'')$. Let us assume that the $\Strand_0$-th strand appears in the $\NStrand$-th position in this order. Now we adjust $h_s\vert_{[t_0-\delta_1',t_0-\varepsilon_1]}$ and $h_s\vert_{[t_0+\new{\varepsilon_2},t_0+\delta_2']}$ by suitable homotopies 
such that 
\begin{align*}
h_{s_0-\delta_1}(t_0-\delta_1'') & =(p_1,...,p_\TheStrand)=h_{s_0-\delta_1}(t_0+\delta_2'') \text{ and } 
\\
h_{s_0+\delta_2}(t_0-\delta_1'') & =(p_1,...,p_\TheStrand)=h_{s_0+\delta_2}(t_0+\delta_2''). 
\end{align*}
If we realize this adjustment with respect to the order at height $t_0-\delta_1''$, which means we homotope the $\Str$-th strand in this order into the position~$p_\Str$, we may assume that the adjustment does not create any crossings in $[t_0-\delta_1',t_0-\varepsilon_1]$ and $[t_0+\varepsilon_2,t_0+\delta_2']$ (see Figure \ref{fig:kernel_Map_orb_braid_stretch}, step (2)). 
\item 
Further, we may adjust the homotopy so that all strands except the $\Strand_0$-th are kept fixed during the homotopy $h_s\vert_{[t_0-\delta_1'',t_0+\delta_2'']}$, 
(see Figure \ref{fig:kernel_Map_orb_braid_stretch}, step~(3)). If we realize this adjustment moving all the strands in front of the $\Strand_0$-th one, we obtain a $\Delta$-move with constant $\Str$-th strand for $\Str\neq\Strand_0$ as in the final diagram in Figure \ref{fig:kernel_Map_orb_braid_stretch}. 
\end{enumerate} 

\begin{figure}[H]
\centerline{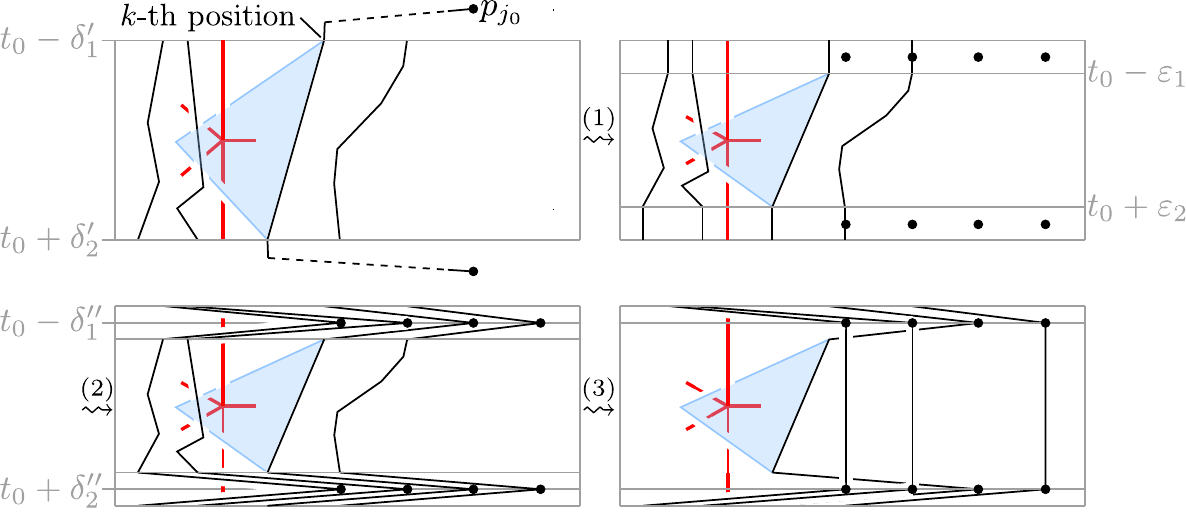}
\caption{Adjustment of the homotopy to read off the induced orbifold Reidemeister move.}
\label{fig:kernel_Map_orb_braid_stretch}
\end{figure}

Since we moved in front of the $\Strand_0$-th strand, this implies that the braids $h_{s_0-\delta_1}$ and $h_{s_0+\delta_2}$ differ by the orbifold Reidemeister move in Figure \ref{fig:orb-Reidemeister-move_Delta}. If  we assume that the braid $\bar{h}_{s_0-\delta_1}$ corresponds to a word $w_{s_0-\delta_1}$ in the generators from Corollary~\ref{cor:gen_set_Z_n}, this implies that the word $w_{s_0+\delta_2}$ which corresponds to the braid $\bar{h}_{s_0+\delta_2}$ differs from $w_{s_0-\delta_1}$ by insertion or deletion of $h_{\NStrand-1}...h_1\twC{\nu}^{\TheOrder_\nu}h_1^{-1}...h_{\NStrand-1}^{-1}$, i.e.\ a conjugate of $\twistC_{\NStrand\nu}^{\TheOrder_\nu}$. Analogously, the corresponding words $W_{s_0-\delta_1}$ and $W_{s_0+\delta_2}$ differ by insertion or deletion of $H_{\NStrand-1}...H_1\TwC{\nu}^{\TheOrder_\nu}H_1^{-1}...H_{\NStrand-1}^{-1}$. 
\end{intermediate}

\begin{figure}[H]
\centerline{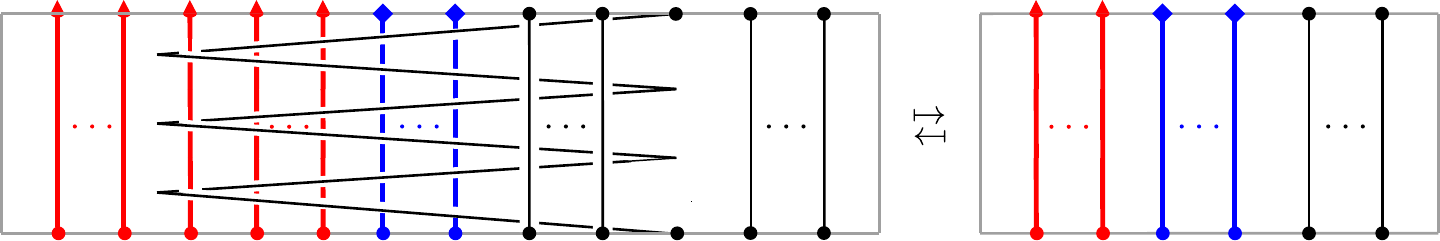}
\caption{Orbifold Reidemeister move that describes the difference between $h_{s_0-\delta_1}$ and $h_{s_0+\delta_2}$.}
\label{fig:orb-Reidemeister-move_Delta}
\end{figure} 

Applying the above adjustments to each cone point intersection of $h_s$, we obtain that the word $H=\rho_1^{\varepsilon_1}...\rho_\TheNSubdiv^{\varepsilon_\TheNSubdiv}$ differs from the empty word by a finite sequence of insertions or deletions 
\begin{itemize}
\item allowed by the relations in $\MapIdOrb{\TheStrand}{\Sigma_\freeprod(\ThePct)}$ and 
\item of 
conjugates of $\TwistC_{\NStrand\nu}^{\pm\TheOrder_\nu}$ for suitable $1\leq\NStrand\leq\TheStrand$ and $1\leq\nu\leq\TheCone$. 
\end{itemize} 

This implies that $H=\rho_1^{\varepsilon_1}...\rho_\TheNSubdiv^{\varepsilon_\TheNSubdiv}$ is contained in the normal closure of the set $\{\TwistC_{\NStrand\nu}^{\TheOrder_\nu}\mid1\leq\NStrand\leq\TheStrand,1\leq\nu\leq\TheCone\}$, i.e.: 
\[
\ker(\ev)\subseteq\langle\langle\TwistC_{\NStrand\nu}^{\TheOrder_\nu}\mid1\leq\NStrand\leq\TheStrand,1\leq\nu\leq\TheCone\rangle\rangle_{\MapIdOrb{\TheStrand}{\Sigma_\freeprod(\ThePct)}}. 
\]

In $\MapIdOrb{\TheStrand}{\Sigma_\freeprod(\ThePct)}$ each element $\TwistC_{\NStrand\nu}$ is a conjugate of $\TwC{\nu}$. Hence, 
\begin{align*}
\ker(\ev)= & \langle\langle \TwistC_{\NStrand\nu}^{\TheOrder_\nu}\mid1\leq\NStrand\leq\TheStrand,1\leq\nu\leq\TheCone\rangle\rangle_{\MapIdOrb{\TheStrand}{\Sigma_\freeprod(\ThePct)}}
\\
= & \langle\langle \TwC{\nu}^{\TheOrder_\nu}\mid1\leq\nu\leq\TheCone\rangle\rangle_{\MapIdOrb{\TheStrand}{\Sigma_\freeprod(\ThePct)}}. 
\end{align*}
The above arguments also yield 
\[
\ker(\ev\vert_{\PMapIdOrb{\TheStrand}{\Sigma_\freeprod(\ThePct)}})=\langle\langle\TwistC_{\NStrand\nu}^{\TheOrder_\nu}\mid1\leq\NStrand\leq\TheStrand, 1\leq\nu\leq\TheCone\rangle\rangle_{\PMapIdOrb{\TheStrand}{\Sigma_\freeprod(\ThePct)}}. 
\] 
\new{Hence, the proposition follows.} 
\end{proof}

This proves Theorem \ref{thm-intro:kernel_ev} and together with Corollary \ref{cor:pres_PMap_free_prod} and Proposition \ref{prop:pres_map_kcp} we obtain presentations of $\Z_\TheStrand(\Sigma_\freeprod(\ThePct))$ and its pure subgroup. 

\begin{corollary}
\label{cor:pres_pure_free_prod}
The pure braid group $\PZ_\TheStrand(\Sigma_\freeprod(\ThePct))$ has a presentation with generators 
\[
\twist_{\Strand\Str}, \twistP_{\NStrand\NPct} \; \text{ and } \; \twistC_{\NStrand\nu} 
\]
for $1\leq\Str,\Strand,\NStrand\leq\TheStrand$ with $\Str<\Strand$, $1\leq\NPct\leq\ThePct$ and $1\leq\nu\leq\TheCone$, and the following defining relations for $1\leq\Str,\Strand,\NStrand,\NNStrand\leq\TheStrand$ with $\Str<\Strand<\NStrand<\NNStrand$, $1\leq\Pc,\NPct\leq\ThePct$ with $\Pc<\NPct$ and $1\leq\mu,\nu,\leq\TheCone$ with $\mu<\nu$: 
\begin{enumerate}
\item 
\label{cor:pres_pure_free_prod_rel1}
$\twistC_{\NStrand\nu}^{\TheOrder_\nu}=1$, 
\item 
\label{cor:pres_pure_free_prod_rel2}
$[\twist_{\Strand\Str},\twist_{\NNStrand\NStrand}]=1$, 
$[\twistP_{\Strand\NPct},\twist_{\NNStrand\NStrand}]=1$, 
$[\twistC_{\Strand\nu},\twist_{\NNStrand\NStrand}]=1$, 
\item 
\label{cor:pres_pure_free_prod_rel3}
$[\twist_{\NNStrand\Str},\twist_{\NStrand\Strand}]=1$, 
$[\twistP_{\NNStrand\NPct},\twist_{\NStrand\Strand}]=1$, 
$[\twistP_{\NNStrand\NPct},\twistP_{\NStrand\Pc}]=1$, 
$[\twistC_{\NNStrand\nu},\twist_{\NStrand\Strand}]=1$, 
$[\twistC_{\NNStrand\nu},\twistP_{\NStrand\NPct}]=1$, 
$[\twistC_{\NNStrand\nu},\twistC_{\NStrand\mu}]=1$, 
\item 
\label{cor:pres_pure_free_prod_rel4}
$[\twist_{\NNStrand\NStrand}\twist_{\NNStrand\Strand}\twist_{\NNStrand\NStrand}^{-1},\twist_{\NStrand\Str}]=1$, 
$[\twist_{\NStrand\Strand}\twist_{\NStrand\Str}\twist_{\NStrand\Strand}^{-1},\twistP_{\Strand\NPct}]=1$, 
\new{$[\twist_{\NStrand\Strand}\twistP_{\NStrand\Pc}\twist_{\NStrand\Strand}^{-1},\twistP_{\Strand\NPct}]=1$,} 
\\
$[\twist_{\NStrand\Strand}\twist_{\NStrand\Str}\twist_{\NStrand\Strand}^{-1},\twistC_{\Strand\nu}]=1$, 
\new{$[\twist_{\NStrand\Strand}\twistC_{\NStrand\mu}\twist_{\NStrand\Strand}^{-1},\twistC_{\Strand\nu}]=1$,} 
\item 
\label{cor:pres_pure_free_prod_rel5}
$\twist_{\Strand\Str}\twist_{\NStrand\Strand}\twist_{\NStrand\Str}=\twist_{\NStrand\Str}\twist_{\Strand\Str}\twist_{\NStrand\Strand}=\twist_{\NStrand\Strand}\twist_{\NStrand\Str}\twist_{\Strand\Str}$, 
\\
$\twist_{\Strand\Str}\twistP_{\Strand\NPct}\twistP_{\Str\NPct}=\twistP_{\Str\NPct}\twist_{\Strand\Str}\twistP_{\Strand\NPct}=\twistP_{\Strand\NPct}\twistP_{\Str\NPct}\twist_{\Strand\Str}$, 
\\
$\twist_{\Strand\Str}\twistC_{\Strand\nu}\twistC_{\Str\nu}=\twistC_{\Str\nu}\twist_{\Strand\Str}\twistC_{\Strand\nu}=\twistC_{\Strand\nu}\twistC_{\Str\nu}\twist_{\Strand\Str}$. 
\end{enumerate}
\end{corollary}

\begin{corollary}
\label{cor:pres_orb_braid_free_prod}
The orbifold braid group $\Z_\TheStrand(\Sigma_\freeprod(\ThePct))$ is presented by generators 
\[
h_1,...,h_{\TheStrand-1},\twP{1},...,\twP{\ThePct},\twC{1},...,\twC{\TheCone}
\]
and the following defining relations for $2\leq\Strand<\TheStrand$, $1\leq\Pc,\NPct\leq\ThePct$ with $\Pc<\NPct$ and $1\leq\mu,\nu\leq\TheCone$ with $\mu<\nu$: 
\begin{enumerate}
\item \label{cor:pres_orb_braid_free_prod_rel1} 
$\twC{\nu}^{\TheOrder_\nu}=1$, 
\item \label{cor:pres_orb_braid_free_prod_rel2} 
braid and commutator relations for the generators $h_1,...,h_{\TheStrand-1}$, 
\item \label{cor:pres_orb_braid_free_prod_rel3} 
$[\twP{\NPct},h_\Strand]=1$ and 
$[\twC{\nu},h_\Strand]=1$, 
\item \label{cor:pres_orb_braid_free_prod_rel4} 
$[h_1\twP{\NPct}h_1,\twP{\NPct}]=1$ \; and \; 
$[h_1\twC{\nu}h_1,\twC{\nu}]=1$, 
\item \label{cor:pres_orb_braid_free_prod_rel5}
$[\twP{\Pc},\twistP_{2\NPct}]=1$, $[\twC{\mu},\twistC_{2\nu}]=1$ \; and \; $[\twP{\NPct},\twistC_{2\nu}]=1$
\\
with \; $\twistP_{2\NPct}=h_1^{-1}\twP{\NPct}h_1$ \; and \; $\twistC_{2\nu}=h_1^{-1}\twC{\nu} h_1$. 
\end{enumerate}
\end{corollary}

\new{This proves Theorem \ref{thm-intro:pres_orb_braid_free_prod}.} 

\subsection{An exact sequence of pure orbifold braid groups}
\label{subsec:ex_seq_pure_orb_braid}

It remains to deduce the main goal: 
an exact sequence of pure orbifold braid groups. For preparation, we briefly recall the classical situation with the short exact sequence  
\[
1\rightarrow\pi_1(D(\TheStrand-1),p_\TheStrand)\rightarrow\PB_\TheStrand\rightarrow\PB_{\TheStrand-1}\rightarrow1. 
\]
The first map embeds homotopy classes of loops in the punctured disk $D(\TheStrand-1)$ 
into $\PB_\TheStrand$ interpreting them as pure braids with $\TheStrand-1$ constant strands in positions of the punctures and the $\TheStrand$-th strand moving along the loops. The second map forgets the $\TheStrand$-th strand. Moreover, $\pi_1(D(\TheStrand-1),p_\TheStrand)$ is a free group with $\TheStrand-1$ generators. 

We want to establish a similar exact sequence 
\begin{equation}
\label{eq:ex_ser_pure_orb_braid_grps}
\piOrb\left(\Sigma_\freeprod(\TheStrand-1+\ThePct),p_\TheStrand\right)
\xrightarrow{\iotaPZ}\PZ_\TheStrand(\Sigma_\freeprod(\ThePct))\xrightarrow{\piPZ}\PZ_{\TheStrand-1}(\Sigma_\freeprod(\ThePct))\rightarrow1 
\end{equation}
for pure orbifold braid groups. The definition of the maps is analogous to the classical case: The map $\piPZ$ forgets the $\TheStrand$-th strand and the map $\iotaPZ$ interprets homotopy classes of $\freeprod$-loops with base point~$p_\TheStrand$ as pure braids in $\Sigma_\freeprod(\ThePct)$ with $\TheStrand-1$ constant strands in positions $p_1,...,p_{\TheStrand-1}$ 
and the $\TheStrand$-th strand moving along the $\freeprod$-loops. 
It will turn out that the map~$\iotaPZ$ has non-trivial kernel $\kernel_\TheStrand$. 

We begin by determining the orbifold fundamental group $\piOrb\left(\Sigma_\freeprod(\TheStrand-1+\ThePct)\right)$. 

\begin{lemma}
\label{lem:fundam_grp_Sigma-n-1_free_prod}
The group $\piOrb\left(\Sigma_\freeprod(\TheStrand-1+\ThePct)\right)$ is isomorphic to $\freegrp{\TheStrand-1+\ThePct}\ast\freeprod$. 
\end{lemma}
\newpage
\begin{proof}
We apply an orbifold version of the Seifert--van Kampen theorem \cite[Chapter III.G, 3.10(4)]{BridsonHaefliger2011}. Let $X_0$ be a punctured disk that satisfies the following conditions: 
\begin{itemize}
\item $X_0$ is open as a subset of $\Sigma(\TheStrand-1+\ThePct)$. 
\item $\freeprod(X_0)$ is a disjoint union of the $\freeprod$-translates of $X_0$. 
\item $X_0$ contains the punctures in positions $p_1,...,p_{\TheStrand-1}$ and $r_1,...,r_\ThePct$. 
\end{itemize}
Let $X:=\freeprod(X_0)$ and $Y$ a $\freeprod$-invariant open neighborhood of the cone points such that $Y\hookrightarrow\Sigma$ is a $\freeprod$-equivariant homotopy equivalence. 
Additionally, we assume that $X\cup Y=\Sigma(\TheStrand-1+\ThePct)$. The intersection $Z:=X\cap Y$ is a disjoint union of disks $\gamma(Z_0)$ for $\gamma\in\freeprod$ and each of these disks is open as a subset of $\Sigma(\TheStrand-1+\ThePct)$ and does not contain any punctures (see Figure \ref{fig:orb_fund_grp_Seifert-van_Kampen} for an example). 
\begin{figure}[H]
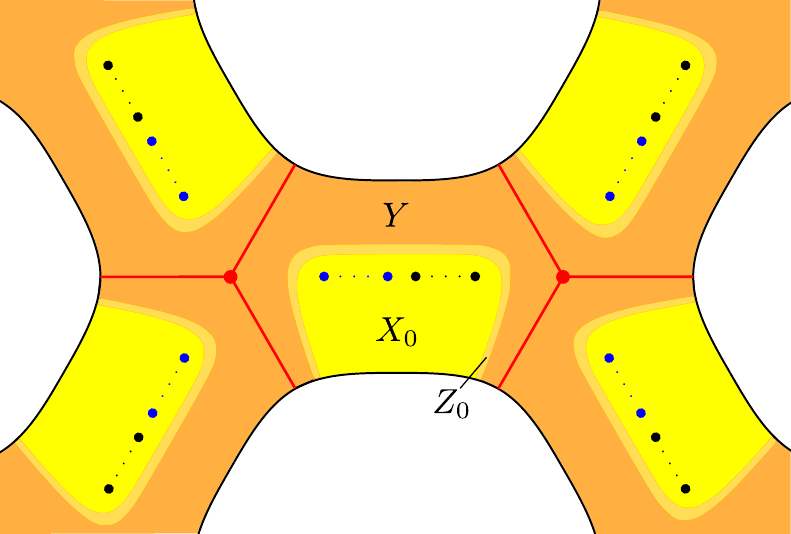
\caption{Decomposition of $\Sigma(\TheStrand-1+\ThePct)$ into open subsets $X$ and $Y$.}
\label{fig:orb_fund_grp_Seifert-van_Kampen}
\end{figure}

In general, we can obtain $X$ and $Y$ as follows: Let $\FD(\TheStrand-1+\ThePct)$ embed into $\CC$ as described in Figure \ref{fig:embedding_fund_domain} and define 
\[
X_0:=\FD(\TheStrand-1+\ThePct)\cap\{x\in\CC\mid\text{Im}(x)<\varepsilon_+\} \quad \text{ and } \quad \tilde{Y}=\FD(\TheStrand-1+\ThePct)\cap\{x\in\CC\mid\text{Im}(x)>\varepsilon_-\}
\]
for $\varepsilon_+>\varepsilon_->0$ such that 
\[
\varepsilon_+<\min\{\text{Im}(x)\mid x\in \FD(\TheStrand-1+\ThePct)\cap\gamma(\FD(\TheStrand-1+\ThePct)) \text{ for some } \gamma\neq1\}. 
\]
See Figure \ref{fig:orb_fund_grp_Seifert-van_Kampen_formal}. Then $Y:=\freeprod(\tilde{Y})$ and $X=\freeprod(X_0)$ satisfy the above properties. 

\begin{figure}[H]
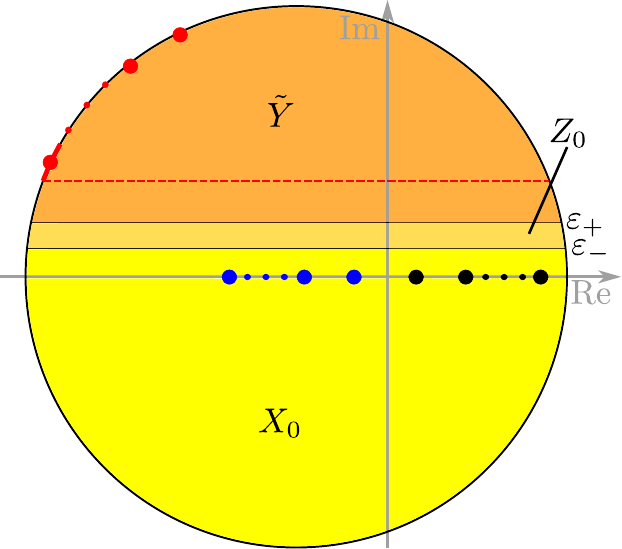
\caption{Subsets $X_0$ and $\tilde{Y}$ inside $\FD(\TheStrand-1+\ThePct)$.}
\label{fig:orb_fund_grp_Seifert-van_Kampen_formal}
\end{figure}

If we consider the orbifold fundamental group with respect to a base point in~$Z$, the Seifert--van Kampen theorem implies that the inclusion maps induce an isomorphism 
\\
\begin{minipage}{1.02\textwidth}
\begin{equation}
\label{lem:fundam_grp_Sigma-n-1_free_prod_eq1}
\piOrb(X_\freeprod)\ast\piOrb(Y_\freeprod)/\langle\langle \iota_X(\gamma)\iota_Y(\gamma)^{-1}\mid\gamma\in\piOrb(Z_\freeprod)\rangle\rangle\rightarrow\piOrb\left(\Sigma_\freeprod(\TheStrand-1+\ThePct)\right). 
\end{equation}
\end{minipage}
\\[7pt]
Here above, $\iota_X:\piOrb(Z_\freeprod)\hookrightarrow\piOrb(X_\freeprod)$ and $\iota_Y:\piOrb(Z_\freeprod)\hookrightarrow\piOrb(Y_\freeprod)$ denote the homomorphisms induced by the inclusion of $Z_\freeprod$ into $X_\freeprod$ and $Y_\freeprod$, respectively. 

Since the inclusion $Y\hookrightarrow\Sigma$ is a $\freeprod$-equivariant homotopy equivalence, the fundamental group of $Y_\freeprod$ equals $\piOrb(\Sigma_\freeprod)=\freeprod$, that was determined in Corollary \ref{cor:fund_grp_orb_ses}.  
Let us use $X_0$ and $Z_0$ to denote the path components of $X$ and $Z$ that contain the base point. In both cases, the subgroup of $\freeprod$ that leaves the path components invariant is trivial. \new{By Corollary \ref{cor:fund_grp_orb_ses},} the canonical maps $\pi_1(X_0)\rightarrow\piOrb(X_\freeprod)$ and $\pi_1(Z_0)\rightarrow\piOrb(Z_\freeprod)$ are isomorphisms. The subsurface $X_0$ is a disk with punctures in $p_1,...,p_{\TheStrand-1}$ and $r_1,...,r_\ThePct$. Hence, 
$\pi_1(X_0)\cong\freegrp{\TheStrand-1+\ThePct}$. The subsurface $Z_0$ is a disk without punctures, i.e.\ $\pi_1(Z_0)$ is trivial. Consequently, 
$\piOrb(X_\freeprod)\cong\freegrp{\TheStrand-1+\ThePct}$ and 
$\piOrb(Z_\freeprod)$ is trivial. By (\ref{lem:fundam_grp_Sigma-n-1_free_prod_eq1}), this implies that the group $\piOrb(\Sigma_\freeprod(\TheStrand-1+\ThePct)$ 
is isomorphic to the free product $\freegrp{\TheStrand-1+\ThePct}\ast\freeprod$. 
\end{proof}

In particular, the proof of Lemma \ref{lem:fundam_grp_Sigma-n-1_free_prod} shows that $\piOrb\left(\Sigma_\freeprod(\TheStrand-1+\ThePct)\right)$ is presented by generators 
\[
\twistn{\Strand}, \twistnP{\NPct} \; \text{ and } \; \twistnC{\nu} 
\]
with $1\leq\Strand<\TheStrand$, $,1\leq\NPct\leq\ThePct$ and $1\leq\nu\leq\TheCone$ represented by the $\freeprod$-loops in Figure~\ref{fig:gens_orb_fund_grp} and defining relations $\twistnC{\nu}^{\TheOrder_\nu}=1$ for $1\leq\nu\leq\TheStrand$. In the following, we will not distinguish between the representatives $\twistn{\Strand},\twistnP{\NPct}$ and $\twistnC{\nu}$ and the represented homotopy classes in $\piOrb\left(\Sigma_\freeprod(\TheStrand-1+\ThePct)\right)$. 
\begin{figure}[H]
\centerline{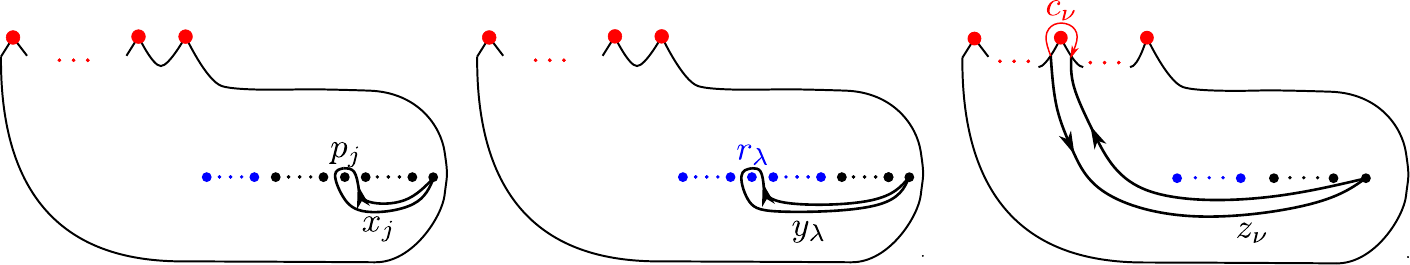}
\caption{The generators $\twistn{\Strand},\twistnP{\NPct}$ and $\twistnC{\nu}$ of $\piOrb(\Sigma_\freeprod(\TheStrand-1+\ThePct))$.}
\label{fig:gens_orb_fund_grp}
\end{figure}
\begin{intermediate}[The homomorphism $\iotaPZ$]
We define assignments $\iotaPZ$ by 
\begin{align*}
\twistn{\Strand} & \mapsto\twist_{\TheStrand\Strand} \; \text{ for } \; 1\leq\Strand<\TheStrand, 
\\
\twistnP{\NPct} & \mapsto \twistP_{\TheStrand\NPct} \; \text{ for } \; 1\leq\NPct<\ThePct \; \text{ and }
\\
\twistnC{\nu} & \mapsto\twistC_{\TheStrand\nu} \; \text{ for } \; 1\leq\nu<\TheCone. 
\end{align*}
Corollary \ref{cor:pres_pure_free_prod} yields that the assignments $\iotaPZ$ preserve the defining relations of the fundamental group $\piOrb\left(\Sigma_\freeprod(\TheStrand-1+\ThePct)\right)\cong\freegrp{\TheStrand-1+\ThePct}\ast\freeprod$. By Theorem \ref{thm:von_Dyck}, the assignments $\iotaPZ$ induce a homomorphism. 
\end{intermediate}

In the following, we will also use the notations $\twistn{\Strand}, \twistnP{\NPct}$ and $\twistnC{\nu}$ for their images under $\iotaPZ$ if we want to emphasize that an element can be represented by a braid where the only moving strand is the $\TheStrand$-th one. 

\begin{intermediate}[The homomorphism $\piPZ$]
We define assignments $\piPZ$ by 
\begin{align*}
\twist_{\Strand\Str} & \mapsto\twist_{\Strand\Str} \; \text{ for } \; 1\leq\Str<\Strand<\TheStrand, 
\\
\twistP_{\NStrand\NPct} & \mapsto\twistP_{\NStrand\NPct} \; \text{ for } \; 1\leq\NStrand<\TheStrand,1\leq\NPct\leq\ThePct \; \text{ and }
\\
\twistC_{\NStrand\nu} & \mapsto\twistC_{\NStrand\nu} \; \text{ for } \; 1\leq\NStrand<\TheStrand, 1\leq\nu\leq\TheCone 
\end{align*}
and the remaining generators $\twistn{\NStrand}, \twistnP{\NPct}$ and $\twistnC{\nu}$, which only move the $\TheStrand$-th strand, map to the identity under $\piPZ$. Once again, by Corollary \ref{cor:pres_pure_free_prod}, it is easy to verify that the assignments $\piPZ$ preserve the relations of $\PZ_\TheStrand(\Sigma_\freeprod(\ThePct))$. By Theorem \ref{thm:von_Dyck}, the assignments $\piPZ$ induce a homomorphism. 
\end{intermediate}

In particular, the definitions of $\iotaPZ$ and $\piPZ$ imply that $\piPZ\circ\iotaPZ$ is the trivial map. 

\begin{intermediate}[A right-inverse homomorphism $\sPZ$ of $\piPZ$]
\phantomsection
\label{def:right-inverse_sPZ}
A right inverse map of $\piPZ$ is induced by mapping the generators $\twist_{\Strand\Str},\twistP_{\NStrand\NPct}$ and $\twistC_{\NStrand\nu}$ of $\PZ_{\TheStrand-1}(\Sigma_\freeprod(\ThePct))$ to their homonyms in $\PZ_\TheStrand(\Sigma_\freeprod(\ThePct))$. Once more, using the presentation from Corollary \ref{cor:pres_pure_free_prod}, we can read off that this assignments preserve the defining relations of $\PZ_{\TheStrand-1}(\Sigma_\freeprod(\ThePct))$. By Theorem \ref{thm:von_Dyck}, this implies that these assignments induce a homomorphism from $\PZ_{\TheStrand-1}(\Sigma_\freeprod(\ThePct))$ to $\PZ_\TheStrand(\Sigma_\freeprod(\ThePct))$. On the level of generators, it is easy to see that this homomorphism yields a right-inverse of $\piPZ$. 
\end{intermediate}

\begin{intermediate}[Exactness]
Since $\piPZ$ is right invertible, it is surjective. As emphasized above $\textrm{im}(\iotaPZ)\subseteq\ker(\piPZ)$. \new{For the opposite inclusion,} let $z\in\ker(\piPZ)$. By Corollary~\ref{cor:gen_set_PZ_n}, $z$ decomposes into generators 
\[
\twist_{\Strand\Str}, \twistP_{\NStrand\NPct} \; \text{ and } \; \twistC_{\NStrand\nu}
\]
for $1\leq\Str,\Strand,\NStrand\leq\TheStrand,\Str<\Strand,1\leq\NPct\leq\ThePct$ and $1\leq\nu\leq\TheCone$. On the one hand, for $\Strand=\TheStrand$ and $\NStrand=\TheStrand$, we denote $\twist_{\TheStrand\Str}=\twistn{\Str}$, $\twistP_{\TheStrand\NPct}=\twistnP{\NPct}$ and $\twistC_{\TheStrand\nu}=\twistnC{\nu}$. On the other hand, for $1\leq\Strand,\NStrand<\TheStrand$, \eqref{eq:def_a_ji} and \eqref{eq:def_a_kc_a_kr} yield 
\begin{align*}
\twist_{\Strand\Str} & = h_{\Strand-1}^{-1}...h_{\Str+1}^{-1}h_\Str^2h_{\Str+1}...h_{\Strand-1}, 
\\
\twistP_{\NStrand\NPct} & = h_{\NStrand-1}^{-1}...h_1^{-1}\twP{\NPct}h_1...h_{\NStrand-1} \; \text{ and }
\\
\twistC_{\NStrand\nu} & = h_{\NStrand-1}^{-1}...h_1^{-1}\twC{\nu}h_1...h_{\NStrand-1}. 
\end{align*}
Since none of these products contains the generator $h_{\TheStrand-1}$, every $z\in\ker(\piPZ)$ decomposes into generators 
\[
\twistn{\NStrand},\twistnP{\NPct},\twistnC{\nu} \; \text{ and } \; h_\Strand,\twP{\NPct},\twC{\nu}
\]
with $1\leq\Strand,\NStrand<\TheStrand,\Strand\neq\TheStrand-2,1\leq\NPct\leq\ThePct$ and $1\leq\nu\leq\TheCone$. \new{Using that $\MapIdOrb{\TheStrand}{\Sigma_\freeprod(\ThePct)}$ projects onto $\Z_\TheStrand(\Sigma_\freeprod(\ThePct))$, the relations from \cite[Lemma 4.21]{Flechsig2023mcg} also hold for the corresponding elements in $\Z_\TheStrand(\Sigma_\freeprod(\ThePct))$.} This allows us to split $z$ into subwords 
\begin{align*}
& W_1(\twistn{1},...,\twistn{\TheStrand-1},\twistnP{1},...,\twistnP{\ThePct},\twistnC{1},...,\twistnC{\TheCone}) \; \text{ and } 
\\
& W_2(h_1,...,h_{\TheStrand-2},\twP{1},...,\twP{\ThePct},\twC{1},...,\twC{\TheCone}) 
\end{align*}
with $z=W_1\cdot W_2$. Here, the motion of the $\TheStrand$-th strand is separated into $W_1$, the braid $W_2$ only moves the first $\TheStrand-1$ strands. Consequently, $W_1$ maps trivially under the map $\piPZ$ that forgets the $\TheStrand$-th strand. Since $z$ is from $\ker(\piPZ)$, this implies that $W_2$ is also a pure braid and 
$\piPZ(W_2)=1$ in $\PZ_{\TheStrand-1}(\Sigma_\freeprod(\ThePct))$. Further, $W_2$ is contained in the subgroup 
\[
\left\langle\twist_{\Strand\Str},\twistP_{\NStrand\NPct},\twistC_{\NStrand\nu} \; \text{ for } \; \begin{matrix}
1\leq\Str,\Strand,\NStrand<\TheStrand, \; \Str<\Strand,
\\
1\leq\NPct\leq\ThePct \text{ and } 1\leq\nu\leq\ThePct
\end{matrix}
\right\rangle\leq\PZ_\TheStrand(\Sigma_\freeprod(\ThePct)).
\]
If we restrict $\piPZ$ to this subgroup, we obtain an isomorphism; the section $\sPZ$ constructed above yields an inverse. Hence, $W_2=1$ in $\PZ_\TheStrand(\Sigma_\freeprod(\ThePct))$, i.e.\ $z=W_1$ is contained in $\textrm{im}(\iotaPZ)$. 
This implies that the sequence from (\ref{eq:ex_ser_pure_orb_braid_grps}) is exact and the homomorphism $\piPZ$ has a section. 
\end{intermediate}

Let us denote $\kernel=\kernel_\TheStrand:=\ker(\iotaPZ)$. Then Definition~\ref{def:semidir_prod} yields: 

\begin{corollary}
\label{cor:pure_orb_braid_semidir_prod}
The group $\PZ_\TheStrand(\Sigma_\freeprod(\ThePct))$ is a semidirect product 
\[
((\freegrp{\TheStrand-1+\ThePct}\ast\freeprod)/\kernel)\rtimes\PZ_{\TheStrand-1}(\Sigma_\freeprod(\ThePct)). 
\]
\end{corollary}

It remains to calculate the kernel $\kernel$. For this purpose, we determine a presentation of $(\freegrp{\TheStrand-1+\ThePct}\ast\freeprod)/\kernel$.  The proof is divided in the following two steps. 
\setlist[enumerate,1]{label=\textup{\arabic*.}, ref=\arabic*}
\begin{enumerate}
\item \label{it:det_ker_1} 
Find a relation that holds in $(\freegrp{\TheStrand-1+\ThePct}\ast\freeprod)/\kernel$ but not in $\freegrp{\TheStrand-1+\ThePct}\ast\freeprod$, i.e.\ the kernel $\kernel$ is non-trivial.  
\item \label{it:det_ker_2}
Construct relations that suffice to reformulate the presentation of the pure orbifold braid group $\PZ_\TheStrand(\Sigma_\freeprod(\ThePct))$ from Corollary \ref{cor:pres_pure_free_prod} such that it satisfies the conditions from Lemma \ref{lem:semidir_prod_pres_it2}. Then the latter presentation describes the semidirect product structure of $\PZ_\TheStrand(\Sigma_\freeprod(\ThePct))$ and we can read off a presentation of $(\freegrp{\TheStrand-1+\ThePct}\ast\freeprod)/\kernel$. In particular, this determines $\kernel$. 
\end{enumerate}
\setlist[enumerate,1]{label=\textup{(\arabic*)}, ref=\thetheorem(\arabic*)}

\subsection*{Step \ref{it:det_ker_1}. The kernel $\kernel$ is non-trivial}
Let us consider the alphabet $X=\{\twistn{1},...,\twistn{\TheStrand-1},\twistnP{1},...,\twistnP{\ThePct},\twistnC{1},...,\twistnC{\TheCone}\}$. Recall that our goal in Step \ref{it:det_ker_2} will be to establish a set of relations $R$ such that the assignments 
\begin{align*}
X & \rightarrow(\freegrp{\TheStrand-1+\ThePct}\ast\freeprod)/\kernel, 
\\
\twistn{\Strand} & \mapsto \twistn{\Strand}\kernel, 
\\
\twistnP{\NPct} & \mapsto \twistnP{\NPct}\kernel \; \text{ and } 
\\
\twistnC{\nu} & \mapsto \twistnC{\nu}\kernel 
\end{align*}
induce an isomorphism $\langle X\mid R\rangle\rightarrow(\freegrp{\TheStrand-1+\ThePct}\ast\freeprod)/\kernel$. Obviously, the set $R$ contains the relations $\twistnC{\nu}^{\TheOrder_\nu}=1$ for $1\leq\nu\leq\TheCone$ coming from the relations in $\freegrp{\TheStrand-1+\ThePct}\ast\freeprod$. The goal in Step \ref{it:det_ker_1} is to observe that the set $R$ contains further relations. Therefore, we recall that $(\freegrp{\TheStrand-1+\ThePct}\ast\freeprod)/\kernel$ embeds into $\PZ_\TheStrand(\Sigma_\freeprod(\ThePct))$ via an embedding induced by $\iotaPZ$. We start observing additional relations in $\PZ_\TheStrand(\Sigma_\freeprod(\ThePct))$. 

Later on it will be important that these additional relations follow from the conjugation relations summarized in \ref{prop:pure_br_semidir_pres_rel6_conj}. In particular, \ref{prop:pure_br_semidir_pres_rel6_conj} contains the following relations: 
\begin{align*}
\label{eq:conj_cp_1}
\twistC_{\NStrand\nu}\twist_{\TheStrand\NStrand}\twistC_{\NStrand\nu}^{-1} & \mystackrel{\ref{prop:pure_br_semidir_pres_rel6_conj_n}}=\twistC_{\TheStrand\nu}^{-1}\twist_{\TheStrand\NStrand}\twistC_{\TheStrand\nu}\stackrel{\vee}=\twistC_{\TheStrand\nu}^{-1}(\twist_{\TheStrand\NStrand}^{-1}\twist_{\TheStrand\NStrand})\twist_{\TheStrand\NStrand}\twistC_{\TheStrand\nu},  \numbereq
\\
\label{eq:conj_cp_2}
\twistC_{\NStrand\nu}\twistC_{\TheStrand\nu}\twistC_{\NStrand\nu}^{-1} & \mystackrel{\ref{prop:pure_br_semidir_pres_rel6_conj_d'}}=\twistC_{\TheStrand\nu}^{-1}\twist_{\TheStrand\NStrand}^{-1}\twistC_{\TheStrand\nu}\twist_{\TheStrand\NStrand}\twistC_{\TheStrand\nu}, \numbereq 
\\
\label{eq:conj_cp_3}
\twistC_{\NStrand\nu}\twist_{\TheStrand\Strand}\twistC_{\NStrand\nu}^{-1} & \mystackrel{\ref{prop:pure_br_semidir_pres_rel6_conj_h}}=\twistC_{\TheStrand\nu}^{-1}\twist_{\TheStrand\NStrand}^{-1}\twistC_{\TheStrand\nu}\twist_{\TheStrand\NStrand}\twist_{\TheStrand\Strand}\twist_{\TheStrand\NStrand}^{-1}\twistC_{\TheStrand\nu}^{-1}\twist_{\TheStrand\NStrand} \twistC_{\TheStrand\nu}, \numbereq
\\
\label{eq:conj_cp_4}
\twistC_{\NStrand\nu}\twistP_{\TheStrand\NPct}\twistC_{\NStrand\nu}^{-1} & \mystackrel{\ref{prop:pure_br_semidir_pres_rel6_conj_y}}=\twistC_{\TheStrand\nu}^{-1}\twist_{\TheStrand\NStrand}^{-1}\twistC_{\TheStrand\nu}\twist_{\TheStrand\NStrand}\twistP_{\TheStrand\NPct}\twist_{\TheStrand\NStrand}^{-1}\twistC_{\TheStrand\nu}^{-1}\twist_{\TheStrand\NStrand}\twistC_{\TheStrand\nu}, \numbereq 
\\
\label{eq:conj_cp_5}
\twistC_{\NStrand\nu}\twistC_{\TheStrand\mu}\twistC_{\NStrand\nu}^{-1} & \mystackrel{\ref{prop:pure_br_semidir_pres_rel6_conj_f'}}=\twistC_{\TheStrand\nu}^{-1}\twist_{\TheStrand\NStrand}^{-1}\twistC_{\TheStrand\nu}\twist_{\TheStrand\NStrand}\twistC_{\TheStrand\mu}\twist_{\TheStrand\NStrand}^{-1}\twistC_{\TheStrand\nu}^{-1}\twist_{\TheStrand\NStrand}\twistC_{\TheStrand\nu}. \numbereq 
\end{align*}
These relations are the $\ev$-images of relations of the form in \ref{lem:PMap_gens_sat_rels_it1d}, \ref{lem:PMap_gens_sat_rels_it1f} in Lemma~\ref{lem:PMap_gens_sat_rels_it1} and  \ref{lem:PMap_gens_sat_rels_it3c}, \ref{lem:PMap_gens_sat_rels_it3e} and \ref{lem:PMap_gens_sat_rels_it3f} in Lemma \ref{lem:PMap_gens_sat_rels_it3}, respectively. Since $\ev$ induces a homomorphism from $\PMapIdOrb{\TheStrand}{\Sigma_\freeprod(\ThePct)}$ to $\PZ_\TheStrand(\Sigma_\freeprod(\ThePct))$, the relations (\ref{eq:conj_cp_1}) to (\ref{eq:conj_cp_5}) are satisfied in $\PZ_\TheStrand(\Sigma_\freeprod(\ThePct))$. 

\begin{lemma}
\label{lem:it_conj}
Given the group $\PZ_\TheStrand(\Sigma_\freeprod(\ThePct))$, the generators satisfy the following relations for each $z\in\NN$: 
\begin{align*}
\twistC_{\NStrand\nu}^z\twist_{\TheStrand\NStrand}\twistC_{\NStrand\nu}^{-z} & =(\twistC_{\TheStrand\nu}^{-1}\twist_{\TheStrand\NStrand}^{-1})^z \twist_{\TheStrand\NStrand}(\twist_{\TheStrand\NStrand}\twistC_{\TheStrand\nu})^z, 
\\
\twistC_{\NStrand\nu}^z\twistC_{\TheStrand\nu}\twistC_{\NStrand\nu}^{-z} & =(\twistC_{\TheStrand\nu}^{-1}\twist_{\TheStrand\NStrand}^{-1})^z\twistC_{\TheStrand\nu}(\twist_{\TheStrand\NStrand}\twistC_{\TheStrand\nu})^z, 
\\
\twistC_{\NStrand\nu}^z\twist_{\TheStrand\Strand}\twistC_{\NStrand\nu}^{-z} & =(\twistC_{\TheStrand\nu}^{-1}\twist_{\TheStrand\NStrand}^{-1})^z(\twistC_{\TheStrand\nu}\twist_{\TheStrand\NStrand})^z\twist_{\TheStrand\Strand}(\twist_{\TheStrand\NStrand}^{-1}\twistC_{\TheStrand\nu}^{-1})^z(\twist_{\TheStrand\NStrand} \twistC_{\TheStrand\nu})^z, 
\\
\twistC_{\NStrand\nu}^z\twistP_{\TheStrand\NPct}\twistC_{\NStrand\nu}^{-z} & =(\twistC_{\TheStrand\nu}^{-1}\twist_{\TheStrand\NStrand}^{-1})^z(\twistC_{\TheStrand\nu}\twist_{\TheStrand\NStrand})^z\twistP_{\TheStrand\NPct}(\twist_{\TheStrand\NStrand}^{-1}\twistC_{\TheStrand\nu}^{-1})^z(\twist_{\TheStrand\NStrand} \twistC_{\TheStrand\nu})^z, 
\\
\twistC_{\NStrand\nu}^z\twistC_{\TheStrand\mu}\twistC_{\NStrand\nu}^{-z} & =(\twistC_{\TheStrand\nu}^{-1}\twist_{\TheStrand\NStrand}^{-1})^z(\twistC_{\TheStrand\nu}\twist_{\TheStrand\NStrand})^z\twistC_{\TheStrand\mu}(\twist_{\TheStrand\NStrand}^{-1}\twistC_{\TheStrand\nu}^{-1})^z(\twist_{\TheStrand\NStrand}\twistC_{\TheStrand\nu})^z. 
\end{align*}
\end{lemma}
\begin{proof}
Relation (\ref{eq:conj_cp_1}) serves as the base case for an induction to show 
\begin{equation}
\label{lem:it_conj_eq_1_z=1}
\twistC_{\NStrand\nu}^z \twist_{\TheStrand\NStrand}\twistC_{\NStrand\nu}^{-z}=(\twistC_{\TheStrand\nu}^{-1}\twist_{\TheStrand\NStrand}^{-1})^z \twist_{\TheStrand\NStrand}(\twist_{\TheStrand\NStrand}\twistC_{\TheStrand\nu})^z 
\end{equation}
for all $z\in\NN$. Assuming the claim for $z-1$, we obtain: 
\begin{align*}
& \twistC_{\NStrand\nu}^z \twist_{\TheStrand\NStrand}\twistC_{\NStrand\nu}^{-z} 
\\
\mystackrel{}= & \twistC_{\NStrand\nu}\textcolor{\col}{\twistC_{\NStrand\nu}^{z-1} \twist_{\TheStrand\NStrand}\twistC_{\NStrand\nu}^{-(z-1)}}\twistC_{\NStrand\nu}^{-1}
\\
\mystackrel{\text{i.h.}}= & \twistC_{\NStrand\nu}(\twistC_{\TheStrand\nu}^{-1}\twist_{\TheStrand\NStrand}^{-1})^{z-1}\twist_{\TheStrand\NStrand}(\twist_{\TheStrand\NStrand}\twistC_{\TheStrand\nu})^{z-1}\twistC_{\NStrand\nu}^{-1} 
\\
\mystackrel{$\vee$}= & \textcolor{\col}{\twistC_{\NStrand\nu}(\twistC_{\TheStrand\nu}^{-1}\twist_{\TheStrand\NStrand}^{-1})^{z-1}(\twistC_{\NStrand\nu}^{-1}}\twistC_{\NStrand\nu})\twist_{\TheStrand\NStrand}(\twistC_{\NStrand\nu}^{-1}\textcolor{\col}{\twistC_{\NStrand\nu})(\twist_{\TheStrand\NStrand}\twistC_{\TheStrand\nu})^{z-1}\twistC_{\NStrand\nu}^{-1}}
\\
\mystackrel{(\ref{eq:conj_cp_1},\ref{eq:conj_cp_2})}= & (\twistC_{\TheStrand\nu}^{-1}\twist_{\TheStrand\NStrand}^{-1})^{z-1}\textcolor{\col}{\twistC_{\NStrand\nu}\twist_{\TheStrand\NStrand}\twistC_{\NStrand\nu}^{-1}}(\twist_{\TheStrand\NStrand}\twistC_{\TheStrand\nu})^{z-1} 
\\
\mystackrel{(\ref{eq:conj_cp_1})}= & (\twistC_{\TheStrand\nu}^{-1}\twist_{\TheStrand\NStrand}^{-1})^z\twist_{\TheStrand\NStrand}(\twist_{\TheStrand\NStrand}\twistC_{\TheStrand\nu})^z. 
\end{align*}
By induction on $z$, this implies that the first equation 
holds for every $z\in\NN$. 

Similarly, we obtain $\twistC_{\NStrand\nu}^z\twistC_{\TheStrand\nu}\twistC_{\NStrand\nu}^{-z}=(\twistC_{\TheStrand\nu}^{-1}\twist_{\TheStrand\NStrand}^{-1})^z\twistC_{\TheStrand\nu}(\twist_{\TheStrand\NStrand}\twistC_{\TheStrand\nu})^z$ for $z=1$ from equation~(\ref{eq:conj_cp_2}) and the induction step follows as for the first equation. 

For $z=1$, the third relation follows from (\ref{eq:conj_cp_3}). Assuming it for $z-1$, we may deduce 
\begin{align*}
\twistC_{\NStrand\nu}^z\twist_{\TheStrand\Strand}\twistC_{\NStrand\nu}^{-z} \mystackrel{}= & \twistC_{\NStrand\nu}\textcolor{\col}{\twistC_{\NStrand\nu}^{z-1}\twist_{\TheStrand\Strand}\twistC_{\NStrand\nu}^{-(z-1)}}\twistC_{\NStrand\nu}^{-1}
\\
\mystackrel{\text{i.h.}}= & \twistC_{\NStrand\nu}(\twistC_{\TheStrand\nu}^{-1}\twist_{\TheStrand\NStrand}^{-1})^{z-1}(\twistC_{\TheStrand\nu}\twist_{\TheStrand\NStrand})^{z-1}\twist_{\TheStrand\Strand}(\twist_{\TheStrand\NStrand}^{-1}\twistC_{\TheStrand\nu}^{-1})^{z-1}(\twist_{\TheStrand\NStrand}\twistC_{\TheStrand\nu})^{z-1}\twistC_{\NStrand\nu}^{-1} 
\\
\mystackrel{$\vee$}= & \textcolor{\col}{\twistC_{\NStrand\nu}(\twistC_{\TheStrand\nu}^{-1}\twist_{\TheStrand\NStrand}^{-1})^{z-1}(\twistC_{\NStrand\nu}^{-1}}\twistC_{\NStrand\nu})(\twistC_{\TheStrand\nu}\twist_{\TheStrand\NStrand})^{z-1}\twist_{\TheStrand\Strand}(\twist_{\TheStrand\NStrand}^{-1}\twistC_{\TheStrand\nu}^{-1})^{z-1}(\twistC_{\NStrand\nu}^{-1}\textcolor{\col}{\twistC_{\NStrand\nu})(\twist_{\TheStrand\NStrand}\twistC_{\TheStrand\nu})^{z-1}\twistC_{\NStrand\nu}^{-1}} 
\\
\mystackrel{(\ref{eq:conj_cp_1},\ref{eq:conj_cp_2})}= & (\twistC_{\TheStrand\nu}^{-1}\twist_{\TheStrand\NStrand}^{-1})^{z-1}\twistC_{\NStrand\nu}(\twistC_{\TheStrand\nu}\twist_{\TheStrand\NStrand})^{z-1}\twist_{\TheStrand\Strand}(\twist_{\TheStrand\NStrand}^{-1}\twistC_{\TheStrand\nu}^{-1})^{z-1}\twistC_{\NStrand\nu}^{-1}(\twist_{\TheStrand\NStrand}\twistC_{\TheStrand\nu})^{z-1} 
\\
\mystackrel{$\vee$}= & (\twistC_{\TheStrand\nu}^{-1}\twist_{\TheStrand\NStrand}^{-1})^{z-1}\textcolor{\col}{\twistC_{\NStrand\nu}(\twistC_{\TheStrand\nu}\twist_{\TheStrand\NStrand})^{z-1}(\twistC_{\NStrand\nu}^{-1}}\twistC_{\NStrand\nu})\twist_{\TheStrand\Strand}(\twistC_{\NStrand\nu}^{-1}\textcolor{\col}{\twistC_{\NStrand\nu})(\twist_{\TheStrand\NStrand}^{-1}\twistC_{\TheStrand\nu}^{-1})^{z-1}\twistC_{\NStrand\nu}^{-1}}(\twist_{\TheStrand\NStrand}\twistC_{\TheStrand\nu})^{z-1} 
\\
\mystackrel{(\ref{eq:conj_cp_1},\ref{eq:conj_cp_2})}= & (\twistC_{\TheStrand\nu}^{-1}\twist_{\TheStrand\NStrand}^{-1})^{z-1}\twistC_{\TheStrand\nu}^{-1}\twist_{\TheStrand\NStrand}^{-1}(\twistC_{\TheStrand\nu}\twist_{\TheStrand\NStrand})^{z-1}\twist_{\TheStrand\NStrand}\twistC_{\TheStrand\nu}\textcolor{\col}{\twistC_{\NStrand\nu}\twist_{\TheStrand\Strand}}
\\
& \textcolor{\col}{\twistC_{\NStrand\nu}^{-1}}\twistC_{\TheStrand\nu}^{-1}\twist_{\TheStrand\NStrand}^{-1}(\twist_{\TheStrand\NStrand}^{-1}\twistC_{\TheStrand\nu}^{-1})^{z-1}\twist_{\TheStrand\NStrand}\twistC_{\TheStrand\nu}(\twist_{\TheStrand\NStrand}\twistC_{\TheStrand\nu})^{z-1}
\\
\mystackrel{(\ref{eq:conj_cp_3})}= & (\twistC_{\TheStrand\nu}^{-1}\twist_{\TheStrand\NStrand}^{-1})^{z-1}\twistC_{\TheStrand\nu}^{-1}\twist_{\TheStrand\NStrand}^{-1}(\twistC_{\TheStrand\nu}\twist_{\TheStrand\NStrand})^{z-1}\textcolor{\short}{\twist_{\TheStrand\NStrand}\twistC_{\TheStrand\nu}(\twistC_{\TheStrand\nu}^{-1}\twist_{\TheStrand\NStrand}^{-1}}\twistC_{\TheStrand\nu}\twist_{\TheStrand\NStrand}\twist_{\TheStrand\Strand}\twist_{\TheStrand\NStrand}^{-1}\twistC_{\TheStrand\nu}^{-1}\textcolor{\short}{\twist_{\TheStrand\NStrand}\twistC_{\TheStrand\nu})}
\\
& \textcolor{\short}{\twistC_{\TheStrand\nu}^{-1}\twist_{\TheStrand\NStrand}^{-1}}(\twist_{\TheStrand\NStrand}^{-1}\twistC_{\TheStrand\nu}^{-1})^{z-1}\twist_{\TheStrand\NStrand}\twistC_{\TheStrand\nu}(\twist_{\TheStrand\NStrand}\twistC_{\TheStrand\nu})^{z-1}
\\
\mystackrel{}= & (\twistC_{\TheStrand\nu}^{-1}\twist_{\TheStrand\NStrand}^{-1})^z(\twistC_{\TheStrand\nu}\twist_{\TheStrand\NStrand})^z\twist_{\TheStrand\Strand}(\twist_{\TheStrand\NStrand}^{-1}\twistC_{\TheStrand\nu}^{-1})^z(\twist_{\TheStrand\NStrand} \twistC_{\TheStrand\nu})^z. 
\end{align*}
Hence, the third relation from the claim follows for every $z\in\NN$. 

The remaining relations $\twistC_{\NStrand\nu}^z\twistP_{\TheStrand\NPct}\twistC_{\NStrand\nu}^{-z}=(\twistC_{\TheStrand\nu}^{-1}\twist_{\TheStrand\NStrand}^{-1})^z(\twistC_{\TheStrand\nu}\twist_{\TheStrand\NStrand})^z\twistP_{\TheStrand\NPct}(\twist_{\TheStrand\NStrand}^{-1}\twistC_{\TheStrand\nu}^{-1})^z(\twist_{\TheStrand\NStrand} \twistC_{\TheStrand\nu})^z$ and $\twistC_{\NStrand\nu}^z\twistC_{\TheStrand\mu}\twistC_{\NStrand\nu}^{-z}=(\twistC_{\TheStrand\nu}^{-1}\twist_{\TheStrand\NStrand}^{-1})^z(\twistC_{\TheStrand\nu}\twist_{\TheStrand\NStrand})^z\twistC_{\TheStrand\mu}(\twist_{\TheStrand\NStrand}^{-1}\twistC_{\TheStrand\nu}^{-1})^z(\twist_{\TheStrand\NStrand}\twistC_{\TheStrand\nu})^z$ are given for $z=1$ in (\ref{eq:conj_cp_4}) and (\ref{eq:conj_cp_5}). The induction step follows as for the third relation.  
\end{proof}

If we further recall from relation \ref{cor:pres_pure_free_prod_rel1} that the $\TheOrder_\nu$-th power of $\twistC_{\NStrand\nu}$ is trivial, the first relation from the previous lemma yields 
\begin{equation}
\label{eq:rel_ker}
\twist_{\TheStrand\NStrand}\stackrel{\ref{cor:pres_pure_free_prod_rel1}}=\twistC_{\NStrand\nu}^{\TheOrder_\nu}\twist_{\TheStrand\NStrand}\twistC_{\NStrand\nu}^{-\TheOrder_\nu}=(\twistC_{\TheStrand\nu}^{-1}\textcolor{\short}{\twist_{\TheStrand\NStrand}^{-1}})^{\TheOrder_\nu}\textcolor{\short}{\twist_{\TheStrand\NStrand}}(\twist_{\TheStrand\NStrand}\twistC_{\TheStrand\nu})^{\TheOrder_\nu}=(\twistC_{\TheStrand\nu}^{-1}\twist_{\TheStrand\NStrand}^{-1})^{\TheOrder_\nu-1}\twistC_{\TheStrand\nu}^{-1}(\twist_{\TheStrand\NStrand}\twistC_{\TheStrand\nu})^{\TheOrder_\nu} 
\end{equation}
for each $1\leq\NStrand<\TheStrand$ and $1\leq\nu\leq\TheCone$. By left multiplication with $\twistC_{\TheStrand\nu}(\twist_{\TheStrand\NStrand}\twistC_{\TheStrand\nu})^{\TheOrder_\nu-1}$, this is equivalent to 
\begin{equation}
\label{eq:non-trivial_ker}
(\twistC_{\TheStrand\nu}\twist_{\TheStrand\NStrand})^{\TheOrder_\nu}=(\twist_{\TheStrand\NStrand}\twistC_{\TheStrand\nu})^{\TheOrder_\nu}. 
\end{equation}
Using that $\iotaPZ$ induces a monomorphism from $(\freegrp{\TheStrand-1+\ThePct}\ast\freeprod)/\kernel$ into $\PZ_\TheStrand(\Sigma(\ThePct))$, the corresponding relations $(\twistnC{\nu}\kernel\twistn{\NStrand}\kernel)^{\TheOrder_\nu}=(\twistn{\NStrand}\kernel\twistnC{\nu}\kernel)^{\TheOrder_\nu}$ for $1\leq\NStrand<\TheStrand$ and $1\leq\nu\leq\TheCone$ hold in $(\freegrp{\TheStrand-1+\ThePct}\ast\freeprod)/\kernel$. Hence, the relations $(\twistnC{\nu}\twistn{\NStrand})^{\TheOrder_\nu}=(\twistn{\NStrand}\twistnC{\nu})^{\TheOrder_\nu}$ are contained in $R$ for each $\NStrand$ and $\nu$. Clearly, the corresponding reduced words do not coincide in the free product $\freegrp{\TheStrand-1+\ThePct}\ast\freeprod$, i.e.\ the above relations lead to non-trivial elements in the kernel $\kernel$. 

\subsection*{Step \ref{it:det_ker_2}. Determining the kernel $\kernel$}

To find a presentation of the group $\PZ_\TheStrand(\Sigma_\freeprod(\ThePct))$ that satisfies the conditions from Lemma~\ref{lem:semidir_prod_pres_it2}, we need to observe further relations that hold in $(\freegrp{\TheStrand-1+\ThePct}\ast\freeprod)/\kernel$. For this purpose, we introduce \textit{partial conjugations}:  

\begin{definition}[Partial conjugation]
\label{def:partial_conj}
Let us consider the alphabet $X$ from above and the free group $\freegrp{}(X)$. For $1\leq\Str,\NStrand\leq\TheStrand, \Str<\NStrand, 1\leq\Pct\leq\ThePct$ and $1\leq\nu\leq\TheCone$, let 
\[
pc_{\twist_{\NStrand\Str}}, pc_{\twistP_{\NStrand\Pct}} \; \text{ and } \; pc_{\twist_{\NStrand\nu}}:\freegrp{}(X)\rightarrow\freegrp{}(X)
\]
be the endomorphisms that replace each letter according to Table \ref{tab:partial_conj} by the words in the second, third or forth column, respectively. These endomorphisms are called \textit{partial conjugations}. 

\begin{table}[H]
{\small
\renewcommand{\arraystretch}{1.25}
\centerline{
\begin{tabular}{c|c|c|c}
 & $pc_{\twist_{\NStrand\Str}}$ & $pc_{\twistP_{\NStrand\Pct}}$ & $pc_{\twistC_{\NStrand\nu}}$
\\ \hline\hline
$\twistn{\NNStrand}$ & $\twistn{\NNStrand}$ & $\twistn{\NNStrand}$ & $\twistn{\NNStrand}$
\\ \hline
$\twistn{\NStrand}$ & $\twistn{\Str}^{-1}\twistn{\NStrand} \twistn{\Str}$ & $\twistnP{\Pct}^{-1}\twistn{\NStrand} \twistnP{\Pct}$ & $\twistnC{\nu}^{-1}\twistn{\NStrand} \twistnC{\nu}$
\\ \hline
$\twistn{\Strand}$ & $\twistn{\Str}^{-1}\twistn{\NStrand}^{-1}\twistn{\Str} \twistn{\NStrand} \twistn{\Strand} \twistn{\NStrand}^{-1}\twistn{\Str}^{-1}\twistn{\NStrand} \twistn{\Str}$ & $\twistnP{\Pct}^{-1}\twistn{\NStrand}^{-1}\twistnP{\Pct}\twistn{\NStrand} \twistn{\Strand} \twistn{\NStrand}^{-1}\twistnP{\Pct}^{-1}\twistn{\NStrand} \twistnP{\Pct}$ & $\twistnC{\nu}^{-1}\twistn{\NStrand}^{-1}\twistnC{\nu}\twistn{\NStrand} \twistn{\Strand} \twistn{\NStrand}^{-1}\twistnC{\nu}^{-1}\twistn{\NStrand} \twistnC{\nu}$
\\ \hline
$\twistn{\Str}$ & $\twistn{\Str}^{-1}\twistn{\NStrand}^{-1}\twistn{\Str} \twistn{\NStrand} \twistn{\Str}$ & - & -
\\ \hline
$\twistn{\Strr}$ & $\twistn{\Strr}$ & - & -
\\ \hline
$\twistnP{\Pc}$ & $\twistnP{\Pc}$ & $\twistnP{\Pct}^{-1}\twistn{\NStrand}^{-1}\twistnP{\Pct}\twistn{\NStrand} \twistnP{\Pc} \twistn{\NStrand}^{-1}\twistnP{\Pct}^{-1}\twistn{\NStrand} \twistnP{\Pct}$ & $\twistnC{\nu}^{-1}\twistn{\NStrand}^{-1}\twistnC{\nu}\twistn{\NStrand} \twistnP{\Pc} \twistn{\NStrand}^{-1}\twistnC{\nu}^{-1}\twistn{\NStrand} \twistnC{\nu}$
\\ \hline
$\twistnP{\Pct}$ & - & $\twistnP{\Pct}^{-1}\twistn{\NStrand}^{-1}\twistnP{\Pct}\twistn{\NStrand} \twistnP{\Pct}$ & -
\\ \hline
$\twistnP{\NPct}$ & - & $\twistnP{\NPct}$ & -
\\ \hline
$\twistnC{\mu}$ & $\twistnC{\mu}$ & $\twistnC{\mu}$ & $\twistnC{\nu}^{-1}\twistn{\NStrand}^{-1}\twistnC{\nu}\twistn{\NStrand} \twistnC{\mu}\twistn{\NStrand}^{-1}\twistnC{\nu}^{-1}\twistn{\NStrand} \twistnC{\nu}$
\\ \hline
$\twistnC{\nu}$ & - & - & $\twistnC{\nu}^{-1}\twistn{\NStrand}^{-1}\twistnC{\nu}\twistn{\NStrand} \twistnC{\nu}$
\\ \hline
$\twistnC{\omicron}$ & - & - & $\twistnC{\omicron}$
\end{tabular}
}
}
{\footnotesize with indices $1\leq\Strr<\Str<\Strand<\NStrand<\NNStrand<\TheStrand$, $1\leq\Pc<\Pct<\NPct\leq\ThePct$ and $1\leq\mu<\nu<\omicron\leq\TheCone$.}
\caption{Definition of the partial conjugations.}
\label{tab:partial_conj}
\end{table}

Given a word $W$ in the above alphabet, we can apply a sequence of various of the partial conjugations one after another. We call the resulting word $W'$ a \textit{partial conjugate} of $W$. For any set~$S$ of words in the alphabet $X$, let $\text{PC}(S)$ denote the \textit{set of their partial conjugates}. 
\end{definition}

For each $a\in\{\twist_{\NStrand\Str},\twistP_{\NStrand\Pct},\twistC_{\NStrand\nu}\}$, the assignments in Table \ref{tab:partial_conj} are chosen such that the partial conjugation $pc_a$ describes the conjugation with $a$ on the subgroup 
\[
(\freegrp{\TheStrand-1+\ThePct}\ast\freeprod)/\kernel\leq\PZ_\TheStrand(\Sigma_\freeprod(\ThePct))
\]
on the level of words in the alphabet $X$. More precisely, for each letter $x\in X$, the words $pc_a(x)$ satisfy $axa^{-1}=pc_a(x)$ in $\PZ_\TheStrand(\Sigma_\freeprod(\ThePct))$. 

\begin{observation}
\label{obs:partial_conj}
Since each partial conjugation maps $\twistnC{\nu}$ to a conjugate, these maps preserve the defining relations $\twistnC{\nu}^{\TheOrder_\nu}=1$ of $\freegrp{\TheStrand-1+\ThePct}\ast\freeprod$, i.e.\  these maps induce endomorphisms 
of $\freegrp{\TheStrand-1+\ThePct}\ast\freeprod$ that we also denote by $pc_a$ with $a=\twist_{\NStrand\Str},\twistP_{\NStrand\Pct}$ or $\twistC_{\NStrand\nu}$. 
By the definition of the partial conjugates, the following diagram commutes: 
\begin{center}
\begin{tikzcd}[column sep=60pt]
\freegrp{}(X) \arrow[d,"\pi"] \arrow[r,"pc_a"]& \freegrp{}(X) \arrow[d,"\pi"] & 
\\
\freegrp{\TheStrand-1+\ThePct}\ast\freeprod \arrow[d,"\iotaPZ"] \arrow[r,"pc_a"] & \freegrp{\TheStrand-1+\ThePct}\ast\freeprod \arrow[d,"\iotaPZ"] & 
\\
(\freegrp{\TheStrand-1+\ThePct}\ast\freeprod)/\kernel\arrow[r,"conj_a"] & (\freegrp{\TheStrand-1+\ThePct}\ast\freeprod)/\kernel & \hspace*{-65pt} \leq\PZ_\TheStrand(\Sigma_\freeprod(\ThePct)). 
\end{tikzcd}
\end{center}
Here above, $\pi$ denotes the projection induced by the identity on letters and $conj_a$ is the automorphism of $\PZ_\TheStrand(\Sigma_\freeprod(\ThePct))$ induced by conjugation with $a$. By definition, $\ker(\iotaPZ\circ\pi)=\langle\langle R\rangle\rangle$. The commutativity of the diagram implies 
\[
r\in\ker(\iotaPZ\circ\pi)=\langle\langle R\rangle\rangle\Rightarrow pc_a(r)\in\ker(\iotaPZ\circ\pi)=\langle\langle R\rangle\rangle, 
\]
i.e.\ if $r=1$ is a relation in $(\freegrp{\TheStrand-1+\ThePct}\ast\freeprod)/\kernel$, then the relation $pc_a(r)=1$ also holds in $(\freegrp{\TheStrand-1+\ThePct}\ast\freeprod)/\kernel$. Consequently, the relation $W'=1$ is contained in $\langle\langle R\rangle\rangle$ for each $W'$ in 
\[
\textup{PC}(\{(\twistn{\NStrand}\twistnC{\nu})^{\TheOrder_\nu}(\twistn{\NStrand}^{-1}\twistnC{\nu}^{-1})^{\TheOrder_\nu}\mid1\leq\NStrand<\TheStrand \text{ and } 1\leq\nu\leq\TheCone\}). 
\]
\end{observation}

Adding these relations to $R$, we obtain the following presentation for $\PZ_\TheStrand(\Sigma_\freeprod(\ThePct))$: 


\begin{proposition}
\label{prop:pure_br_semidir_pres}
The pure orbifold braid group $\PZ_\TheStrand(\Sigma_\freeprod(\ThePct))$ is generated by 
\[
\twist_{\Strand\Str},\twistP_{\NStrand\NPct},\twistC_{\NStrand\nu} \; \text{ and } \; \twistn{\NStrand},\twistnP{\NPct},\twistnC{\nu}
\]
for $1\leq\Str,\Strand,\NStrand<\TheStrand$ with $\Str<\Strand$, $1\leq\NPct\leq\ThePct$ and $1\leq\nu\leq\TheCone$, and the following defining relations: 
\begin{enumerate}[label={\textup{(R)}},ref={\textup{\thetheorem(R)}}]
\item 
\label{prop:pure_br_semidir_pres_rel7_pi_1}
$\twistnC{\nu}^{\TheOrder_\nu}=1$ for $1\leq\nu\leq\TheCone$ and 
\\
$W'=1$ for $W'\in\textup{PC}(\{(\twistn{\NStrand} \twistnC{\nu})^{\TheOrder_\nu}(\twistn{\NStrand}^{-1}\twistnC{\nu}^{-1})^{\TheOrder_\nu}\mid1\leq\NStrand<\TheStrand \text{ and } 1\leq\nu\leq\TheCone\})$. 
\end{enumerate}
For $1\leq\Str,\Strand,\NStrand,\NNStrand<\TheStrand$ with $\Str<\Strand<\NStrand<\NNStrand$,\; $1\leq\Pc,\NPct<\ThePct$ with $\Pc<\NPct$ and $1\leq\mu,\nu\leq\TheCone$ with $\mu<\nu$, the following relations hold: 
\begin{enumerate}[label={\textup{(S\arabic*)}},ref={\textup{\thetheorem(S\arabic*)}}]
\item 
\label{prop:pure_br_semidir_pres_rel1}
$\twistC_{\NStrand\nu}^{\TheOrder_\nu}=1$, 
\item 
\label{prop:pure_br_semidir_pres_rel2}
$[\twist_{\Strand\Str},\twist_{\NNStrand\NStrand}]=1$, 
$[\twistP_{\Strand\NPct},\twist_{\NNStrand\NStrand}]=1$, 
$[\twistC_{\Strand\nu},\twist_{\NNStrand\NStrand}]=1$,  
\item 
\label{prop:pure_br_semidir_pres_rel3}
$[\twist_{\NNStrand\Str},\twist_{\NStrand\Strand}]=1$,  
$[\twistP_{\NNStrand\NPct},\twist_{\NStrand\Strand}]=1$, 
$[\twistP_{\NNStrand\NPct},\twistP_{\NStrand\Pc}]=1$,  
$[\twistC_{\NNStrand\nu},\twist_{\NStrand\Strand}]=1$, 
$[\twistC_{\NNStrand\nu},\twistP_{\NStrand\NPct}]=1$,  
$[\twistC_{\NNStrand\nu},\twistC_{\NStrand\mu}]=1$, 
\item 
\label{prop:pure_br_semidir_pres_rel4}
$[\twist_{\NNStrand\NStrand}\twist_{\NNStrand\Strand}\twist_{\NNStrand\NStrand}^{-1},\twist_{\NStrand\Str}]=1$, 
$[\twist_{\NStrand\Strand}\twist_{\NStrand\Str}\twist_{\NStrand\Strand}^{-1},\twistP_{\Strand\NPct}]=1$, 
$[\twist_{\NStrand\Strand}\twistP_{\NStrand\Pc}\twist_{\NStrand\Strand}^{-1},\twistP_{\Strand\NPct}]=1$, 
\\
$[\twist_{\NStrand\Strand}\twist_{\NStrand\Str}\twist_{\NStrand\Strand}^{-1},\twistC_{\Strand\nu}]=1$, 
$[\twist_{\NStrand\Strand}\twistC_{\NStrand\mu}\twist_{\NStrand\Strand}^{-1},\twistC_{\Strand\nu}]=1$, 
$[\twist_{\NStrand\Strand}\twistP_{\NStrand\NPct}\twist_{\NStrand\Strand}^{-1},\twistC_{\Strand\nu}]=1$, 
\item 
\label{prop:pure_br_semidir_pres_rel5}
$\twist_{\Strand\Str}\twist_{\NStrand\Strand}\twist_{\NStrand\Str}=\twist_{\NStrand\Str}\twist_{\Strand\Str}\twist_{\NStrand\Strand}=\twist_{\NStrand\Strand}\twist_{\NStrand\Str}\twist_{\Strand\Str}$, 
\\
$\twist_{\Strand\Str}\twistP_{\Strand\NPct}\twistP_{\Str\NPct}=\twistP_{\Str\NPct}\twist_{\Strand\Str}\twistP_{\Strand\NPct}=\twistP_{\Strand\NPct}\twistP_{\Str\NPct}\twist_{\Strand\Str}$, 
\\
$\twist_{\Strand\Str}\twistC_{\Strand\nu}\twistC_{\Str\nu}=\twistC_{\Str\nu}\twist_{\Strand\Str}\twistC_{\Strand\nu}=\twistC_{\Strand\nu}\twistC_{\Str\nu}\twist_{\Strand\Str}$. 
\end{enumerate}
For $1\leq\Strr,\Str,\Strand,\NStrand,\NNStrand\leq\TheStrand$ with $\Strr<\Str<\Strand<\NStrand<\NNStrand$,\; $1\leq\Pc,\Pct,\NPct<\ThePct$ with $\Pc<\Pct<\NPct$ and $1\leq\mu,\nu,\omicron\leq\TheCone$ with $\mu<\nu<\omicron$, the following relations hold: 
\begin{enumerate}[label={\textup{(C)}},ref={\textup{\thetheorem(C)}}]
\item 
\label{prop:pure_br_semidir_pres_rel6_conj} 
\begin{enumerate}
\setlength{\itemsep}{-10pt}
\item \label{prop:pure_br_semidir_pres_rel6_conj_a}
$\twist_{\NNStrand\NStrand}\twistn{\Strand} \twist_{\NNStrand\NStrand}^{-1}=\twist_{\NNStrand\NStrand}^{-1}\twistn{\Strand}\twist_{\NNStrand\NStrand}=\twistn{\Strand}$, 
\\
\item \label{prop:pure_br_semidir_pres_rel6_conj_b}
$\twist_{\NNStrand\Strand}\twistn{\Strand} \twist_{\NNStrand\Strand}^{-1}=\twistn{\Strand}^{-1}\twistn{\NNStrand}^{-1}\twistn{\Strand} \twistn{\NNStrand} \twistn{\Strand}$, 
\\
\item \label{prop:pure_br_semidir_pres_rel6_conj_c}
$\twist_{\NNStrand\Strand}^{-1}\twistn{\Strand} \twist_{\NNStrand\Strand}=\twistn{\NNStrand} \twistn{\Strand} \twistn{\NNStrand}^{-1}$, 
\\
\item \label{prop:pure_br_semidir_pres_rel6_conj_d}
$\twist_{\NNStrand\Str}\twistn{\Strand} \twist_{\NNStrand\Str}^{-1}=\twistn{\Str}^{-1}\twistn{\NNStrand}^{-1}\twistn{\Str} \twistn{\NNStrand} \twistn{\Strand} \twistn{\NNStrand}^{-1}\twistn{\Str}^{-1}\twistn{\NNStrand} \twistn{\Str}, 
$
\\
\item \label{prop:pure_br_semidir_pres_rel6_conj_e}
$\twist_{\NNStrand\Str}^{-1}\twistn{\Strand} \twist_{\NNStrand\Str}=\twistn{\NNStrand} \twistn{\Str} \twistn{\NNStrand}^{-1}\twistn{\Str}^{-1}\twistn{\Strand} \twistn{\Str} \twistn{\NNStrand} \twistn{\Str}^{-1}\twistn{\NNStrand}^{-1}$, 
\\
\item \label{prop:pure_br_semidir_pres_rel6_conj_f}
$\twistP_{\NNStrand\NPct}\twistn{\Strand} \twistP_{\NNStrand\NPct}^{-1}=\twistnP{\NPct}^{-1}\twistn{\NNStrand}^{-1}\twistnP{\NPct}\twistn{\NNStrand} \twistn{\Strand} \twistn{\NNStrand}^{-1}\twistnP{\NPct}^{-1}\twistn{\NNStrand} \twistnP{\NPct}, 
$
\\
\item \label{prop:pure_br_semidir_pres_rel6_conj_g}
$\twistP_{\NNStrand\NPct}^{-1}\twistn{\Strand} \twistP_{\NNStrand\NPct}=\twistn{\NNStrand} \twistnP{\NPct}\twistn{\NNStrand}^{-1}\twistnP{\NPct}^{-1}\twistn{\Strand} \twistnP{\NPct}\twistn{\NNStrand} \twistnP{\NPct}^{-1}\twistn{\NNStrand}^{-1}$, 
\\
\item \label{prop:pure_br_semidir_pres_rel6_conj_h}
$\twistC_{\NNStrand\nu}\twistn{\Strand} \twistC_{\NNStrand\nu}^{-1}=\twistnC{\nu}^{-1}\twistn{\NNStrand}^{-1}\twistnC{\nu}\twistn{\NNStrand} \twistn{\Strand} \twistn{\NNStrand}^{-1}\twistnC{\nu}^{-1}\twistn{\NNStrand} \twistnC{\nu}, 
$
\\
\item \label{prop:pure_br_semidir_pres_rel6_conj_i}
$\twistC_{\NNStrand\nu}^{-1}\twistn{\Strand} \twistC_{\NNStrand\nu}=\twistn{\NNStrand} \twistnC{\nu}\twistn{\NNStrand}^{-1}\twistnC{\nu}^{-1}\twistn{\Strand} \twistnC{\nu}\twistn{\NNStrand} \twistnC{\nu}^{-1}\twistn{\NNStrand}^{-1}$, 
\\
\item \label{prop:pure_br_semidir_pres_rel6_conj_j}
$\twist_{\Strand\Str}\twistn{\Strand} \twist_{\Strand\Str}^{-1}=\twistn{\Str}^{-1}\twistn{\Strand} \twistn{\Str}$, 
\\
\item \label{prop:pure_br_semidir_pres_rel6_conj_k}
$\twist_{\Strand\Str}^{-1}\twistn{\Strand} \twist_{\Strand\Str}=\twistn{\Strand} \twistn{\Str} \twistn{\Strand} \twistn{\Str}^{-1}\twistn{\Strand}^{-1}$, 
\\
\item \label{prop:pure_br_semidir_pres_rel6_conj_l}
$\twistP_{\Strand\NPct}\twistn{\Strand} \twistP_{\Strand\NPct}^{-1}=\twistnP{\NPct}^{-1}\twistn{\Strand} \twistnP{\NPct}$, 
\\
\item \label{prop:pure_br_semidir_pres_rel6_conj_m}
$\twistP_{\Strand\NPct}^{-1}\twistn{\Strand} \twistP_{\Strand\NPct}=\twistn{\Strand} \twistnP{\NPct}\twistn{\Strand} \twistnP{\NPct}^{-1}\twistn{\Strand}^{-1}$, 
\\
\item \label{prop:pure_br_semidir_pres_rel6_conj_n}
$\twistC_{\Strand\nu}\twistn{\Strand} \twistC_{\Strand\nu}^{-1}=\twistnC{\nu}^{-1}\twistn{\Strand} \twistnC{\nu}$, 
\\
\item \label{prop:pure_br_semidir_pres_rel6_conj_o}
$\twistC_{\Strand\nu}^{-1}\twistn{\Strand} \twistC_{\Strand\nu}=\twistn{\Strand} \twistnC{\nu}\twistn{\Strand} \twistnC{\nu}^{-1}\twistn{\Strand}^{-1}$, 
\\
\item \label{prop:pure_br_semidir_pres_rel6_conj_p}
$\twist_{\Str\Strr}\twistn{\Strand} \twist_{\Str\Strr}^{-1}=\twist_{\Str\Strr}^{-1}\twistn{\Strand} \twist_{\Str\Strr}=\twistn{\Strand}$, 
\\
\item \label{prop:pure_br_semidir_pres_rel6_conj_q}
$\twistP_{\Str\NPct}\twistn{\Strand} \twistP_{\Str\NPct}^{-1}=\twistP_{\Str\NPct}^{-1}\twistn{\Strand} \twistP_{\Str\NPct}=\twistn{\Strand}$, 
\\
\item \label{prop:pure_br_semidir_pres_rel6_conj_r}
$\twistC_{\Str\nu}\twistn{\Strand} \twistC_{\Str\nu}^{-1}=\twistC_{\Str\nu}^{-1}\twistn{\Strand} \twistC_{\Str\nu}=\twistn{\Strand}$, 
\\
\item \label{prop:pure_br_semidir_pres_rel6_conj_s}
$\twist_{\Strand\Str}\twistnP{\NPct}\twist_{\Strand\Str}^{-1}=\twist_{\Strand\Str}^{-1}\twistnP{\NPct}\twist_{\Strand\Str}=\twistnP{\NPct}$, 
\\
\item \label{prop:pure_br_semidir_pres_rel6_conj_t}
$\twistP_{\Strand\Pc}\twistnP{\Pct}\twistP_{\Strand\Pc}^{-1}=\twistP_{\Strand\Pc}^{-1}\twistnP{\Pct}\twistP_{\Strand\Pc}=\twistnP{\Pct}$, 
\\
\item \label{prop:pure_br_semidir_pres_rel6_conj_u}
$\twistP_{\Strand\NPct}\twistnP{\NPct}\twistP_{\Strand\NPct}^{-1}=\twistnP{\NPct}^{-1}\twistn{\Strand}^{-1}\twistnP{\NPct}\twistn{\Strand} \twistnP{\NPct}$, 
\\
\item \label{prop:pure_br_semidir_pres_rel6_conj_v}
$\twistP_{\Strand\NPct}^{-1}\twistnP{\NPct}\twistP_{\Strand\NPct}=\twistn{\Strand} \twistnP{\NPct}\twistn{\Strand}^{-1}$, 
\\
\item \label{prop:pure_br_semidir_pres_rel6_conj_w}
$\twistP_{\Strand\NPct}\twistnP{\Pct}\twistP_{\Strand\NPct}^{-1}=\twistnP{\NPct}^{-1}\twistn{\Strand}^{-1}\twistnP{\NPct}\twistn{\Strand} \twistnP{\Pct}\twistn{\Strand}^{-1}\twistnP{\NPct}^{-1}\twistn{\Strand} \twistnP{\NPct}$, 
\\
\item \label{prop:pure_br_semidir_pres_rel6_conj_x}
$\twistP_{\Strand\NPct}^{-1}\twistnP{\Pct}\twistP_{\Strand\NPct}=\twistn{\Strand} \twistnP{\NPct}\twistn{\Strand}^{-1}\twistnP{\NPct}^{-1}\twistnP{\Pct}\twistnP{\NPct}\twistn{\Strand} \twistnP{\NPct}^{-1}\twistn{\Strand}^{-1}$, 
\\
\item \label{prop:pure_br_semidir_pres_rel6_conj_y}
$\twistC_{\Strand\nu}\twistnP{\NPct}\twistC_{\Strand\nu}^{-1}=\twistnC{\nu}^{-1}\twistn{\Strand}^{-1}\twistnC{\nu}\twistn{\Strand} \twistnP{\NPct}\twistn{\Strand}^{-1}\twistnC{\nu}^{-1}\twistn{\Strand} \twistnC{\nu}$, 
\\
\item \label{prop:pure_br_semidir_pres_rel6_conj_z}
$\twistC_{\Strand\nu}^{-1}\twistnP{\NPct}\twistC_{\Strand\nu}=\twistn{\Strand} \twistnC{\nu}\twistn{\Strand}^{-1}\twistnC{\nu}^{-1}\twistnP{\NPct}\twistnC{\nu}\twistn{\Strand} \twistnC{\nu}^{-1}\twistn{\Strand}^{-1}$, 
\end{enumerate}
\setlist[enumerate,2]{label=\textup{\alph*')}, ref=\alph*')}
\begin{enumerate}
\setlength{\itemsep}{-10pt}
\item \label{prop:pure_br_semidir_pres_rel6_conj_a'}
$\twist_{\Strand\Str}\twistnC{\nu}\twist_{\Strand\Str}^{-1}=\twist_{\Strand\Str}^{-1}\twistnC{\nu}\twist_{\Strand\Str}=\twistnC{\nu}$, 
\\
\item \label{prop:pure_br_semidir_pres_rel6_conj_b'}
$\twistP_{\Strand\NPct}\twistnC{\nu}\twistP_{\Strand\NPct}^{-1}=\twistP_{\Strand\NPct}^{-1}\twistnC{\nu}\twistP_{\Strand\NPct}=\twistnC{\nu}$, 
\\
\item \label{prop:pure_br_semidir_pres_rel6_conj_c'}
$\twistC_{\Strand\mu}\twistnC{\nu}\twistC_{\Strand\mu}^{-1}=\twistC_{\Strand\mu}^{-1}\twistnC{\nu}\twistC_{\Strand\mu}=\twistnC{\nu}$, 
\\
\item \label{prop:pure_br_semidir_pres_rel6_conj_d'}
$\twistC_{\Strand\nu}\twistnC{\nu}\twistC_{\Strand\nu}^{-1}=\twistnC{\nu}^{-1}\twistn{\Strand}^{-1}\twistnC{\nu}\twistn{\Strand} \twistnC{\nu}$, 
\\
\item \label{prop:pure_br_semidir_pres_rel6_conj_e'}
$\twistC_{\Strand\nu}^{-1}\twistnC{\nu}\twistC_{\Strand\nu}=\twistn{\Strand} \twistnC{\nu}\twistn{\Strand}^{-1}$, 
\\
\item \label{prop:pure_br_semidir_pres_rel6_conj_f'}
$\twistC_{\Strand\omicron}\twistnC{\nu}\twistC_{\Strand\omicron}^{-1}=\twistnC{\omicron}^{-1}\twistn{\Strand}^{-1}\twistnC{\omicron}\twistn{\Strand} \twistnC{\nu}\twistn{\Strand}^{-1}\twistnC{\omicron}^{-1}\twistn{\Strand} \twistnC{\omicron}$, 
\\
\item \label{prop:pure_br_semidir_pres_rel6_conj_g'}
$\twistC_{\Strand\omicron}^{-1}\twistnC{\nu}\twistC_{\Strand\omicron}=\twistn{\Strand} \twistnC{\omicron}\twistn{\Strand}^{-1}\twistnC{\omicron}^{-1}\twistnC{\nu}\twistnC{\omicron}\twistn{\Strand} \twistnC{\omicron}^{-1}\twistn{\Strand}^{-1}$. 
\end{enumerate}
\end{enumerate}
\end{proposition}
\setlist[enumerate,2]{label=\textup{\alph*)}, ref=\alph*)}
\begin{proof}
Recall that in Corollary \ref{cor:pres_pure_free_prod} we have found a presentation for $\PZ_\TheStrand(\Sigma_\freeprod(\ThePct))$. It remains to check that the group given by the above presentation is isomorphic. Therefore, we consider assignments $\map$ from the above generating set to the generating set from Corollary \ref{cor:pres_pure_free_prod}. 
For $1\leq\Str,\Strand,\NStrand<\TheStrand$ with $\Str<\Strand$, $1\leq\NPct\leq\ThePct$ and $1\leq\nu\leq\TheCone$, let $\map$ be defined by the following assignments: 
\begin{align*}
\twist_{\Strand\Str} & \mapsto \twist_{\Strand\Str}, & \twistn{\Strand} & \mapsto \twist_{\TheStrand\Strand}, 
\\
\twistP_{\NStrand\NPct} & \mapsto \twistP_{\NStrand\NPct}, \text{ and } & \twistnP{\NPct} & \mapsto  \twistP_{\TheStrand\NPct}, 
\\
\twistC_{\NStrand\nu} & \mapsto \twistC_{\NStrand\nu} &  \twistnC{\nu} & \mapsto \twistC_{\TheStrand\nu}. 
\end{align*}
We prove that $\map$ and its inverse induce homomorphisms. 

\begin{intermediate}[The map $\map$ induces a homomorphism]
The relations \ref{prop:pure_br_semidir_pres_rel1}-\ref{prop:pure_br_semidir_pres_rel5} are the relations of a subgroup $\PZ_{\TheStrand-1}(\Sigma_\freeprod(\ThePct))$. Hence, they are covered by relations \ref{cor:pres_pure_free_prod_rel1}-\ref{cor:pres_pure_free_prod_rel5}. 

Moreover, it follows from the proof of Corollary \ref{cor:pres_PMap_free_prod} that $\map$ preserves the relations from \ref{prop:pure_br_semidir_pres_rel6_conj}: based on Lemma \ref{lem:PMap_gens_sat_rels} this Corollary 
shows 
that the analogs of the relations from \ref{prop:pure_br_semidir_pres_rel6_conj} follow from the relations in the presentation of the pure subgroup $\PMapIdOrb{\TheStrand}{\Sigma_\freeprod(\ThePct)}$ from Corollary~\ref{cor:pres_PMap_free_prod}. Since the evaluation map $\ev$ restricts to a homomorphism from $\PMapIdOrb{\TheStrand}{\Sigma_\freeprod(\ThePct)}$ to $\PZ_\TheStrand(\Sigma_\freeprod(\ThePct))$, this implies that the analogous relations follow from the relations in the presentation of $\PZ_\TheStrand(\Sigma_\freeprod(\ThePct))$ from Corollary \ref{cor:pres_pure_free_prod}. 

For each $1\leq\nu\leq\TheCone$, the relation $\twistnC{\nu}^{\TheOrder_\nu}=1$ from \ref{prop:pure_br_semidir_pres_rel7_pi_1} maps to $\twistC_{\TheStrand\nu}^{\TheOrder_\nu}=1$ which is covered by relation \ref{cor:pres_pure_free_prod_rel1}. 
By Lemma \ref{lem:it_conj} and equation (\ref{eq:non-trivial_ker}), the relations from Corollary~\ref{cor:pres_pure_free_prod} also imply the relation $(\twist_{\TheStrand\NStrand}\twistC_{\TheStrand\nu})^{\TheOrder_\nu}=(\twistC_{\TheStrand\nu}\twist_{\TheStrand\NStrand})^{\TheOrder_\nu}$, i.e.\  the above assignments $\rho$ respect the relations $(\twistn{\NStrand} \twistnC{\nu})^{\TheOrder_\nu}(\twistn{\NStrand}^{-1}\twistnC{\nu}^{-1})^{\TheOrder_\nu}=1$ for each $\Strand$ and $\nu$. Furthermore, the partial conjugates of $(\twistn{\NStrand} \twistnC{\nu})^{\TheOrder_\nu}(\twistn{\NStrand}^{-1}\twistnC{\nu}^{-1})^{\TheOrder_\nu}$ map to $\PZ_\TheStrand(\Sigma_\freeprod(\ThePct))$-conjugates of $(\twist_{\TheStrand\NStrand}\twistC_{\TheStrand\nu})^{\TheOrder_\nu}(\twist_{\TheStrand\NStrand}^{-1}\twistC_{\TheStrand\nu}^{-1})^{\TheOrder_\nu}$. Hence, $\map$ also respects the relation $W'=1$ for each partial conjugate $W'$ of $(\twist_{\TheStrand\NStrand}\twistC_{\TheStrand\nu})^{\TheOrder_\nu}(\twist_{\TheStrand\NStrand}^{-1}\twistC_{\TheStrand\nu}^{-1})^{\TheOrder_\nu}$. By Theorem \ref{thm:von_Dyck}, this proves that $\rho$ induces a homomorphism. 
\end{intermediate}

\begin{intermediate}[The inverse map $\map^{-1}$ induces a homomorphism]
The relations \ref{cor:pres_pure_free_prod_rel1}-\ref{cor:pres_pure_free_prod_rel5} with every index $<\TheStrand$ map exactly to the relations \ref{prop:pure_br_semidir_pres_rel1}-\ref{prop:pure_br_semidir_pres_rel5}. 
Moreover, $\twistC_{\TheStrand\nu}^{\TheOrder_\nu}=1$ maps to $\twistnC{\nu}^{\TheOrder_\nu}=1$ from relation \ref{prop:pure_br_semidir_pres_rel7_pi_1}. That $\map^{-1}$ respects the remaining relations follows from relation \ref{prop:pure_br_semidir_pres_rel6_conj}.   For this observation, we give the $\map^{-1}$-image of the relations involving an index $\TheStrand$ and show how they fit into the conjugation relations from \ref{prop:pure_br_semidir_pres_rel6_conj}: 
\begin{align*}
[\twist_{\Strand\Str},\twistn{\NStrand}]=1 & \Leftrightarrow \twist_{\Strand\Str}\twistn{\NStrand} \twist_{\Strand\Str}^{-1}\stackrel{\ref{prop:pure_br_semidir_pres_rel6_conj_p}}=\twistn{\NStrand}, 
\\
[\twistP_{\Strand\NPct},\twistn{\NStrand}]=1 & \Leftrightarrow \twistP_{\Strand\NPct}\twistn{\NStrand} \twistP_{\Strand\NPct}^{-1}\stackrel{\ref{prop:pure_br_semidir_pres_rel6_conj_q}}=\twistn{\NStrand}, 
\\
[\twistC_{\Strand\nu},\twistn{\NStrand}]=1 & \Leftrightarrow \twistC_{\Strand\nu}\twistn{\NStrand} \twistC_{\Strand\nu}^{-1}\stackrel{\ref{prop:pure_br_semidir_pres_rel6_conj_r}}=\twistn{\NStrand}, 
\\
[\twistn{\Str},\twist_{\NStrand\Strand}]=1 & \Leftrightarrow \twist_{\NStrand\Strand}\twistn{\Str} \twist_{\NStrand\Strand}^{-1}\stackrel{\ref{prop:pure_br_semidir_pres_rel6_conj_a}}=\twistn{\Str}, 
\\
[\twistnP{\NPct},\twist_{\NStrand\Strand}]=1 & \Leftrightarrow \twist_{\NStrand\Strand}\twistnP{\NPct}\twist_{\NStrand\Strand}^{-1}\stackrel{\ref{prop:pure_br_semidir_pres_rel6_conj_s}}=\twistnP{\NPct}, 
\\
[\twistnP{\NPct},\twistP_{\NStrand\Pc}]=1 & \Leftrightarrow \twistP_{\NStrand\Pc}\twistnP{\NPct}\twistP_{\NStrand\Pc}^{-1}\stackrel{\ref{prop:pure_br_semidir_pres_rel6_conj_t}}=\twistnP{\NPct}, 
\\
[\twistnC{\nu},\twist_{\NStrand\Strand}]=1 & \Leftrightarrow \twist_{\NStrand\NStrand}\twistnC{\nu}\twist_{\NStrand\Strand}^{-1}\stackrel{\ref{prop:pure_br_semidir_pres_rel6_conj_a'}}=\twistnC{\nu}, 
\\
[\twistnC{\nu},\twistP_{\NStrand\NPct}]=1 & \Leftrightarrow \twistP_{\NStrand\NPct}\twistnC{\nu}\twistP_{\NStrand\NPct}^{-1}\stackrel{\ref{prop:pure_br_semidir_pres_rel6_conj_b'}}=\twistnC{\nu}, 
\\
[\twistC_{\Strand\mu},\twistnC{\nu}]=1 & \Leftrightarrow \twistC_{\Strand\mu}\twistnC{\nu}\twistC_{\Strand\mu}^{-1}\stackrel{\ref{prop:pure_br_semidir_pres_rel6_conj_c'}}=\twistnC{\nu}, 
\\
\twist_{\Strand\Str}\twistn{\Strand} \twistn{\Str}=\twistn{\Str} \twist_{\Strand\Str}\twistn{\Strand} & \Leftrightarrow \twist_{\Strand\Str}^{-1}\twistn{\Str} \twist_{\Strand\Str}\stackrel{\ref{prop:pure_br_semidir_pres_rel6_conj_c}}=\twistn{\Strand} \twistn{\Str} \twistn{\Strand}^{-1}, 
\\
\twist_{\Strand\Str}\twistn{\Strand} \twistn{\Str}=\twistn{\Strand} \twistn{\Str} \twist_{\Strand\Str} & \Leftrightarrow \twistn{\Strand} \twistn{\Str}\stackrel{\vee}=(\textcolor{\col}{\twistn{\Str}^{-1}\twistn{\Strand}\twistn{\Str}}\textcolor{\colo}{\twistn{\Str}^{-1}\twistn{\Strand}^{-1}\twistn{\Str})\twistn{\Strand}\twistn{\Str}}\stackrel{\ref{prop:pure_br_semidir_pres_rel6_conj_j},\ref{prop:pure_br_semidir_pres_rel6_conj_b}}=\twist_{\Strand\Str}\twistn{\Strand} \twistn{\Str} \twist_{\Strand\Str}^{-1}, 
\\
\twistnP{\NPct}\twistP_{\Str\NPct}\twistn{\Str}=\twistP_{\Str\NPct}\twistn{\Str} \twistnP{\NPct} & \Leftrightarrow \twistP_{\Str\NPct}^{-1}\twistnP{\NPct}\twistP_{\Str\NPct}\stackrel{\ref{prop:pure_br_semidir_pres_rel6_conj_v}}=\twistn{\Str} \twistnP{\NPct}\twistn{\Str}^{-1}, 
\\
\twistn{\Str} \twistnP{\NPct}\twistP_{\Str\NPct}=\twistnP{\NPct}\twistP_{\Str\NPct}\twistn{\Str} & \Leftrightarrow \twistnP{\NPct}^{-1}\twistn{\Str} \twistnP{\NPct}\stackrel{\ref{prop:pure_br_semidir_pres_rel6_conj_l}}=\twistP_{\Str\NPct}\twistn{\Str} \twistP_{\Str\NPct}^{-1}, 
\\
\twistn{\Str} \twistnC{\nu}\twistC_{\Str\nu}=\twistC_{\Str\nu}\twistn{\Str} \twistnC{\nu} & \Leftrightarrow \twistn{\Str} \twistnC{\nu}\stackrel{\vee}=(\textcolor{\col}{\twistnC{\nu}^{-1}\twistn{\Str} \twistnC{\nu}}\textcolor{\colo}{\twistnC{\nu}^{-1}\twistn{\Str}^{-1}\twistnC{\nu})\twistn{\Str} \twistnC{\nu}}\stackrel{\ref{prop:pure_br_semidir_pres_rel6_conj_n},\ref{prop:pure_br_semidir_pres_rel6_conj_d'}}=\twistC_{\Str\nu}\twistn{\Str} \twistnC{\nu}\twistC_{\Str\nu}^{-1}, 
\\
\twistn{\Str} \twistnC{\nu}\twistC_{\Str\nu}=\twistnC{\nu}\twistC_{\Str\nu}\twistn{\Str} & \Leftrightarrow \twistC_{\Str\nu}\twistn{\Str} \twistC_{\Str\nu}^{-1}\stackrel{\ref{prop:pure_br_semidir_pres_rel6_conj_n}}=\twistnC{\nu}^{-1}\twistn{\Str} \twistnC{\nu}, 
\end{align*}
\begin{align*}
[\twistn{\NStrand} \twistn{\Strand} \twistn{\NStrand}^{-1},\twist_{\NStrand\Str}]=1 & \Leftrightarrow \twistn{\NStrand} \twistn{\Strand} \twistn{\NStrand}^{-1} & \mystackrel{$\vee$}= & (\textcolor{\col}{\twistn{\Str}^{-1}\twistn{\NStrand} \twistn{\Str}} \textcolor{\colo}{\twistn{\Str}^{-1}\twistn{\NStrand}^{-1}\twistn{\Str})\twistn{\NStrand} \twistn{\Strand} \twistn{\NStrand}^{-1}(\twistn{\Str}^{-1}\twistn{\NStrand} \twistn{\Str}}\textcolor{\col}{\twistn{\Str}^{-1}\twistn{\NStrand}^{-1}\twistn{\Str}})
\\
&  & \mystackrel{\ref{prop:pure_br_semidir_pres_rel6_conj_j},\ref{prop:pure_br_semidir_pres_rel6_conj_d}}= & \twist_{\NStrand\Str}\twistn{\NStrand} \twistn{\Strand} \twistn{\NStrand}^{-1}\twist_{\NStrand\Str}^{-1}, 
\\
[\twistn{\Strand} \twistn{\Str} \twistn{\Strand}^{-1},\twistP_{\Strand\NPct}]=1 & \Leftrightarrow \twistn{\Strand} \twistn{\Str} \twistn{\Strand}^{-1} & \mystackrel{$\vee$}= & (\textcolor{\col}{\twistnP{\NPct}^{-1}\twistn{\Strand} \twistnP{\NPct}}\textcolor{\colo}{\twistnP{\NPct}^{-1}\twistn{\Strand}^{-1}\twistnP{\NPct})\twistn{\Strand} \twistn{\Str} \twistn{\Strand}^{-1}(\twistnP{\NPct}^{-1}\twistn{\Strand} \twistnP{\NPct}}\textcolor{\col}{\twistnP{\NPct}^{-1}\twistn{\Strand}^{-1}\twistnP{\NPct}})
\\
&  & \mystackrel{\ref{prop:pure_br_semidir_pres_rel6_conj_l},\ref{prop:pure_br_semidir_pres_rel6_conj_f}}= & \twistP_{\Strand\NPct}\twistn{\Strand} \twistn{\Str} \twistn{\Strand}^{-1}\twistP_{\Strand\NPct}^{-1}, 
\\
[\twistn{\Strand} \twistnP{\Pc}\twistn{\Strand}^{-1},\twistP_{\Strand\NPct}]=1 & \Leftrightarrow \twistn{\Strand} \twistnP{\Pc}\twistn{\Strand}^{-1} & \mystackrel{$\vee$}= & (\textcolor{\col}{\twistnP{\NPct}^{-1}\twistn{\Strand}\twistnP{\NPct}}\textcolor{\colo}{\twistnP{\NPct}^{-1}\twistn{\Strand}^{-1}\twistnP{\NPct})\twistn{\Strand}\twistnP{\Pc}\twistn{\Strand}^{-1}(\twistnP{\NPct}^{-1}\twistn{\Strand}\twistnP{\NPct}}\textcolor{\col}{\twistnP{\NPct}^{-1}\twistn{\Strand}^{-1}\twistnP{\NPct}})
\\
&  & \mystackrel{\ref{prop:pure_br_semidir_pres_rel6_conj_l},\ref{prop:pure_br_semidir_pres_rel6_conj_w}}= & \twistP_{\Strand\NPct}\twistn{\Strand} \twistnP{\Pc}\twistn{\Strand}^{-1}\twistP_{\Strand\NPct}^{-1},  
\\
[\twistn{\Strand} \twistn{\Str} \twistn{\Strand}^{-1},\twistC_{\Strand\nu}]=1 & \Leftrightarrow \twistn{\Strand} \twistn{\Str} \twistn{\Strand}^{-1} & \mystackrel{$\vee$}= & (\textcolor{\col}{\twistnC{\nu}^{-1}\twistn{\Strand}\twistnC{\nu}}\textcolor{\colo}{\twistnC{\nu}^{-1}\twistn{\Strand}^{-1}\twistnC{\nu})\twistn{\Strand}\twistn{\Str}\twistn{\Strand}^{-1}(\twistnC{\nu}^{-1}\twistn{\Strand}\twistnC{\nu}}\textcolor{\col}{\twistnC{\nu}^{-1}\twistn{\Strand}^{-1}\twistnC{\nu}})
\\
&  & \mystackrel{\ref{prop:pure_br_semidir_pres_rel6_conj_n},\ref{prop:pure_br_semidir_pres_rel6_conj_h}}= & \twistC_{\Strand\nu}\twistn{\Strand} \twistn{\Str} \twistn{\Strand}^{-1}\twistC_{\Strand\nu}^{-1}, 
\\
[\twistn{\Strand} \twistnC{\mu}\twistn{\Strand}^{-1},\twistC_{\Strand\nu}]=1 & \Leftrightarrow \twistn{\Strand} \twistnC{\mu} \twistn{\Strand}^{-1} & \mystackrel{$\vee$}= & (\textcolor{\col}{\twistnC{\nu}^{-1}\twistn{\Strand}\twistnC{\nu}}\textcolor{\colo}{\twistnC{\nu}^{-1}\twistn{\Strand}^{-1}\twistnC{\nu})\twistn{\Strand}\twistnC{\mu}\twistn{\Strand}^{-1}(\twistnC{\nu}^{-1}\twistn{\Strand}\twistnC{\nu}}\textcolor{\col}{\twistnC{\nu}^{-1}\twistn{\Strand}^{-1}\twistnC{\nu}})
\\
&  & \mystackrel{\ref{prop:pure_br_semidir_pres_rel6_conj_n},\ref{prop:pure_br_semidir_pres_rel6_conj_f'}}= & \twistC_{\Strand\nu}\twistn{\Strand} \twistnC{\mu}\twistn{\Strand}^{-1}\twistC_{\Strand\nu}^{-1}. 
\end{align*}
\end{intermediate}
This shows that $\map^{-1}$ induces an inverse homomorphism and $\PZ_\TheStrand(\Sigma_\freeprod(\ThePct))$ has the above presentation. 
\end{proof}

The proof of Proposition \ref{prop:pure_br_semidir_pres} in particular shows that the relations $W'=1$ for 
\[
W'\in\text{PC}(\{(\twistn{\NStrand}\twistnC{\nu})^{\TheOrder_\nu}(\twistn{\NStrand}^{-1}\twistnC{\nu}^{-1})^{\TheOrder_\nu}\mid1\leq\NStrand<\TheStrand \text{ and } 1\leq\nu\leq\TheCone\})
\]
and some relations from \ref{prop:pure_br_semidir_pres_rel6_conj} are not required to obtain a presentation of $\PZ_\TheStrand(\Sigma_\freeprod(\ThePct))$. But these relations are the key to prove the following: 

\begin{lemma}
\label{lem:semidir_prod_pres_PZ_n}
The presentation of $\PZ_\TheStrand(\Sigma_\freeprod(\ThePct))$ from Proposition \textup{\ref{prop:pure_br_semidir_pres}} satisfies the conditions from Lemma~\textup{\ref{lem:semidir_prod_pres_it2}}, i.e.\ $\PZ_\TheStrand(\Sigma_\freeprod(\ThePct))=\langle X\mid R\rangle\rtimes\PZ_{\TheStrand-1}(\Sigma_\freeprod(\ThePct))$ with 
\[
X=\{\twistn{1},...,\twistn{\TheStrand-1},\twistnP{1},...,\twistnP{\ThePct},\twistnC{1},...,\twistnC{\TheCone}\}
\]
and $R$ contains the relations from \textup{\ref{prop:pure_br_semidir_pres_rel7_pi_1}}.  
\end{lemma}
\begin{proof}
As described above the normal subgroup is presented by $\langle X\mid R\rangle$. The quotient is generated by 
\[
\twist_{\Strand\Str},\twistP_{\NStrand\NPct} \; \text{ and } \; \twistC_{\NStrand\nu}
\]
for $1\leq\Str,\Strand,\NStrand<\TheStrand,\Str<\Strand$, $1\leq\NPct\leq\ThePct$ and $1\leq\nu\leq\TheCone$ with defining relations from \ref{prop:pure_br_semidir_pres_rel1}-\ref{prop:pure_br_semidir_pres_rel5}, i.e.\ the presentation of $\PZ_{\TheStrand-1}(\Sigma_\freeprod(\ThePct))$ from Corollary~\ref{cor:pres_pure_free_prod}. Let $\langle Y\mid S\rangle$ denote this presentation of $\PZ_{\TheStrand-1}(\Sigma_\freeprod(\ThePct))$. Further, the presentation contains the relations in \ref{prop:pure_br_semidir_pres_rel6_conj} which are of the form $axa^{-1}=\phi_a(x)$ with 
\begin{align*}
x\in & \{\twistn{\NStrand},\twistnP{\NPct},\twistnC{\nu}\mid1\leq\NStrand<\TheStrand,1\leq\NPct\leq\ThePct,1\leq\nu\leq\TheCone\}\text{ and }
\\
a\in & \{\twist_{\Strand\Str},\twistP_{\NStrand\NPct},\twistC_{\NStrand\nu}\mid1\leq\Str,\Strand,\NStrand<\TheStrand, \Str<\Strand, 1\leq\NPct\leq\TheCone, 1\leq\nu\leq\TheCone\} & 
\end{align*} 
and $\phi_a(x)=pc_a(x)$ is a word in the alphabet $X$. It remains to check that this presentation satisfies the conditions from Lemma \ref{lem:semidir_prod_pres_it2}. That is, the assignments $x\mapsto\phi_a(x)$ induce an automorphism $\phi_a\in\Aut(\langle X\mid R\rangle)$ 
and the assignments $\phi:a\mapsto\phi_a$ induce a homomorphism $\PZ_{\TheStrand-1}(\Sigma_\freeprod(\ThePct))\rightarrow\Aut(\langle X\mid R\rangle)$. We follow the Steps \ref{rem:semidir_prod_pres_step1} and \ref{rem:semidir_prod_pres_step2} described in Remark \ref{rem:semidir_prod_pres}. 

\begin{step}
\label{lem:semidir_prod_pres_PZ_n_step1}
The assignments $\phi:a\mapsto\phi_a$ induce a homomorphism. 
\end{step}

This step requires to check that $\phi$ preserves the relations from \ref{prop:pure_br_semidir_pres_rel1}-\ref{prop:pure_br_semidir_pres_rel5}. 

For the relation $\twistC_{\NStrand\nu}^{\TheOrder_\nu}=1$ from \ref{prop:pure_br_semidir_pres_rel1}, we need to verify that $\phi_{\twistC_{\NStrand\nu}}^{\TheOrder_\nu}$ induces a trivial action on every generator. 
\new{On the basis of the relations from \ref{prop:pure_br_semidir_pres_rel6_conj},} Lemma~\ref{lem:it_conj} implies 
\begin{align*}
\twistC_{\NStrand\nu}^{\TheOrder_\nu}\twistn{\NNStrand} \twistC_{\NStrand\nu}^{-\TheOrder_\nu} & =\twistn{\NNStrand}, 
\\
\twistC_{\NStrand\nu}^{\TheOrder_\nu}\twistn{\NStrand} \twistC_{\NStrand\nu}^{-\TheOrder_\nu} & =(\twistnC{\nu}^{-1}\twistn{\NStrand}^{-1})^{\TheOrder_\nu}\twistn{\NStrand}(\twistn{\NStrand} \twistnC{\nu})^{\TheOrder_\nu}, 
\\
\twistC_{\NStrand\nu}^{\TheOrder_\nu}\twistn{\Strand} \twistC_{\NStrand\nu}^{-\TheOrder_\nu} & =(\twistnC{\nu}^{-1}\twistn{\NStrand}^{-1})^{\TheOrder_\nu}(\twistnC{\nu}\twistn{\NStrand})^{\TheOrder_\nu}\twistn{\Strand}(\twistn{\NStrand}^{-1}\twistnC{\nu}^{-1})^{\TheOrder_\nu}(\twistn{\NStrand} \twistnC{\nu})^{\TheOrder_\nu}, 
\\
\twistC_{\NStrand\nu}^{\TheOrder_\nu}\twistnP{\NPct}\twistC_{\NStrand\nu}^{-\TheOrder_\nu} & =(\twistnC{\nu}^{-1}\twistn{\NStrand}^{-1})^{\TheOrder_\nu}(\twistnC{\nu}\twistn{\NStrand})^{\TheOrder_\nu}\twistnP{\NPct}(\twistn{\NStrand}^{-1}\twistnC{\nu}^{-1})^{\TheOrder_\nu}(\twistn{\NStrand} \twistnC{\nu})^{\TheOrder_\nu}, 
\\
\twistC_{\NStrand\nu}^{\TheOrder_\nu}\twistnC{\mu}\twistC_{\NStrand\nu}^{-\TheOrder_\nu} & =(\twistnC{\nu}^{-1}\twistn{\NStrand}^{-1})^{\TheOrder_\nu}(\twistnC{\nu}\twistn{\NStrand})^{\TheOrder_\nu}\twistnC{\mu}(\twistn{\NStrand}^{-1}\twistnC{\nu}^{-1})^{\TheOrder_\nu}(\twistn{\NStrand} \twistnC{\nu})^{\TheOrder_\nu}, 
\\
\twistC_{\NStrand\nu}^{\TheOrder_\nu}\twistnC{\nu}\twistC_{\NStrand\nu}^{-\TheOrder_\nu} & =(\twistnC{\nu}^{-1}\twistn{\NStrand}^{-1})^{\TheOrder_\nu}\twistnC{\nu}(\twistn{\NStrand} \twistnC{\nu})^{\TheOrder_\nu}, 
\\
\twistC_{\NStrand\nu}^{\TheOrder_\nu}\twistnC{\omicron}\twistC_{\NStrand\nu}^{-\TheOrder_\nu} & =\twistnC{\omicron}. 
\end{align*}
The relations $(\twistnC{\nu}\twistn{\NStrand})^{\TheOrder_\nu}(\twistnC{\nu}^{-1}\twistn{\NStrand}^{-1})^{\TheOrder_\nu}=1$ for $1\leq\NStrand\leq\TheStrand$ and $1\leq\nu\leq\TheCone$ from \ref{prop:pure_br_semidir_pres_rel7_pi_1} imply that 
$\phi_{\twistC_{\NStrand\nu}}^{\TheOrder_\nu}(x)=x$ for  each $x\in X$. 

Further, it remains to check that $\phi$ preserves the relations \ref{prop:pure_br_semidir_pres_rel2}-\ref{prop:pure_br_semidir_pres_rel5} of $\PZ_{\TheStrand-1}(\Sigma_\freeprod(\ThePct))$. Therefore, we recall 
that the evaluation map $\ev$ for each $\TheStrand\in\NN$ induces a homomorphism 
\[
\PMapIdOrb{\TheStrand}{\Sigma_\freeprod(\ThePct)}\rightarrow\PZ_\TheStrand(\Sigma_\freeprod(\ThePct)). 
\]
Hence, the evaluation map $\ev$ preserves the relations from $\PMapIdOrb{\TheStrand}{\Sigma_\freeprod(\ThePct)}$. 
In particular, we emphasize that the relations \ref{prop:pure_br_semidir_pres_rel2}-\ref{prop:pure_br_semidir_pres_rel5} of $\PZ_{\TheStrand-1}(\Sigma_\freeprod(\ThePct))$ have corresponding relations in $\PMapIdOrb{\TheStrand-1}{\Sigma_\freeprod(\ThePct)}$ (see \ref{cor:pres_PMap_free_prod_rel1}-\ref{cor:pres_PMap_free_prod_rel4}). 

Now we recall from Corollary \ref{cor:pure_orb_mcg_ses} that $\PMapIdOrb{\TheStrand}{\Sigma_\freeprod(\ThePct)}$ is a semidirect product $\freegrp{\TheStrand-1+\ThePct+\TheCone}\rtimes\PMapIdOrb{\TheStrand-1}{\Sigma_\freeprod(\ThePct)}$. Hence, Lemma \ref{lem:semidir_prod_pres_it2} in particular implies that the conjugation relations from Lemma \ref{lem:PMap_gens_sat_rels} induce assignments $\psi:A\mapsto\psi_A$ that yield a homomorphism $\psi:\PMapIdOrb{\TheStrand-1}{\Sigma_\freeprod(\ThePct)}\rightarrow\Aut(\freegrp{\TheStrand-1+\ThePct+\TheCone})$. 
That is, the assignments $\psi:A\mapsto\psi_A$ preserve the defining relations of $\PMapIdOrb{\TheStrand-1}{\Sigma_\freeprod(\ThePct)}$ from Corollary \ref{cor:pres_PMap_free_prod}. 

Since $\ev:\PMapIdOrb{\TheStrand}{\Sigma_\freeprod(\ThePct)}\rightarrow\PZ_\TheStrand(\Sigma_\freeprod(\ThePct))$ is a homomorphism, the assignments of $\ev$ in particular preserve the relations from Lemma~\ref{lem:PMap_gens_sat_rels}. 
The $\ev$-images of the relations from Lemma \ref{lem:PMap_gens_sat_rels} are the relations from \ref{prop:pure_br_semidir_pres_rel6_conj}. Thus, the assignments $\phi:a\mapsto\phi_a$ induced by the relations in \ref{prop:pure_br_semidir_pres_rel6_conj} correspond to the assignments $\psi:A\mapsto\psi_A$ induced by Lemma \ref{lem:PMap_gens_sat_rels}. 

Finally, we obtain that the assignments $\phi$ correspond to the assignments $\psi$ and $\psi$ preserves the relations that correspond to \ref{prop:pure_br_semidir_pres_rel2}-\ref{prop:pure_br_semidir_pres_rel5}. As a consequence, the fact that $(\freegrp{\TheStrand-1+\ThePct}\ast\freeprod)/\kernel$ is a quotient of $\freegrp{\TheStrand-1+\ThePct+\TheCone}$ allows us to deduce that $\phi$ preserves the relations \ref{prop:pure_br_semidir_pres_rel2}-\ref{prop:pure_br_semidir_pres_rel5}. This finishes the proof of Step \ref{lem:semidir_prod_pres_PZ_n_step1} and shows that the assignments $\phi:a\mapsto\phi_a$ induce a homomorphism. 

\begin{step}
\label{lem:semidir_prod_pres_PZ_n_step2}
The assignments $\phi_a:x\mapsto\phi_a(x)$ induce an automorphism in $\Aut(\langle X\mid R\rangle)$. 
\end{step}


\new{To observe that $\phi_a$ preserves the relations $\twistnC{\nu}^{\TheOrder_\nu}=1$,} we recall that the conjugation relations from \ref{prop:pure_br_semidir_pres_rel6_conj} identify each $\PZ_{\TheStrand-1}(\Sigma_\freeprod(\ThePct))$-conjugate of $\twistnC{\nu}$ with an $\langle X\mid R\rangle$-conjugate of $\twistnC{\nu}$. This in particular implies that $\phi_a(\twistnC{\nu}^{\TheOrder_\nu})=\phi_a(\twistnC{\nu})^{\TheOrder_\nu}=1$. 

Further, recall that the map $\phi_a$ 
describes the conjugations $conj_a$ 
on the level of words in the alphabet $X$, i.e.\ this map coincides with the partial conjugation $pc_a$. 
Hence, Observation \ref{obs:partial_conj} shows that the set of partial conjugates 
\[
\text{PC}(\{(\twistn{\NStrand}\twistnC{\nu})^{\TheOrder_\nu}(\twistn{\NStrand}^{-1}\twistnC{\nu}^{-1})^{\TheOrder_\nu}\mid1\leq\NStrand<\TheStrand \text{ and } 1\leq\nu\leq\TheCone\}) 
\]
is invariant under $\phi_a$. 
This implies, if $W'$ is a word from this set of partial conjugates, the relation $W'=1$ maps to $pc_a(W')=1$ 
which is also covered by the relations in $R$. 
Hence, $\phi_a$ induces an $\langle X\mid R\rangle$
-automorphism for every $a\in\PZ_{\TheStrand-1}(\Sigma_\freeprod(\ThePct))$, which was claimed in Step \ref{lem:semidir_prod_pres_PZ_n_step2}. 

Finally, the presentation from Proposition \ref{prop:pure_br_semidir_pres} satisfies the conditions from \linebreak Lemma~\ref{lem:semidir_prod_pres_it2}. Thus, we obtain that $\PZ_\TheStrand(\Sigma_\freeprod(\ThePct))$ is a semidirect product 
\[
\langle X\mid R\rangle\rtimes\PZ_{\TheStrand-1}(\Sigma_\freeprod(\ThePct)). 
\]
This proves the claim. 
\end{proof}

Lemma \ref{lem:semidir_prod_pres_PZ_n} also allows us to determine $\kernel$. 

\begin{proposition}
The kernel $\kernel=\kernel_\TheStrand$ of $\iotaPZ$ is given by the normal closure of 
\[
\textup{PC}(\{(\twistn{\NStrand} \twistnC{\nu})^{\TheOrder_\nu}(\twistn{\NStrand}^{-1}\twistnC{\nu}^{-1})^{\TheOrder_\nu}\mid1\leq\NStrand<\TheStrand \; \text{ and } \; 1\leq\nu\leq\TheCone\}) 
\]
inside $\piOrb\left(\Sigma_\freeprod(\TheStrand-1+\ThePct)\right)\cong\freegrp{\TheStrand-1+\ThePct}\ast\freeprod$. 
\end{proposition}
\begin{proof}
By Corollary \ref{cor:pure_orb_braid_semidir_prod}, the group $\PZ_\TheStrand(\Sigma_\freeprod(\ThePct))$ is a semidirect product 
\[
((\freegrp{\TheStrand-1+\ThePct}\ast\freeprod)/\kernel)\rtimes\PZ_{\TheStrand-1}(\Sigma_\freeprod(\ThePct)). 
\]
Further, Lemma \ref{lem:semidir_prod_pres_PZ_n} implies that $\PZ_\TheStrand(\Sigma_\freeprod(\ThePct))$ decomposes as a semidirect product with normal subgroup $\langle X\mid R\rangle$ with 
\[
X=\{\twistn{1},...,\twistn{\TheStrand-1},\twistnP{1},...,\twistnP{\ThePct},\twistnC{1},...,\twistnC{\TheCone}\} 
\]
and defining relations $R$ from Proposition \ref{prop:pure_br_semidir_pres_rel7_pi_1}. 
This implies that the subgroup of $\PZ_\TheStrand(\Sigma_\freeprod(\ThePct))$ generated by $\twistn{1},...,\twistn{\TheStrand-1},\twistnP{1},...,\twistnP{\ThePct}$ and $\twistnC{1},...,\twistnC{\TheCone}$ is isomorphic to both $(\freegrp{\TheStrand-1+\ThePct}\ast\freeprod)/\kernel$ and the group with presentation $\langle X\mid R\rangle$, i.e.\ $(\freegrp{\TheStrand-1+\ThePct}\ast\freeprod)/\kernel$ has the presentation $\langle X\mid R\rangle$. In particular, we obtain that the kernel $\kernel$ is the normal closure of the set 
\[
\text{PC}(\{(\twistn{\NStrand}\twistnC{\nu})^{\TheOrder_\nu}(\twistn{\NStrand}^{-1}\twistnC{\nu}^{-1})^{\TheOrder_\nu}\mid1\leq\NStrand<\TheStrand \text{ and } 1\leq\nu\leq\TheCone\}) 
\]
inside $\freegrp{\TheStrand-1+\ThePct}\ast\freeprod$. 
\end{proof}

\newpage

\new{Finally,} we have proven Theorem \ref{thm-intro:kernel_ex_seq}: 

\begin{corollary}
\label{cor:pure_orb_braid_semidir_prod_with_K}
The pure braid group $\PZ_\TheStrand(\Sigma_\freeprod(\ThePct))$ fits into the following exact sequence of pure orbifold braid groups 
\[
1\rightarrow\kernel_\TheStrand\rightarrow\piOrb\left(\Sigma_\freeprod(\TheStrand-1+\ThePct)\right)\xrightarrow{\iotaPZ}\PZ_\TheStrand(\Sigma_\freeprod(\ThePct))\xrightarrow{\piPZ}\PZ_{\TheStrand-1}(\Sigma_\freeprod(\ThePct))\rightarrow1 
\]
with \new{$\piOrb\left(\Sigma_\freeprod(\TheStrand-1+\ThePct)\right)\cong\freegrp{\TheStrand-1+\ThePct}\ast\freeprod$ and} 
\begin{equation}
\label{eq:ker_n}
\kernel_\TheStrand=\langle\langle\textup{PC}(\{(\twistn{\NStrand} \twistnC{\nu})^{\TheOrder_\nu}(\twistn{\NStrand}^{-1}\twistnC{\nu}^{-1})^{\TheOrder_\nu}\mid 1\leq\NStrand<\TheStrand,1\leq\nu\leq\TheCone\})\rangle\rangle_{\new{\freegrp{\TheStrand-1+\ThePct}\ast\freeprod}}. 
\end{equation}
Moreover, the homomorphism $\piPZ$ has a right inverse homomorphism $\sPZ$, i.e.\ the group $\PZ_\TheStrand(\Sigma_\freeprod(\ThePct))$ is a semidirect product $((\freegrp{\TheStrand-1+\ThePct}\ast\freeprod)/\kernel_\TheStrand)\rtimes\PZ_{\TheStrand-1}(\Sigma_\freeprod(\ThePct))$. 
\end{corollary}

Corollary \ref{cor:pure_orb_braid_semidir_prod_with_K} corrects Theorem 2.14 in \cite{Roushon2021}. This raises the natural question if the consequences described in \cite{Roushon2021} still hold. In particular, Roushon deduced that 
\begin{itemize}
\item $\PZ_\TheStrand(\Sigma_\freeprod(\ThePct))$ and consequently the affine Artin groups of types $\tilde{A}_\TheStrand, \tilde{B}_\TheStrand,\tilde{C}_\TheStrand$ and $\tilde{D}_\TheStrand$, and the braid groups of finite complex type $G(de,e,r)$ for $d,r\geq2$ are virtually poly-free, see \cite[Theorem 2.19]{Roushon2021}.
\item $\Z_\TheStrand(\Sigma_\freeprod(\ThePct))$ and consequently the affine Artin groups of type $\tilde{D}_\TheStrand$ satisfy the Farrell--Jones isomorphism conjecture, see \cite[Theorem 2.20]{Roushon2021}. 
\end{itemize}

\enew{By \cite{Allcock2002}, the Artin group of type $\tilde{A}_\TheStrand$ embeds into the Artin group of type $B_{\TheStrand+1}$. The same article shows that the Artin group of type $\tilde{C}_\TheStrand$ embeds into the Artin group of type $A_\TheStrand$. For these spherical Artin groups it is known that they are virtually poly-free, see \cite{Brieskorn1973}. Since the property of being virtually poly-free passes to subgroups, this settles the questions for Artin groups of type $\tilde{A}_\TheStrand$ and $\tilde{C}_\TheStrand$. Since the complex braid groups of type $G(de,e,r)$ by \cite[Proposition 4.1]{CorranLeeLee2015} are isomorphic to a subgroup of $A(B_r)$, they are also virtually poly-free. Moreover, to the best of the author's knowledge, it remains open for the Artin groups of type $\tilde{B}_\TheStrand$ and $\tilde{D}_\TheStrand$ whether they are virtually poly-free. 

For the Artin group of type $\tilde{D}_\TheStrand$, it remains open whether they satisfy the Farrell--Jones isomorphism conjecture. The question whether the Artin group of type $\tilde{B}_\TheStrand$ satisfies the Farrell--Jones isomorphism conjecture is treated in \cite{Roushon2021a}. But, as we will point out below, Proposition~4.1 in that article is not correct. Proposition~4.1 seems to be essential for proving \cite[Theorems 3.2 and 3.3]{Roushon2021a} which yield that the Artin group of type $\tilde{B}_\TheStrand$ satisfies the Farrell--Jones isomorphism conjecture. To the best of the author's knowledge, this question is also open.} 

\subsection*{Fixed strands and punctures}

At this point, a remark about fixed strands and punctures is in order. Instead of considering the group $\Z_\TheStrand(\Sigma_\freeprod(\ThePct))$ of braids in 
an orbifold with punctures in $\freeprod(\{r_1,...,r_\ThePct\})$, we also could have considered a subgroup $\Z_{\TheStrand+\ThePct}^{\text{fix}(\ThePct)}(\Sigma_\freeprod)$ of $\Z_{\TheStrand+\ThePct}(\Sigma_\freeprod)$ where $\ThePct$ strands do not move. More precisely, each element in the subgroup $\Z_{\TheStrand+\ThePct}^{\text{fix}(\ThePct)}(\Sigma_\freeprod)$ is represented by a braid that ends in the points $r_1,...,r_\ThePct,p_1,...,p_\TheStrand$ such that the strands that end in positions $r_1,...,r_\ThePct$ are constant. 

However, using similar observations as above, it turns out that the group $\Z_{\TheStrand+\ThePct}^{\text{fix}(\ThePct)}(\Sigma_\freeprod)$ differs from $\Z_\TheStrand(\Sigma_\freeprod(\ThePct))$. On the level of braid diagrams this is reflected by the following observation: 
\begin{figure}[H]
\centerline{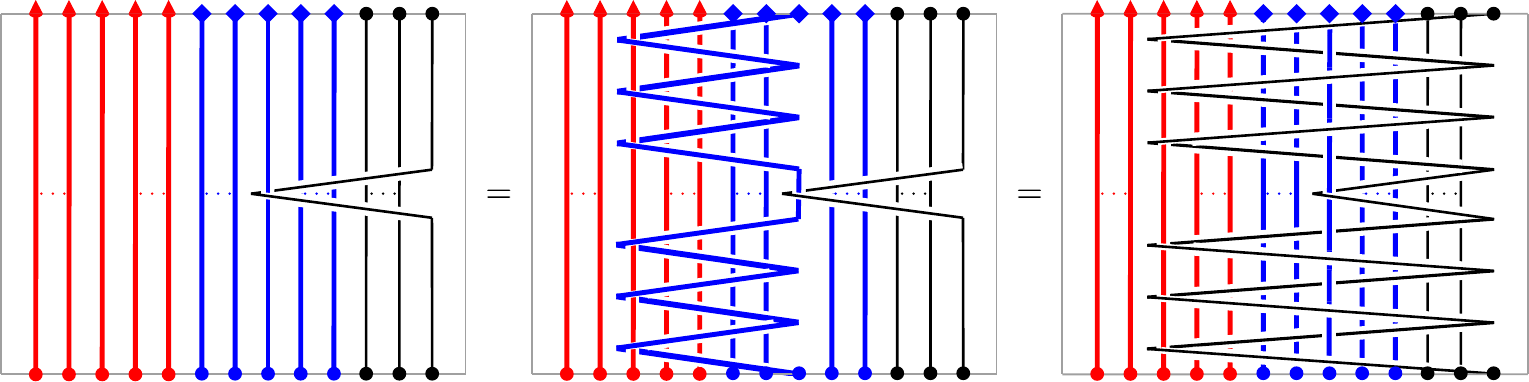}
\caption{A relation in $\Z_{\TheStrand+\ThePct}^{\text{fix}(\ThePct)}(\Sigma_\freeprod)$.}
\label{fig:add_relation}
\end{figure}
In the group $\Z_{\TheStrand+\ThePct}^{\text{fix}(\ThePct)}(\Sigma_\freeprod)$ the relation described in Figure \ref{fig:add_relation} holds. In contrast, the group $\Z_\TheStrand(\Sigma_\freeprod(\ThePct))$ does not allow the transformation described in Figure~\ref{fig:add_relation}. Since we require that the strand that ends in position $r_\NPct$ is fixed by every representative, the braid in the middle of Figure \ref{fig:add_relation} is not contained in $\Z_\TheStrand(\Sigma_\freeprod(\ThePct))$. 

To make this precise, let us consider the pure subgroup $\PZ_{\TheStrand+\ThePct}^{\text{fix}(\ThePct)}(\Sigma_\freeprod)$ of $\Z_{\TheStrand+\ThePct}^{\text{fix}(\ThePct)}(\Sigma_\freeprod)$. By the same arguments as in Corollary \ref{cor:gen_set_PZ_n}, we obtain a generating set for this group. To underline the similarities with $\PZ_\TheStrand(\Sigma_\freeprod(\ThePct))$ let us denote the generating set of this group by braids 
\[
\twist_{\Strand\Str},\twistP_{\NStrand\NPct} \; \text{ and } \; \twistC_{\NStrand\nu} 
\]
for $1\leq\Str,\Strand,\NStrand\leq\TheStrand$ with $\Str<\Strand$, \; $1\leq\NPct\leq\ThePct$ and $1\leq\nu\leq\TheCone$ with braid diagrams as in Figures \ref{fig:a_ji} and \ref{fig:a_kc_nu_a_kr_lambda}. Now let $\omega:\Z_\TheStrand(\Sigma_\freeprod(\ThePct))\rightarrow\Z_{\TheStrand+\ThePct}^{\text{fix}(\ThePct)}(\Sigma_\freeprod)\leq\Z_{\TheStrand+\ThePct}(\Sigma_\freeprod)$ be the homomorphism that is induced by sending the punctures in positions $r_1,...,r_\ThePct$ to fixed strands. Concerning this homomorphism, the relation described in Figure~\ref{fig:add_relation} implies: 

\begin{proposition}
\label{prop:emb_orb_braid_grps_not_inj}
\new{For $\ThePct\geq1$,} the homomorphism $\omega$ is not injective.  
\end{proposition}
\begin{proof}
For the proof, we consider the restriction 
\[
\omega\vert_{\PZ_\TheStrand(\Sigma_\freeprod(\ThePct))}:\PZ_\TheStrand(\Sigma_\freeprod(\ThePct))\rightarrow\PZ_{\TheStrand+\ThePct}^{\text{fix}(\ThePct)}(\Sigma_\freeprod)\leq\PZ_{\TheStrand+\ThePct}(\Sigma_\freeprod). 
\]
We will prove that this restriction is not injective. For this purpose, we use the following exact sequence obtained from Corollary \ref{cor:pure_orb_braid_semidir_prod_with_K}: 
\[
1\rightarrow\kernel_{\TheStrand+\ThePct}\rightarrow\freegrp{\TheStrand-1+\ThePct}\ast\freeprod\xrightarrow{\iotaPZ}\PZ_{\TheStrand+\ThePct}(\Sigma_\freeprod)\xrightarrow{\piPZ}\PZ_{\TheStrand-1+\ThePct}(\Sigma_\freeprod)\rightarrow1 
\]
with 
\begin{equation}
\label{eq:ker_n+ThePct}
\kernel_{\TheStrand+\ThePct}=\left\langle\left\langle\text{PC}\left(\left\lbrace\begin{matrix}
(\twistn{\NStrand} \twistnC{\nu})^{\TheOrder_\nu}(\twistn{\NStrand}^{-1}\twistnC{\nu}^{-1})^{\TheOrder_\nu},
\\
(\twistnP{\NPct} \twistnC{\nu})^{\TheOrder_\nu}(\twistnP{\NPct}^{-1}\twistnC{\nu}^{-1})^{\TheOrder_\nu}
\end{matrix}
\biggm\vert \begin{matrix}
1\leq\NStrand<\TheStrand,1\leq\nu\leq\TheCone
\\
1\leq\NPct\leq\ThePct,1\leq\nu\leq\TheCone
\end{matrix}
\right\rbrace\right)\right\rangle\right\rangle_{\freegrp{\TheStrand-1+\ThePct}\ast\freeprod}. 
\end{equation}
Restricting this exact sequence to the subgroup $\PZ_{\TheStrand+\ThePct}^{\text{fix}(\ThePct)}(\Sigma_\freeprod)$, we further obtain: 
\begin{equation}
\label{eq:ex_seq_PZ_fix}
1\rightarrow\kernel_{\TheStrand+\ThePct}\rightarrow\freegrp{\TheStrand-1+\ThePct}\ast\freeprod\xrightarrow{\iotaPZ}\PZ_{\TheStrand+\ThePct}^{\text{fix}(\ThePct)}(\Sigma_\freeprod)\xrightarrow{\piPZ}\PZ_{\TheStrand-1+\ThePct}^{\text{fix}(\ThePct)}(\Sigma_\freeprod)\rightarrow1. 
\end{equation}

Restricting the homomorphism $\mathrm{s}_{\PZ_{\TheStrand+\ThePct}}:\PZ_{\TheStrand-1+\ThePct}(\Sigma_\freeprod)\rightarrow\PZ_{\TheStrand+\ThePct}(\Sigma_\freeprod)$ as defined on page \pageref{def:right-inverse_sPZ}, 
yields a section of the homomorphism $\PZ_{\TheStrand+\ThePct}^{\text{fix}(\ThePct)}(\Sigma_\freeprod)\rightarrow\PZ_{\TheStrand-1+\ThePct}^{\text{fix}(\ThePct)}(\Sigma_\freeprod)$. Thus, the group $\PZ_{\TheStrand+\ThePct}^{\text{fix}(\ThePct)}(\Sigma_\freeprod)$ is the semidirect product 
\[
(\freegrp{\TheStrand-1+\ThePct}\ast\freeprod)/\kernel_{\TheStrand+\ThePct}\rtimes\PZ_{\TheStrand-1+\ThePct}^{\text{fix}(\ThePct)}(\Sigma_\freeprod). 
\]
In comparison, the group $\PZ_\TheStrand(\Sigma_\freeprod(\ThePct))$ is the semidirect product 
\[
(\freegrp{\TheStrand-1+\ThePct}\ast\freeprod)/\kernel_\TheStrand\rtimes\PZ_{\TheStrand-1}(\Sigma_\freeprod(\ThePct)). 
\]
To deduce that the two groups do not coincide, it remains to check the following: 
\begin{claim*}
$\kernel_\TheStrand\subsetneq\kernel_{\TheStrand+\ThePct}$. 
\end{claim*}

On the one hand, we have $\kernel_\TheStrand$ from \eqref{eq:ker_n} and $\kernel_{\TheStrand+\ThePct}$ as in \eqref{eq:ker_n+ThePct}. This directly implies that $\kernel_\TheStrand$ is contained in $\kernel_{\TheStrand+\ThePct}$. 

On the other hand, let us fix $1\leq\Pc\leq\ThePct$ and $1\leq\omicron\leq\TheCone$ and let $\ZZ\ast\cyc{\TheOrder_{\omicron}}$ be the free product with $\ZZ=\langle\twistnP{\Pc}\rangle$ and $\cyc{\TheOrder_\omicron}=\langle\twistnC{\omicron}\rangle$. The following assignments induce a homomorphism: 
\begin{align*}
q_{\Pc,\omicron}:\freegrp{\TheStrand-1+\ThePct}\ast\freeprod & \rightarrow\ZZ\ast\cyc{\TheOrder_\nu},
\\
\twistnP{\Pc} & \mapsto \twistnP{\Pc},
\\
\twistnC{\omicron} & \mapsto \twistnC{\omicron},
\\
x & \mapsto 1 \quad \text{ for the remaining generators.} 
\end{align*}
Under $q_{\Pc,\omicron}$ the element $(\twistnP{\Pc}\twistnC{\omicron})^{\TheOrder_\omicron}(\twistnP{\Pc}^{-1}\twistnC{\omicron}^{-1})^{\TheOrder_\omicron}$ maps to a non-trivial normal form in $\ZZ\ast\cyc{\TheOrder_\nu}$, i.e\ $(\twistnP{\Pc}\twistnC{\omicron})^{\TheOrder_\omicron}(\twistnP{\Pc}^{-1}\twistnC{\omicron}^{-1})^{\TheOrder_\omicron}\not\in\ker(q_{\Pc,\omicron})$. 
Moreover, for each 
$1\leq\NStrand<\TheStrand$ and \linebreak $1\leq\nu\leq\TheCone$, 
the element $(\twistn{\NStrand}\twistnC{\nu})^{\TheOrder_\nu}(\twistn{\NStrand}^{-1}\twistnC{\nu}^{-1})^{\TheOrder_\nu}$ maps trivially under $q_{\Pc,\omicron}$. Consequently, 
\[
\text{PC}(\{(\twistn{\NStrand}\twistnC{\nu})^{\TheOrder_\nu}(\twistn{\NStrand}^{-1}\twistnC{\nu}^{-1})^{\TheOrder_\mu}\mid 1\leq\NStrand<\TheStrand,1\leq\nu\leq\ThePct\})\subseteq\ker(q_{\Pc,\omicron})
\]
and thus $\kernel_\TheStrand\subseteq\ker(q_{\Pc,\omicron})$. This implies that $\kernel_{\TheStrand+\ThePct}\neq\kernel_\TheStrand$ and finally $\kernel_{\TheStrand+\ThePct}\supsetneq\kernel_\TheStrand$. 

Hence, the group $\PZ_{\TheStrand+\ThePct}^{\text{fix}(\ThePct)}(\Sigma_\freeprod)$ in comparison to $\PZ_\TheStrand(\Sigma_\freeprod(\ThePct))$ satisfies the additional relations 
\begin{equation}
\label{eq:add_rel}
(\twistP_{\TheStrand\NPct}\twistC_{\TheStrand\nu})^{\TheOrder_\nu}=(\twistC_{\TheStrand\nu}\twistP_{\TheStrand\NPct})^{\TheOrder_\nu}
\end{equation}
for $1\leq\NPct\leq\ThePct$ and $1\leq\nu\leq\TheCone$. This implies that $\omega\vert_{\PZ_\TheStrand(\Sigma_\freeprod(\ThePct))}$ is not injective for $\ThePct\geq1$. In particular, the homomorphism $\omega$ is not injective in this case. 
\end{proof}

This proves that Proposition 4.1 from \cite{Roushon2021a} is not correct. Comparing $\Z_{\TheStrand+\ThePct}^{\text{fix}(\ThePct)}(\Sigma_\freeprod)$ to $\Z_\TheStrand(\Sigma_\freeprod(\ThePct))$, \eqref{eq:add_rel} boils down to additional relations $(\twP{\NPct}\twC{\nu})^{\TheOrder_\nu}=(\twC{\nu}\twP{\NPct})^{\TheOrder_\nu}$ for $1\leq\NPct\leq\ThePct$ and $1\leq\nu\leq\TheCone$. 

\subsection*{The kernel $\kernel_\TheStrand$ for $\TheStrand\leq2$}

Under additional assumptions, 
we can give a more compact description for the normal generating set of the kernel $\kernel_\TheStrand$ if $\TheStrand\leq2$. 

For $\TheStrand=1$, no non-trivial partial conjugations in $\PZ_1(\Sigma_\freeprod(\ThePct))$ exist. Hence, the map 
\[
\piOrb\left(\Sigma_\freeprod(\ThePct)\right)\rightarrow\PZ_1(\Sigma_\freeprod(\ThePct))
\]
is an isomorphism. In particular, the kernel $\kernel_\TheStrand$ is trivial for $\TheStrand=1$. In this case, both groups are isomorphic to $\new{\freegrp{\ThePct}\ast\freeprod}$.
%

For $\TheStrand=2$, $\freeprod=\cycm$ and $\ThePct=0$, the map 
\[
\piOrb\left(D_{\cycm}(1)\right)\rightarrow\PZ_2(D_{\cycm})
\]
by Corollary \ref{cor:pure_orb_braid_semidir_prod_with_K} has kernel $\kernel_2=\langle\langle \text{PC}((\twistn{1}\twistnsC)^\TheOrder(\twistn{1}^{-1}\twistnsC^{-1})^\TheOrder\rangle\rangle$. In this case, the only partial conjugation is induced by $\twistC_{1\twistnsC}$ that maps 
\[
\begin{matrix}\twistnsC & \mapsto & \twistnsC^{-1}\twistn{1}^{-1}\twistnsC \twistn{1}\twistnsC & \text{ and } 
\\
\twistn{1} & \mapsto & \twistnsC^{-1}\twistn{1}\twistnsC, & \end{matrix}
\]
i.e.\ the conjugation by $\twistnsC^{-1}\twistn{1}^{-1}$. Hence,  
the kernel $\kernel_2$ is normally generated by $(\twistn{1}\twistnsC)^\TheOrder(\twistn{1}^{-1}\twistnsC^{-1})^\TheOrder$. In particular, the exact sequence for $\PZ_2^{\text{fix}(1)}(D_{\cycm})$ from \eqref{eq:ex_seq_PZ_fix} yields 
\[
1\rightarrow\langle\langle(\twistn{1}\twistnsC)^\TheOrder(\twistn{1}^{-1}\twistnsC^{-1})^\TheOrder\rangle\rangle\rightarrow\ZZ\ast\cyc{\TheOrder}\xrightarrow{\iotaPZ}\PZ_2^{\text{fix}(1)}(D_{\cycm})\xrightarrow{\piPZ}\underbrace{\PZ_1^{\text{fix}(1)}(D_{\cycm})}_{=1}\rightarrow1,  
\]
\new{i.e.\ $\Z_2^{\text{fix}(1)}(D_{\cycm})=\PZ_2^{\text{fix}(1)}(D_{\cycm})=\langle\twistn{1},z\mid z^\TheOrder=1,(\twistn{1}z)^\TheOrder=(z\twistn{1})^\TheOrder\rangle\neq\ZZ\ast\cycm$.} 

\bibliographystyle{plain}
\bibliography{paper_braid}

\end{document}